%% file: main_arxiv.tex
\documentclass[11pt,a4paper,reqno]{amsart}

\input{Packages_arxiv}

\input{Macros_arxiv}

\usepackage{etex}
\usepackage[a4paper, top=88pt, bottom=88pt, left=72pt, right=72pt, headsep=16pt, footskip=28pt]{geometry}

\usepackage{hyperref}
\hypersetup{
	colorlinks,
	linkcolor=blue,
	citecolor=blue,
	urlcolor=blue,
}

\newcommand*{\arXiv}[1]{\bgroup\color{blue}\href{https://arxiv.org/abs/#1}{arXiv:#1}\egroup}

\ExplSyntaxOn
\NewDocumentCommand{\doi}{m}
 {
  \group_begin:
  \tl_set:Nn \l_tmpa_tl {#1}

  \tl_replace_all:Nnn \l_tmpa_tl {\_} {\c_underscore_str}
  \regex_replace_once:nnN { \s+ \z } {} \l_tmpa_tl
  \regex_replace_once:nnN { [\.,] \z } {} \l_tmpa_tl

  \tl_replace_all:Nnn \l_tmpa_tl {<} {\c_percent_str 3C}
  \tl_replace_all:Nnn \l_tmpa_tl {>} {\c_percent_str 3E}
  \tl_replace_all:Nnn \l_tmpa_tl {(} {\c_percent_str 28}
  \tl_replace_all:Nnn \l_tmpa_tl {)} {\c_percent_str 29}
  \tl_replace_all:Nnn \l_tmpa_tl {:} {\c_percent_str 3A}
  \tl_replace_all:Nnn \l_tmpa_tl {;} {\c_percent_str 3B}

  \href{https\c_colon_str//doi.org/\tl_use:N \l_tmpa_tl}
       {\textcolor{blue}{doi:\texttt{\detokenize{#1}}}}
  \group_end:
 }
\ExplSyntaxOff

\allowdisplaybreaks

\usepackage{enumitem, moreenum}
\setlist[enumerate]{nosep}
\setlist[itemize]{nosep}
\usepackage{mleftright} \mleftright
\renewcommand{\qedsymbol}{$\blacksquare$}
\renewenvironment{proof}[1][\proofname]{\noindent{\bfseries\sffamily #1.} }{\hfill\qedsymbol\medskip}

\usepackage{caption}

\usepackage{cleveref}

\usepackage[foot]{amsaddr}

\setlist{topsep=0.3ex, itemsep=0.3ex}

\title[Verifiable Regularity Criterion for Conditional Expectation Operators]{Verifiable Regularity Criterion for
Conditional Expectation Operators and Conditional Mean Embeddings with Applications to
Nonparametric Regression,
Bayesian Inverse Problems,
and Koopman Operators}

\author{Maximiliano Hertel$^{1}$}
\address{$^1$Optimization-based Control Group, Institute of Mathematics,
Technische Universit\"at Ilmenau, Germany
({\textsc{\{maximiliano.hertel, karl.worthmann\}@tu-ilmenau.de}}).}%
\author{Ilja Klebanov$^{2}$}
\address{$^{2}$Department of Mathematics and Computer Science, Freie Universit{\"a}t Berlin, 
Germany (\textsc{klebanov@zedat.fu-berlin.de})}
\author{Manuel Schaller$^{3}$}
\address{$^3$Faculty of Mathematics, Chemnitz University of Technology, Germany
(\textsc{manuel.schaller@math.tu-chemnitz.de}).}%
\author{Karl Worthmann$^{1}$}

\thanks{M.\ Hertel gratefully acknowledges funding by Carl-Zeiss-Stiftung, project KI-MSO-O: AI-supported analysis, modeling and synthesis (design) of organoids, P2024-02-019.
I.\ Klebanov was funded by the Deutsche Forschungsgemeinschaft (DFG, German Research Foundation) under Germany’s Excellence Strategy (EXC-2046/1, project 390685689) of the Berlin Mathematics Research Center MATH$+$.
K.\ Worthmann is grateful for the support of the German Research Foundation (DFG) within the research unit ALeSCo: Active Learning for Systems and Control, project number 535860958.
The authors furthermore thank Mattes Mollenhauer for the insightful and productive discussions.}

\begin{document}

\begin{abstract}
Conditional expectation operators (CEOs) and their associated conditional mean embeddings (CMEs) play a central role across applied mathematics and machine learning, appearing in nonparametric regression, Bayesian inverse problems, and Koopman operator theory. A fundamental question is when a CEO maps a function space on 
$\cY$ into a prescribed function space on 
$\cX$, particularly a reproducing kernel Hilbert space (RKHS). We show that such mapping properties are characterized by the regularity of the Radon–Nikodym density of the conditional law, and establish a simple, verifiable sufficient condition under which the CEO is bounded and Hilbert–Schmidt. 
For RKHSs norm-equivalent to Sobolev spaces, this condition reduces to Sobolev regularity of the conditional density. 
The result yields a direct route to validate 
CME representations and 
error bounds for 
Galerkin-type and CME-based estimators. We verify the regularity condition in three settings: nonparametric regression, Bayesian inverse problems, and Koopman operator theory for stochastic dynamical systems. We show in each case that classical regularity results on the underlying probabilistic model imply the required mapping properties. The resulting framework offers a unified perspective on conditional expectation operators across probability, operator theory, kernel methods, and stochastic dynamics.
\end{abstract}

\maketitle
\pagestyle{headings}


\smallskip
\noindent \textbf{Mathematics subject classications.} 47B10, 
37A30, 
46E22, 
45P05, 
62G08, 
62F15 

\section{Introduction}

Conditional expectation is among the most fundamental operations in probability theory, and understanding its regularity as a function of the conditioning variable is a recurring theme across applied mathematics.
This regularity question can be phrased 
in terms of the \emph{conditional expectation operator} (CEO; \citealt{mollenhauer2023nonparametric})
\[
    K \colon \cF_{\cY} \to \cF_{\cX}, \qquad (Kg)(x) \coloneqq \bE[g(Y)\,|\, X=x],
\]
where \(X\) and \(Y\) are random variables taking values in spaces \(\cX\) and \(\cY\), and \(\cF_{\cX}\) and \(\cF_{\cY}\) are spaces of functions on \(\cX\) and \(\cY\), respectively.
Further, one may seek to represent $Kg$ in a reproducing kernel Hilbert space (RKHS) using a \emph{conditional mean embedding} (CME; \cite{song2009hilbert,Klebanov2020CME}).
Both, the CEO and its lifted counterpart, the CME in the RKHS feature space, specialize to well-studied objects across a range of settings:
\begin{itemize}
    \item \emph{Nonparametric regression}, where $Kg$ recovers the regression function and, more generally, encodes the conditional law of $Y$ given $X=x$ \citep{GyorfiKohlerKrzyzakWalk2002,Caponnetto2007,Tsybakov2009nonparametric};
    \item \emph{Bayesian inverse problems}, where $K$ computes posterior expectations given observed data \citep{Stuart2010IP};
    \item \emph{stochastic dynamical systems}, where $K_t g(x)=\bE[g(X_t) \, | \, X_0=x]$ defines the stochastic Koopman operator or Markov (transition) operator/semigroup \citep{Koopman1931,BrunBudiKaisKutz22,MeynTweedie2009}.
\end{itemize}
A fundamental question across all of these settings concerns the \emph{regularity} of the image \(Kg\). 
In other words, for which spaces \(\cF_{\cX}\) and \(\cF_{\cY}\), is the linear operator $K$ well-defined and, possibly, 
bounded, compact or Hilbert--Schmidt? Similarly, such mapping properties play a crucial role in the RKHS setting:
In operator-based CME theory, this representability condition is the key assumption under which the standard operator-based representation of conditional expectations and conditional mean embeddings is valid \citep{Grunewalder2012CMEregressors,Klebanov2020CME}.
However, this representability assumption is well known to be restrictive and difficult to verify, especially in continuous settings.
Closely related questions about the mapping properties of CEOs also arise in stochastic Koopman theory and in the analysis of approximation methods such as (kernel) extended dynamic mode decomposition (EDMD) \citep{WillRowl15, hertel2025koopmanstochasticdynamicserror}.

The starting point of this work is the observation that these regularity questions can be studied via the Radon--Nikodym density of the conditional law.
Assume that, for \(P_X\)-almost every \(x\in \cX\), the conditional distribution \(P_{Y\mid X=x}\) is absolutely continuous with respect to a reference measure \(\nu_{\cY}\) on $\cY$, and denote its density by
\[
\big( \eta(y) \big) (x) \coloneqq \frac{dP_{Y\mid X=x}}{d\nu_{\cY}}(y).
\]
Viewing $K$ as an integral operator with integral kernel $\eta$ yields the key observation underlying this work, which amounts to the following simple implication for any Hilbert space $\cH$ of functions on $\cX$:
\begin{equation*}
\eta \in L^2(\nu_{\cY};\cH)
\quad \Longrightarrow\quad
K\colon L^2(\nu_{\cY})\to \cH \text{ is Hilbert--Schmidt.}
\end{equation*}
This yields a conceptually transparent route to RKHS representability: instead of verifying directly that $\bE[g(Y) \, | \, X=\quark]\in \cH$ for all $g$, it suffices to establish regularity of the conditional density in the conditioning variable,
ensuring in particular the applicability of the CME operator-based representation.
Based on this observation, this paper makes four main contributions, summarized in \Cref{fig:paper_overview_intro_new}:
\begin{itemize}
    \item \textbf{Regularity criterion for Conditional Expectation Operators} (CEOs). 
    We develop an abstract regularity condition for CEOs based entirely on the regularity of the conditional density (\Cref{sec:cond_exp_koopman_operator}).    
    \item \textbf{Regularity criterion for Conditional Mean Embeddings} (CMEs).
    The operator-based CME representation is usually available only under an
    \emph{a priori} representability assumption on the CEO. We instead
    \emph{derive} this assumption from density regularity of the conditional law,
    turning it into a verifiable condition for the CME covariance-operator
    representation in the RKHS feature space (\Cref{sec:CME_and_abstract_results}).
    \item \textbf{Error bounds for data-driven approximations}.
    We establish error bounds for Galerkin-type (kernel EDMD) and CME-based estimators of the CEO
    (\Cref{sec:error_bound_implications}) with explicit learning rates depending, e.g., on fill distance or eigenvalue decay.
    \item \textbf{Verification of regularity criterion and consequences.}
    We verify the density regularity condition in three notable settings: nonparametric regression, Bayesian inverse problems, and stochastic differential equations (\Cref{sec:verification_density_regularity}).
\end{itemize}

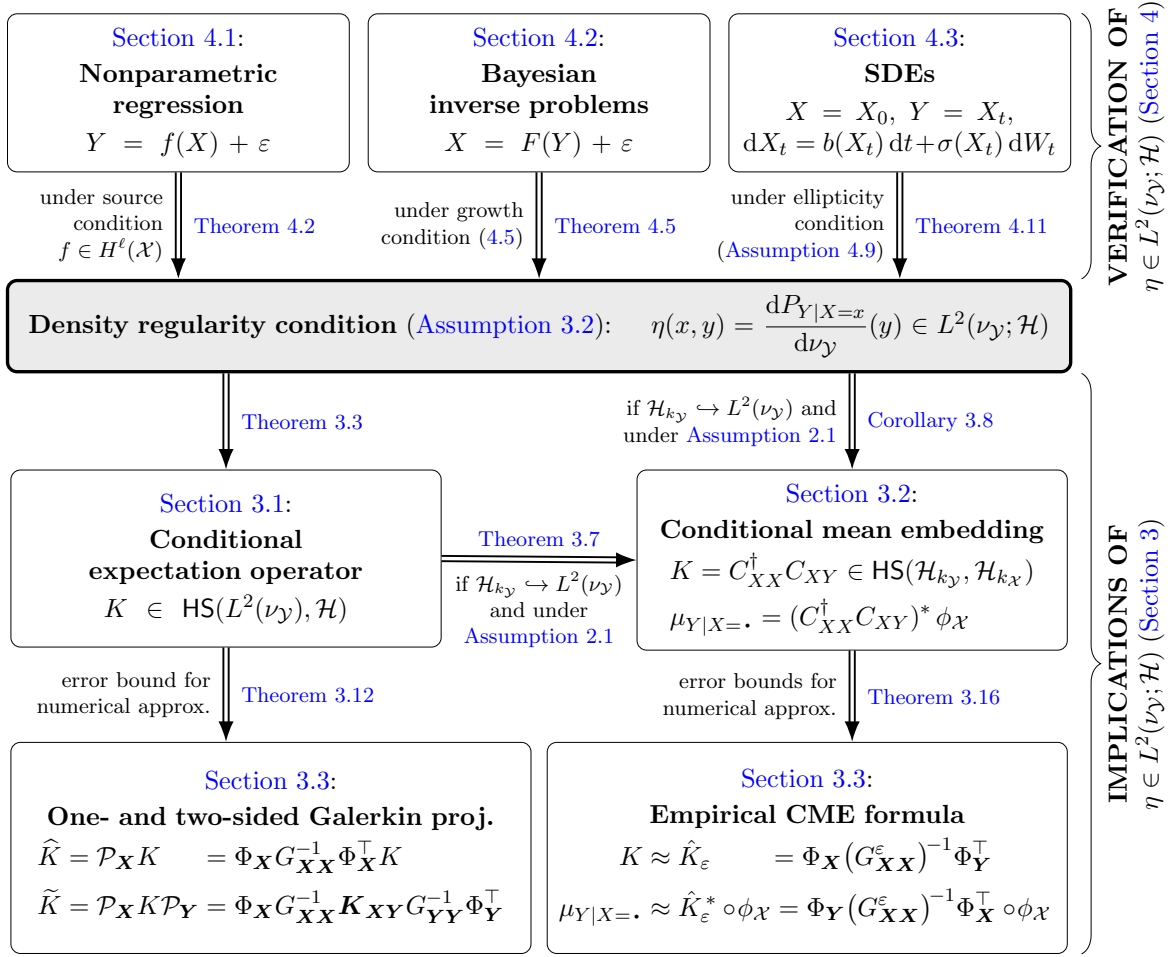
\begin{figure}[!t]
\centering
\adjustbox{scale=0.9}{
\begin{tikzpicture}[
    >=Latex,
    appbox/.style={
        draw, rounded corners, align=center, inner sep=6pt,
        text width=4.6cm, minimum height=2.3cm
    },
    mainbox/.style={
    draw, rounded corners, align=center, inner sep=6pt,
    text width=15.2cm, fill=gray!16,
    line width=1.4pt
    },
    resbox/.style={
        draw, rounded corners, align=center, inner sep=6pt,
        text width=5.9cm, minimum height=2.6cm
    },
    numbox/.style={
        draw, rounded corners, align=center, inner sep=6pt,
        text width=7.22cm, minimum height=3.1cm
    },
    arr/.style={->, thick, double, double distance=1.4pt},
    lab/.style={font=\footnotesize, inner sep=2pt},
    lableft/.style={lab, anchor=east, align=right, xshift=-4pt},
    labright/.style={lab, anchor=west, align=left, xshift=4pt}
]

\node[appbox, anchor=north] (reg) at (-5.3,-0.8)
{
\Cref{sec:regression}:
\\[0.7mm]
\textbf{Nonparametric regression}
\\[1mm]
$Y = f(X) + \varepsilon$
};

\node[appbox, anchor=north] (bip) at (0,-0.8)
{
\Cref{sec:Bayesian_IP}:
\\[0.7mm]
\textbf{Bayesian\\ inverse problems}
\\[1mm]
$X = F(Y) + \varepsilon$
};

\node[appbox, anchor=north] (sde) at (5.3,-0.8)
{
\Cref{sec:koopman_sde_class}:
\\[0.7mm]
\textbf{SDEs}
\\[1mm]
$X = X_0,\ Y = X_t$,
\\
$\rd X_t = b(X_t)\,\rd t + \sigma(X_t)\,\rd W_t$
};

\node[mainbox, anchor=north] (ass) at (0,-4.7)
{
\textbf{Density regularity condition}
(\Cref{assump:general_Radon_Nikodym_setting}):
$\quad\eta(x,y) = \dfrac{\rd P_{Y\mid X=x}}{\rd \nu_{\cY}}(y)
\in L^2(\nu_{\cY};\cH)$
};

\node[resbox, anchor=north] (ceo) at (-4.6,-7.5)
{
\Cref{sec:cond_exp_koopman_operator}:
\\[0.7mm]
\textbf{Conditional\\ expectation operator}
\\[1mm]
$K \in \HS(L^2(\nu_{\cY}) , \cH)$
};

\node[resbox, anchor=north] (cme) at (4.6,-7.5)
{
\Cref{sec:CME_and_abstract_results}:
\\[0.7mm]
\textbf{Conditional mean embedding}
\\[1mm]
$\begin{aligned}
&K
=
C_{XX}^{\dagger} C_{XY} \in \HS( \cH_{k_\cY} , \cH_{k_\cX})
\\
&\mu_{Y\mid X=\quark}
=
(C_{XX}^{\dagger} C_{XY})^{*}\,\phi_{\cX}    
\end{aligned}$
};

\node[numbox, anchor=north] (galerkin) at (-3.94,-11.5)
{
\Cref{sec:error_bound_implications}:
\\[0.7mm]
\textbf{One- and two-sided Galerkin proj.}
\\[1mm]
$\begin{aligned}
\widehat{K}
&= \cP_{\bsX} K \phantom{\cP_{\bsY}}
= \Phi_{\bsX} G_{\bsX\bsX}^{-1} \Phi_{\bsX}^{\top} K
\\[0.5mm]
\widetilde{K}
&= \cP_{\bsX} K \cP_{\bsY}
= \Phi_{\bsX} G_{\bsX\bsX}^{-1} \bsK_{\bsX\bsY}
G_{\bsY\bsY}^{-1} \Phi_{\bsY}^{\top}
\end{aligned}$
};

\node[numbox, anchor=north] (cmeformula) at (3.94,-11.5)
{
\Cref{sec:error_bound_implications}:
\\[0.7mm]
\textbf{Empirical CME formula}
\\[1mm]
$\begin{aligned}
K &\approx \hat{K}_{\varepsilon} \phantom{\circ \phi_{\cX}\ }
= \Phi_{\bsX}\bigl(G_{\bsX\bsX}^{\varepsilon}\bigr)^{\! -1}\Phi_{\bsY}^{\top}
\\[0.5mm]
\! \mu_{Y\mid X=\quark} &\approx \hat{K}_{\varepsilon}^{\,*}  \circ \! \phi_{\cX}
= \Phi_{\bsY}\bigl(G_{\bsX\bsX}^{\varepsilon}\bigr)^{\! -1}\Phi_{\bsX}^{\top}  \circ \! \phi_{\cX}
\end{aligned}$
};


\draw[arr] (reg.south) --
    node[lableft]
    {under source \\ condition \\
     $f \in H^{\ell}(\cX)$}
    node[labright]{\Cref{thm:kRR_well_specified_homoskedastic}}
    (reg.south |- ass.north);

\draw[arr] (bip.south) --
    node[lableft]
    {under growth
     \\
     condition \eqref{eq:log_partition_growth_assumption}
     }
     node[labright]{\Cref{thm:unbounded_F_weighted_sobolev}}
    (bip.south |- ass.north);

\draw[arr] (sde.south) --
    node[lableft]
    {under ellipticity\\ condition\\
     (\Cref{ass:sde_unif_elliptic_regular})}
    node[labright]{\Cref{thm:sde_density_koopman_unified}} 
    (sde.south |- ass.north);

\draw[arr] (ceo.north |- ass.south) --
    node[labright]{\Cref{thm:CEO_bounds_various_spaces}}
    (ceo.north);

\draw[arr] (cme.north |- ass.south) --
    node[labright]{\Cref{cor:regularity_conditional_expectation_operator_CME}}
    node[lableft]{if $\cH_{k_{\cY}} \hookrightarrow L^2(\nu_{\cY})$ and\\
    under \Cref{ass:kernel_rv_pair}}
    (cme.north);

\draw[arr] (ceo.east) --
    node[lab, anchor=south, yshift=3pt, align=center]
    {\Cref{thm:cme_representation}}
    node[lab, anchor=north, yshift=-3pt, align=center]
    {if $\cH_{k_{\cY}} \hookrightarrow L^2(\nu_{\cY})$
    \\
    and under
    \\
    \Cref{ass:kernel_rv_pair}
    }
    (cme.west);

\draw[arr] (ceo.south) --
    node[lableft]{error bound for\\ numerical approx.}
    node[labright]{\Cref{thm:err_bounds_proj_unified}}
    (ceo.south |- galerkin.north);

\draw[arr] (cme.south) --
    node[lableft]{error bounds for\\ numerical approx.}
    node[labright]{\Cref{prop:cme_learning_rate}}    
    (cme.south |- cmeformula.north);

\coordinate (bracex) at (7.9,0);

\draw[decorate, decoration={brace, amplitude=8pt, raise=2pt}]
    (bracex |- sde.north) -- (bracex |- ass.north)
    node[midway, right=25pt, rotate=90, anchor=center,
         font=\normalsize, align=center]
    {\textbf{VERIFICATION OF}
    \\
    $\eta \in L^2(\nu_{\cY};\cH)$
     (\Cref{sec:verification_density_regularity})};

\draw[decorate, decoration={brace, amplitude=8pt, raise=2pt}]
    (bracex |- ass.south) -- (bracex |- cmeformula.south)
    node[midway, right=25pt, rotate=90, anchor=center,
         font=\normalsize, align=center]
    {\textbf{IMPLICATIONS OF}
    \\
    $\eta \in L^2(\nu_{\cY};\cH)$
     (\Cref{sec:consequences_density_regularity})};
\end{tikzpicture}
}
\caption{ \small
Overview of the paper. The central density regularity condition
$\eta\in L^2(\nu_{\cY};\cH)$ (\Cref{assump:general_Radon_Nikodym_setting})
is verified in three application settings under suitable conditions (top,
\Cref{sec:verification_density_regularity}). It implies regularity of the
CEO (\Cref{thm:CEO_bounds_various_spaces}) and
validity of the conditional mean embedding representation
(\Cref{cor:regularity_conditional_expectation_operator_CME}), which in turn
yield error bounds for Galerkin-type and CME-based numerical approximations
(\Cref{sec:error_bound_implications}).
}
\label{fig:paper_overview_intro_new}
\end{figure}

\medskip

\noindent
\textbf{Outline.}
\Cref{sec:setup} introduces notation and background on RKHSs.
\Cref{sec:consequences_density_regularity} develops the operator-theoretic consequences of the density regularity condition: mapping properties of CEOs (\Cref{sec:cond_exp_koopman_operator}), verifiable conditions for conditional mean embeddings (\Cref{sec:CME_and_abstract_results}), and the resulting approximation error bounds (\Cref{sec:error_bound_implications}).
\Cref{sec:verification_density_regularity} verifies the density regularity condition in nonparametric regression (\Cref{sec:regression}), Bayesian inverse problems (\Cref{sec:Bayesian_IP}), and stochastic differential equations (\Cref{sec:koopman_sde_class}).
\Cref{sec:conclusions} concludes the work.

\section{Preliminaries and Notation}
\label{sec:setup}
Throughout this paper, $\cX,\cY$ denote standard Borel spaces and $(\Omega,\Sigma,\bP)$ refers to some underlying probability space.
We will work with two random variables $X \colon \Omega \to \cX$, $Y \colon \Omega \to \cY$ and denote the distributions of $X,Y$ and $(X,Y)$ by $P_X$, $P_Y$ and $P_{XY}$ respectively.
By standard results 
\citep[Thm.~8.5]{Kallenberg2021FoundationsModernProbability}, there exists a regular conditional distribution $P_{Y|X}$, unique up to $P_X$-almost everywhere equality.
Equivalently, $P_{Y|X}$ is a Markov kernel from $\cX$ to $\cY$, assigning to each $x \in \cX$ a probability measure $P_{Y|X=x}$ on $\cY$.
Throughout this paper, we fix one such version of the regular conditional distribution. Accordingly, for every measurable function $g: \cY \to \bR$ for which the expectation exists, we write 
\begin{equation}
\label{equ:generic_CEO}
K \colon \cF_{\cY} \to \cF_{\cX},
\qquad
(Kg)(x) \coloneqq
\bE[g(Y)  \, | \, X = x] \coloneqq \int_\cY g(y) \, P_{Y|X=x}(\rd y).
\end{equation}
Here, $\cF_{\cX}$ and $\cF_{\cY}$ denote spaces of functions on $\cX$ and $\cY$, respectively, and will be specified in each result separately, whenever relevant.
Technically speaking, $Kg$ denotes a measurable function which represents the corresponding conditional expectation, that is, $(Kg)(X) = \bE[g(Y) \, | \, \sigma(X)]$ almost surely, where $\sigma(X)$ denotes the $\sigma$-algebra generated by $X$; such a representative exists as a function of $X$ due to the Doob--Dynkin lemma \citep[Lemma~1.14]{Kallenberg2021FoundationsModernProbability}.
Consequently, identities involving conditional expectations are understood up to $P_X$-almost everywhere equality, while the chosen representative allows us to regard them as pointwise-defined functions of $x$.

We write $\bN \coloneqq \{1,2,\dots\}$ and $\bN_{0} \coloneqq \bN \cup \{0\}$, and for a multi-index $\alpha = (\alpha_{1},\dots,\alpha_{d}) \in \bN_{0}^{d}$ we write $|\alpha| \coloneqq \alpha_{1}+\cdots+\alpha_{d}$ and $\partial^{\alpha} = \partial_{x_{1}}^{\alpha_{1}} \cdots \partial_{x_{d}}^{\alpha_{d}}$ (also denoted $D^{\alpha}$); we write $\partial_{x}^{\alpha}$ when differentiating with respect to a variable $x$ to distinguish it from other arguments. By $\norm{\quark}_{2}$ we denote the Euclidean norm on $\bR^{d}$; norms and inner products on other (Hilbert) spaces carry an explicit subscript indicating the space. For an open set $D \subseteq \bR^{d}$, we denote by $\bslambda$ the Lebesgue measure on $D$.

For $\ell \ge 0$, we denote by $H^{\ell}( \bR^d)$ the $L^{2}$-based Sobolev space of order $\ell$, equipped with the norm 
\[
\norm{f}_{H^{\ell}(\bR^{d})}^{2}
\coloneqq
\int_{\bR^{d}} \big(1+\norm{\omega}_{2}^{2}\big)^{\ell} \, \absval{\widehat f(\omega)}^{2} \, \rd\omega
\]
where $\widehat{f}$ is the Fourier transform of $f$. If $D \subsetneq\bR^d$ is bounded with Lipschitz boundary (in the following abbreviated by Lipschitz bounded) we define the corresponding space by restriction, that is,
\[
\norm{f}_{H^{\ell}(D)}
\coloneqq
\inf\big\{ \norm{F}_{H^{\ell}(\bR^{d})} \mid F \in H^{\ell}(\bR^{d}),\ F|_{D} = f \big\}.
\]
For integer $\ell \in \bN_{0}$ these definitions agree with the usual weak-derivative Sobolev space $W^{\ell,2}(\bR^d)$, with norm $\norm{f}_{W^{\ell,2}(\bR^d)}^{2} = \sum_{|\alpha| \le \ell} \norm{\partial^{\alpha} f}_{L^{2}(\bR^d)}^{2}$ and analogously for Lipschitz bounded $D\subsetneq \bR^{d}$.

For normed spaces $E,F$, we denote by $\cL(E,F)$ the space of bounded linear operators $A \colon E \to F$, equipped with the operator norm
\[
\norm{A}_{\cL(E,F)}
\coloneqq
\sup_{\norm{u}_{E} \le 1} \norm{Au}_{F}.
\]
If $E$ and $F$ coincide as vector spaces their norms are equivalent, we write $E \simeq F$.

By $C_{b}(\cX)$ we denote the space of continuous and bounded real-valued functions on $\cX$, equipped with the supremum norm $\norm{f}_{\infty} \coloneqq \sup_{x \in \cX} |f(x)|$. For an operator $A \colon \cG \to \cH$ between Hilbert spaces $\cG,\cH$, we denote by $A^{*}$ its Hilbert space adjoint and by $A^{\dagger}$ its Moore--Penrose pseudoinverse. If $A$ is compact with non-zero singular values $(\sigma_i)_{i \in \bN}$, we denote its Hilbert--Schmidt and trace-class norms by
\[
\norm{A}_{\HS}
\coloneqq
\norm{A}_{\HS(\cG,\cH)}
\coloneqq
\bigg( \sum_{i=1}^{\infty} \absval{\sigma_i}^{2} \bigg)^{\!\! 1/2}
\in [0,\infty],
\qquad
\norm{A}_{\mathsf{Tr}(\cG,\cH)}
\coloneqq
\sum_{i=1}^{\infty} \absval{\sigma_i}
\in [0,\infty],
\]
which are finite if and only if $A$ is Hilbert--Schmidt respectively trace-class; the associated spaces are denoted by $\HS(\cG,\cH)$ and $\Tr(\cG,\cH)$, respectively. By $\mathrm{Id}$ we denote the identity operator on a given space. Finally, for quantities $a,b$ depending on a parameter (e.g.\ $n \in \bN$), we write $a \lesssim b$ if there exists a constant $C>0$, independent of the parameter, such that $a \le Cb$, and $a \asymp b$ if $a \lesssim b$ and $b \lesssim a$.

\subsection{Reproducing Kernel Hilbert Spaces (RKHS)}

We briefly recall the basics and notation for reproducing kernel Hilbert spaces used throughout the paper; standard references are \citep{Aronszajn1950RKHS, Berlinet2004RKHS, Wendland2005Scattered, Steinwart2008SVM, Saitoh2016rkhs}.
Let $k_{\cX} \colon \cX \times \cX \to \bR$ be a symmetric and positive definite kernel on a set $\cX$.
We denote by $\cH_{k_{\cX}}$ the corresponding reproducing kernel Hilbert space and by
$\phi_{\cX}(x) \coloneqq k_{\cX}(x,\quark) \in \cH_{k_{\cX}}$
the canonical feature map.
The reproducing property then reads
\begin{equation}
\label{equ:basic_reproducing_property_RKHS}
h(x)
=
\innerprod{h}{\phi_{\cX}(x)}_{\cH_{k_{\cX}}}
\qquad
\text{for all } h \in \cH_{k_{\cX}} \text{ and } x \in \cX .
\end{equation}
When an RKHS is used together with a random variable, the following compatibility assumption ensures that RKHS elements can be regarded as square-integrable random variables and that the associated feature maps and covariance operators (cf.\ \Cref{sec:CME_and_abstract_results}) are well-defined \citep{Klebanov2020CME}.

\begin{assumption}[Kernel--random variable pair]\label{ass:kernel_rv_pair}
The pair $(X,k_{\cX})$, consisting of a random variable $X \colon \Omega \to \cX$ and a kernel $k_{\cX}$ on $\cX$, satisfies the following:
\begin{enumerate}[label=(\roman*)]
\item
\label{item:basic_assumptions_kernel}
$k_{\cX}$ is measurable, symmetric and positive definite, and $\cH_{k_{\cX}}$ is separable.

\item
\label{item:basic_assumptions_feature_map_L2}
The feature map is square-integrable along $X$, that is,
\[
\bE\big[ \norm{\phi_{\cX}(X)}_{\cH_{k_{\cX}}}^{2} \big]
=
\bE\big[ k_{\cX}(X,X) \big]
< \infty .
\]

\item
\label{item:basic_assumptions_ae_separation}
The canonical embedding into $L^{2}(P_X)$ is injective: for every $h \in \cH_{k_{\cX}}$,
\[
h = 0 \quad P_X\text{-a.e.}
\qquad
\Longrightarrow
\qquad
h = 0 \quad \text{in } \cH_{k_{\cX}} .
\]
\end{enumerate}
\end{assumption}

\medskip

\Cref{ass:kernel_rv_pair}\ref{item:basic_assumptions_feature_map_L2} implies that every $h \in \cH_{k_{\cX}}$ is square-integrable with respect to $P_X$ and that the canonical embedding
\[
\iota_{k_{\cX},X}
\colon
\cH_{k_{\cX}} \hookrightarrow L^{2}(P_X),
\qquad
h \mapsto h,
\]
is bounded. Indeed, by the reproducing property and the Cauchy--Schwarz inequality,
\[
\norm{h}_{L^{2}(P_X)}^{2}
=
\bE\big[ \absval{h(X)}^{2} \big]
\le
\norm{h}_{\cH_{k_{\cX}}}^{2}
\bE\big[ k_{\cX}(X,X) \big] .
\]
\Cref{ass:kernel_rv_pair}\ref{item:basic_assumptions_ae_separation} ensures that this embedding is injective, so that elements of $\cH_{k_{\cX}}$ can be identified uniquely with their $P_X$-almost everywhere equivalence classes in $L^{2}(P_X)$.
In particular, we may view $\cH_{k_{\cX}}$ as a subspace of $L^{2}(P_X)$ and write $f \in \cH_{k_{\cX}}$ for functions $f \in L^{2}(P_X)$ whenever there exists $h \in \cH_{k_{\cX}}$ such that $f=h$ $P_X$-almost everywhere.
The injectivity condition holds, for example, if $k_{\cX}$ is continuous, $\cX$ is a topological space, and the topological support of $P_X$ equals $\cX$.

We recall the restriction property of RKHSs.
For $\cD \subseteq \cX$, let $k_{\cX}|_{\cD \times \cD}$ denote the restriction of $k_{\cX}$ to $\cD \times \cD$.
Then the corresponding RKHS
\[
\cH_{k_{\cX}}(\cD)
\coloneqq
\cH_{k_{\cX}|_{\cD \times \cD}}
\]
can be identified with the restriction space \citep[Theorem~10.47]{Wendland2005Scattered}
\begin{equation}
\label{equ:Wendland_restriction_RKHS}
\cH_{k_{\cX}}(\cD)
=
\{ h|_{\cD} \mid h \in \cH_{k_{\cX}} \},
\qquad
\norm{g}_{\cH_{k_{\cX}}(\cD)}
=
\inf
\big\{
\norm{h}_{\cH_{k_{\cX}}}
\mid
h \in \cH_{k_{\cX}},\,
h|_{\cD}=g
\big\}.
\end{equation}
We will further make use of the following standard reweighting construction,
which modifies a given kernel by a Gaussian weight and thereby enlarges the
corresponding RKHS by functions of controlled growth at infinity.
\begin{lemma}
	\label{lemma:reweighted_kernel}
	Let $\cX = \bR^{d}$, $\Gamma \in \bR^{d\times d}$ symmetric and strictly positive definite, and $k_0$ be a symmetric and positive definite kernel on $\bR^d$.
	Then, for $\tau > 0$, the reweighted kernel
    \begin{equation}
    \label{equation:Gaussian_reweighted_kernel}
    k_{\tau}(x,x')
	\coloneqq
	m_\tau(x)^{-1} \, k_0(x,x') \, m_\tau(x')^{-1},
	\qquad
	m_\tau(x) \coloneqq
	\exp\big(-\tfrac{\tau}{2}\norm{x}_{\Gamma}^{2}\big),
    \end{equation}	
	is also a symmetric and positive definite kernel on $\bR^d$ with
	$\cH_{k_\cX}
	=
	\{ f \mid m_\tau f \in \cH_{k_0} \}$
	and
	$\norm{f}_{\cH_{k_{\cX}}}
	=
	\norm{ m_\tau f }_{\cH_{k_0}}$.
\end{lemma}

\begin{proof}
	This is the standard rescaling property of RKHSs, see e.g.\ \citep[Corollary~2.5]{Saitoh2016rkhs}.
\end{proof}

\subsection{Sobolev Reproducing Kernel Hilbert Spaces}
\label{subsec:sob}
A particularly important class of kernels are those whose reproducing kernel Hilbert spaces are norm-equivalent to Sobolev spaces \citep{Wendland2005Scattered}. Two features make them central to this work.
First, Sobolev spaces provide the canonical scale of quantitative smoothness in analysis, approximation theory, and statistical learning. Second, and most importantly for our purposes, the abstract density regularity condition $\eta \in L^2(\nu_\cY;\cH)$ then becomes one of \emph{Sobolev regularity of the conditional density} in the conditioning variable---a reformulation that makes it verifiable through classical smoothness estimates for the conditional law, available in a wide range of models.
We verify this Sobolev regularity in \Cref{sec:verification_density_regularity}, for nonparametric regression, Bayesian inverse problems, and stochastic differential equations governed by uniformly elliptic diffusion.

\begin{proposition}[Mat\'ern and Wendland kernels yield Sobolev-type RKHSs]
\label{prop:matern_wendland_sobolev}
Let $\rho > 0$, $d\in\bN$ and either $\cX = \bR^d$ or $\cX \subseteq \bR^d$ be a bounded Lipschitz domain.
Further, let $K_\beta$, $\beta > 0$, denote the modified Bessel function of the second kind and $\phi_{d,j}$ denote the Wendland function of order $j\in\bN_0$ \citep[Definition~9.11]{Wendland2005Scattered}, with the additional assumption $d \geq 3$ if $j = 0$.
Then the RKHSs corresponding to the Mat\'ern and Wendland kernels
\begin{align}
\label{equ:def_Matern_kernel}
k_{\beta}^{\mathrm{Mat}}(x,x')
&=
\frac{2^{1-\beta}}{\Gamma(\beta)}
\left(\sqrt{2\beta}\,\|x-x'\|_{2}\right)^\beta
K_\beta \left(\sqrt{2\beta}\,\|x-x'\|_{2}\right),
\\
\label{equ:def_Wendland_kernel}
k_{d,j}^{\mathrm{Wen}}(x,x')
&=
\phi_{d,j} \big(\rho^{-1} \|x-x'\|_2\big),
\end{align}
satisfy
\[
\cH_{k_{\beta}^{\mathrm{Mat}}}(\cX)
\simeq
H^{\beta+\frac d2}(\cX),
\qquad
\cH_{k_{d,j}^{\mathrm{Wen}}}(\cX)
\simeq
H^{j+\frac{d+1}{2}}(\cX).
\]
Moreover, the Mat\'ern kernel $k_{\beta}^\mathrm{Mat}$ is $\lceil 2 \beta - 1 \rceil$-times continuously differentiable on $\bR^d$. 
\end{proposition}

\begin{proof}
    First, let $\cX = \bR^d$. The characterization of the RKHS generated by a Mat\'ern kernel $k_\beta^{\mathrm{Mat}}$, $\beta >0$, can be found, for example, in \citet{PorcBeviSchabOate2024}. Moreover, \citet[][Appendix B.1]{DaCosta2026} show that the Matérn kernel is exactly $\lceil 2 \beta - 1 \rceil$-times differentiable for $\beta \notin \bN$ and $(2 \beta - 1)$-times differentiable for $\beta \in \bN$, such that in any case $k_\beta^\mathrm{Mat} \in C^{\lceil 2\beta - 1 \rceil}(\bR^d \times \bR^d)$. 
    
    For $j \in \bN_0$, with the additional assumption $d \geq 3$ if $j = 0$, \citet[Theorem~10.35]{Wendland2005Scattered} yields the assertion on Wendland kernels $k_{d,j}^\mathrm{Wen}$.

    For a bounded Lipschitz domain $\cX \subsetneq \bR^d$, both characterizations follow from the restriction result by \citet[Theorem~10.47 and Corollary~10.48]{Wendland2005Scattered}.
\end{proof}

\begin{proposition}[A sufficient criterion for $\cH_k(\cX) \hookrightarrow L^2(\cX)$]
\label{prop:rkhs_in_L2_via_spectral_density}
Let $k(x,x')=\psi(x-x')$ be a translation-invariant kernel on $\bR^d$, where
$\psi \in L^{1}(\bR^{d})$ is continuous and positive definite such that $\widehat\psi\in L^\infty(\bR^d)$.
Then the corresponding RKHS $\cH_k$ is continuously embedded into $L^2(\bR^d)$.
Moreover, for every Lebesgue-measurable set $\cX\subseteq\bR^d$, the corresponding (restricted) RKHS $\cH_k(\cX)$ is continuously embedded into $L^2(\cX)$.
More precisely, for all $f\in\cH_k(\cX)$,
\[
\|f\|_{L^2(\cX)}^2
\le
(2\pi)^{d/2}\|\widehat\psi\|_{L^\infty(\bR^d)}\,\|f\|_{\cH_k(\cX)}^2 .
\]
\end{proposition}

\begin{proof}
The proof is provided in \Cref{sec:technical_details_proofs}.
\end{proof}

By Bochner's theorem \citep[Theorem~6.6]{Wendland2005Scattered}, $\widehat{\psi}$ is the Lebesgue density of the spectral measure associated with the kernel whenever it exists, and the assumption $\widehat{\psi} \in L^\infty(\bR^d)$ follows from $\psi \in L^1(\bR^d)$ by the Riemann--Lebesgue lemma. 

\begin{remark}
\label{rem:rkhs_in_L2_via_spectral_density_examples}
\Cref{prop:rkhs_in_L2_via_spectral_density} 
applies to many well-established kernels such as Mat\'ern and Wendland kernels,
cf.\ \Cref{prop:matern_wendland_sobolev}, Gaussian and Laplace kernels, as well as
Cauchy kernels with $\beta > d/2$, whose profiles are given by
\[
\psi^{\mathrm{Gauss}}(x)
=
\exp\bigg(-\frac{\|x\|_2^2}{2\sigma^2}\bigg),
\quad\ 
\psi^{\mathrm{Lap}}(x)
=
\exp\bigg(-\frac{\|x\|_2}{\sigma}\bigg),
\quad\ 
\psi^{\mathrm{Cauchy}}(x)
=
\bigg(1+\frac{\|x\|_2^2}{\sigma^2}\bigg)^{\!\! -\beta}.
\]
\end{remark}

\section{\texorpdfstring{Implications of the Density Regularity Condition $\eta \in L^2(\nu_{\cY};\cH)$}{Implications of the Density Regularity Condition}}
\label{sec:consequences_density_regularity}

In this section, we develop the consequences of the density
regularity condition introduced in
\Cref{assump:general_Radon_Nikodym_setting}. The key observation is that the
CEO \eqref{equ:generic_CEO} admits an integral representation whose
integral kernel is precisely the Radon--Nikodym density of the conditional
law, so that regularity of the density translates directly into mapping properties of
$K$. \Cref{sec:cond_exp_koopman_operator} establishes these mapping
properties, \Cref{sec:CME_and_abstract_results} derives from them the
covariance-operator representation of conditional mean embeddings, and
\Cref{sec:error_bound_implications} provides error bounds for the resulting
numerical approximations.

\subsection{Verifiable Regularity Criterion 
for Conditional Expectation 
Operators (CEOs)}
\label{sec:cond_exp_koopman_operator}

For random variables \(X\) and \(Y\) taking values in spaces \(\cX\) and \(\cY\), we study \emph{conditional expectation operators} (CEOs) \(K \colon \cF_{\cY} \to \cF_{\cX}\) of the form \eqref{equ:generic_CEO}.
In the context of a stochastic dynamical system $(X_t)_{t\geq 0}$, taking $X=X_0$ and $Y=X_t$ recovers the associated Markov operator, commonly referred to as the (stochastic) \emph{Koopman operator}.
Our first aim is to clarify the mapping properties of \(K\) that are
crucial in view of data-driven and finite-rank approximations: for which function spaces $\cF_{\cX}$, $\cF_{\cY}$ is \(K\) a well-defined bounded linear operator, and under which assumptions is it Hilbert--Schmidt, hence compact?
Compactness is the classical prerequisite for the convergence of Galerkin-type approximation schemes for integral operators \citep[Chapter~13]{Kress2014IntegralEquations}, and underlies the growing body of work on data-driven Koopman operator approximation via (kernel) extended dynamic mode decomposition (EDMD) \citep{WillRowl15,KlusKoltaiSchutte2016:num_approx_koopman,KohnPhil24,hertel2025koopmanstochasticdynamicserror}.
For this purpose, we rewrite $K$ as a kernel integral operator
\begin{equation}
\label{equ:rewrite_conditional_expectation_using_Radon_Nikodym}
K g
=
T_{\eta} g
\coloneqq
\int_{\cY} g(y)\, \eta(\quark,y)\, \nu_{\cY}(\rd y),
\quad
\text{where}
\quad
\eta(x,y)\coloneqq \frac{\rd P_{Y\mid X=x}}{\rd \nu_{\cY}}(y)
\end{equation}
denotes the Radon--Nikodym derivative of the conditional distribution $P_{Y\mid X=x}$ with respect to an arbitrary measure $\nu_{\cY}$ on $\cY$ (under the assumption of absolute continuity $P_{Y\mid X=x} \ll \nu_{\cY}$ for $P_X$-almost every $x\in\cX$).
As usual when working with Bochner spaces, we will often view $\eta$ as a map $\eta \colon \cY \to \{ \text{functions over } \cX \}$ and write 
\[
\eta(y) \coloneqq \eta(\quark,y) \colon \cX \to \bR.
\]
We begin with a general result for integral operators with arbitrary integral kernel $\eta$, and later discuss the consequences when $\eta$ is the Radon--Nikodym derivative introduced above.
While the arguments are fairly standard, we give a self-contained proof for the sake of completeness.

\begin{proposition}
\label{prop:CME_regularization_mu}
Let $(\cY,\Sigma_{\cY},\nu_{\cY})$ be a measure space, let $\cH$ be a
separable Hilbert space and let $\eta\in L^2(\nu_{\cY};\cH)$. Then the integral
operator
\[
T_{\eta} \colon L^2(\nu_{\cY})\to \cH,
\qquad
T_{\eta} g\coloneqq \int_{\cY} g(y)\, \eta(y)\, \nu_{\cY}(\rd y),
\]
is well-defined and Hilbert--Schmidt. In particular, it is bounded, compact,
and satisfies
\[
\norm{T_{\eta}}_{\cL(L^2(\nu_{\cY}),\cH)}
\leq
\norm{T_{\eta}}_{\HS}
=
\norm{\eta}_{L^2(\nu_{\cY};\cH)}.
\]
Moreover, its adjoint $T_{\eta}^* \colon \cH\to L^2(\nu_{\cY})$ is given by
$(T_{\eta}^* h)(y) = \innerprod{\eta(y)}{h}_{\cH}$ for all $h\in\cH$.
\end{proposition}

\begin{proof}
For $g\in L^2(\nu_{\cY})$, the map $y\mapsto g(y)\eta(y)$ is Bochner integrable,
since
\[
\int_{\cY}|g(y)|\,\norm{\eta(y)}_{\cH}\,\nu_{\cY}(\rd y)
\leq
\norm{g}_{L^2(\nu_{\cY})}\norm{\eta}_{L^2(\nu_{\cY};\cH)}.
\]
Hence $T_{\eta}$ is well-defined and linear.
Further, the operator
$U_\eta\colon \cH\to L^2(\nu_{\cY})$ given by
$(U_\eta h)(y) = \innerprod{\eta(y)}{h}_{\cH}$ satisfies
$U_\eta h \in L^2(\nu_{\cY})$, because
\[
\norm{U_\eta h}_{L^2(\nu_{\cY})}^2
=
\int_{\cY}
\left|\innerprod{\eta(y)}{h}_{\cH}\right|^2
\,\nu_{\cY}(\rd y)
\leq
\norm{h}_{\cH}^2\norm{\eta}_{L^2(\nu_{\cY};\cH)}^2.
\]
Moreover, $T_{\eta}^* = U_\eta$ since, for every $g\in L^2(\nu_{\cY})$ and $h\in\cH$, 
\[
\innerprod{T_{\eta} g}{h}_{\cH}
=
\Innerprod{\int_{\cY} g(y)\eta(y)\,\nu_{\cY}(\rd y)}{h}_{\cH}
=
\int_{\cY} g(y)\innerprod{\eta(y)}{h}_{\cH}\,\nu_{\cY}(\rd y)
=
\innerprod{g}{U_\eta h}_{L^2(\nu_{\cY})}.
\]
Let $(h_m)_{m\in\bN}$ be an orthonormal basis of $\cH$.
By invariance of the Hilbert--Schmidt norm under taking adjoints, Parseval's identity and Tonelli's theorem,
\begin{align*}
\norm{T_{\eta}}_{\HS}^2
&=
\norm{T_{\eta}^*}_{\HS}^2
=
\sum_{m=1}^\infty
\norm{U_\eta h_m}_{L^2(\nu_{\cY})}^2
=
\sum_{m=1}^\infty
\int_{\cY}
\left|
\innerprod{\eta(y)}{h_m}_{\cH}
\right|^2
\,\nu_{\cY}(\rd y)
=
\int_{\cY}
\norm{\eta(y)}_{\cH}^2
\,\nu_{\cY}(\rd y)
\\
&=
\norm{\eta}_{L^2(\nu_{\cY};\cH)}^2\, ,
\end{align*}
proving the claim.
\end{proof}

We now specialize \Cref{prop:CME_regularization_mu} to the CEO, taking $\eta$ to be the Radon--Nikodym derivative from \eqref{equ:rewrite_conditional_expectation_using_Radon_Nikodym}. The condition $\eta \in L^2(\nu_{\cY};\cH)$ is then a regularity condition on the transition density in the conditioning variable; for $\cH = H^\ell(\cX)$, it amounts to Sobolev regularity of the sections $x \mapsto \eta(x,y)$, an assumption satisfied by many Markov processes with sufficiently regular transition densities, such as uniformly elliptic diffusions (revisited in \Cref{sec:koopman_sde_class}).
The theorem below shows how this $\cH$-regularity transfers to classical target spaces by composition with embeddings: for sufficiently regular $\cX \subseteq \bR^{d}$ equipped with $\nu_{\cX} = \bslambda$, Sobolev embedding theory determines whether $K$ remains Hilbert--Schmidt or trace class when regarded as an operator into $L^2(\nu_{\cX})$, and whether it becomes compact when regarded as an operator into $C_b(\cX)$.
For $\ell > d/2$, the same Sobolev spaces can also be realized as RKHSs through Mat\'ern or Wendland kernels, cf.\ \Cref{prop:matern_wendland_sobolev}, thereby connecting the present regularity criterion with RKHS interpolation theory to obtain explicit error rates of Galerkin-type approximations in~\Cref{subsubsec:galerkin_proj_approx}. 
The result follows from \Cref{prop:CME_regularization_mu} and the standard stability properties of compact, Hilbert--Schmidt and trace class operators under composition \citep[see, e.g.,][]{Reed1980Methods, Conway2007}.

\begin{assumption}[Density regularity condition]
\label{assump:general_Radon_Nikodym_setting}
$(\cX,\Sigma_{\cX})$ and $(\cY,\Sigma_{\cY})$ are standard Borel spaces,
$X\colon\Omega\to\cX$ and $Y\colon\Omega\to\cY$ are random variables.
We fix a regular conditional distribution $P_{Y\mid X= \quark}$ of $Y$ given $X$.
Further, $\nu_{\cY}$ is a $\sigma$-finite measure on $(\cY,\Sigma_{\cY})$ such that $P_{Y\mid X=x}$ is absolutely continuous with respect to $\nu_{\cY}$ for $P_X$-almost every $x\in\cX$, with jointly measurable Radon--Nikodym derivative $\eta(x,y)\coloneqq \frac{\rd P_{Y\mid X=x}}{\rd \nu_{\cY}}(y)$.
Further, $\cH$ is a separable Hilbert space of real-valued functions on $\cX$ such that $\eta \in L^2(\nu_{\cY};\cH)$, where we identify $\eta(y)$ with the section $\eta(\quark,y)$ for $y \in \cY$, that is, we view $\eta$ as Bochner-measurable map $\cY \to \cH, y \mapsto \eta(y) = \eta(\quark, y)$.
\end{assumption}

\begin{theorem}
\label{thm:CEO_bounds_various_spaces}
Let \Cref{assump:general_Radon_Nikodym_setting} hold and, whenever $\cH = H^{\ell}(\cX)$ with $\ell \geq 0$, let either $\cX = \bR^{d}$ or $\cX \subsetneq \bR^{d}$ Lipschitz bounded.
Then:
\begin{enumerate}[label = (\alph*)]
\item
\label{item:CEO_H_HS}
$K \in \HS( L^2(\nu_{\cY}) , \cH )$ with $\norm{K}_{\cL(L^2(\nu_{\cY}),\cH)}
\leq
\norm{K}_{\HS(L^2(\nu_{\cY}),\cH)}
=
\norm{\eta}_{L^2(\nu_{\cY};\cH)}$.
\item
\label{item:CEO_L2_HS}
If $\cH$ is continuously embedded in $L^2(\nu_\cX)$ for some measure $\nu_{\cX}$ on $\cX$ with embedding operator $\cE \colon \cH \hookrightarrow L^2(\nu_\cX)$, then
$K \in \HS( L^2(\nu_{\cY}) , L^2(\nu_\cX) )$ with
\[
\|K\|_{\HS( L^2(\nu_{\cY}) , L^2(\nu_\cX) )} \le \|\cE\|_{\cL(\cH,L^2(\nu_\cX))} \, \|\eta\|_{L^2(\nu_\cY;\cH)}\, .
\]
In particular, this is the case for $\cH = H^{\ell}(\cX)$ with $\ell \geq 0$ and $\nu_\cX = \bslambda$.
\item
\label{item:CEO_L2_Tr}
If, in addition, $\cE$ is Hilbert--Schmidt, then
$K \in \Tr( L^2(\nu_{\cY}) , L^2(\nu_\cX) )$ with
\[
\|K\|_{\Tr( L^2(\nu_{\cY}) , L^2(\nu_\cX) )}
\le \|\cE\|_{\HS(\cH,L^2(\nu_\cX))} \, \|\eta\|_{L^2(\nu_\cY;\cH)}\, .
\]
In particular, this is the case for Lipschitz bounded $\cX \subsetneq \bR^d$ and $\cH = H^{\ell}(\cX)$ with $\ell > d/2$ and $\nu_\cX = \bslambda$. 
\item
\label{item:CEO_Cb_compact}
If $\cH = H^{\ell}(\cX)$ with $\ell > d/2$, then $K \in \cL( L^2(\nu_{\cY}) ; C_{b}(\cX) )$ is compact.
\end{enumerate}
\end{theorem}

\begin{proof}
\ref{item:CEO_H_HS} follows from \Cref{prop:CME_regularization_mu}, since, for $g\in L^2(\nu_{\cY})$, the Radon--Nikodym identity for conditional expectations gives $Kg = T_{\eta} g$ for $P_X$-almost every $x\in\cX$.

\ref{item:CEO_L2_HS} holds because
$K \in \HS( L^2(\nu_{\cY}) , L^2(\nu_\cX) )$ is Hilbert--Schmidt as the composition of the Hilbert--Schmidt operator $K \in \HS( L^2(\nu_{\cY}) , \cH )$ and the bounded operator $\cE \in \cL(\cH , L^2(\nu_\cX) )$,
with the norm bound following from the corresponding submultiplicativity property.
This holds in particular for $\cH = H^\ell(\cX)$ with $\ell \geq 0$ since in this case $\cE$ is bounded by definition.

\ref{item:CEO_L2_Tr} holds since
$K \in \Tr( L^2(\nu_{\cY}) , L^2(\nu_\cX) )$ is trace class as the composition of the Hilbert--Schmidt operators $K \in \HS( L^2(\nu_{\cY}) , \cH )$ and $\cE \in \HS(\cH , L^2(\nu_\cX) )$,
again, with the corresponding submultiplicativity argument for the norms.
This holds in particular for Lipschitz bounded $\cX$ and $\cH = H^\ell(\cX)$ with $\ell > d/2$, since in this case $\cE$ is Hilbert-Schmidt \citep{Maurin1961}. 

\ref{item:CEO_Cb_compact} holds because
$K \in \cL( L^2(\nu_{\cY}) , C_{b}(\cX) )$ is compact as the composition of the Hilbert--Schmidt operator $K \in \HS( L^2(\nu_{\cY}) , \cH )$ and the bounded embedding $H^\ell(\cX) \hookrightarrow C_b(\cX)$ (by the Sobolev embedding theorem; \citealt[Theorem~4.12]{Adams2003}).
\end{proof}

\begin{remark}\label{rem:CEO_on_Lp}
    The integral operator from \Cref{prop:CME_regularization_mu} remains well-defined and bounded as a map $T_{\eta} \colon L^p(\nu_\cY)\to \cH$ whenever $\eta\in L^q(\nu_{\cY} ;\cH)$, for conjugate exponents $1 \leq p,q \leq \infty$ with $p^{-1} + q^{-1} = 1$: Hölder's inequality yields
    \[
    \int_{\cY}|g(y)|\,\norm{\eta(y)}_{\cH}\,\nu_{\cY}(\rd y)
    \leq
    \norm{g}_{L^p(\nu_\cY)}\norm{\eta}_{L^q(\nu_{\cY};\cH)}.
    \]
    In this sense, \Cref{prop:CME_regularization_mu} corresponds to the Hilbert case $p = q = 2$. An interesting special case is $p = \infty$ (hence $q=1$): choosing $\cH \simeq H^\ell(\cX)$ with $\ell > d/2$ (under the domain assumptions of \Cref{thm:CEO_bounds_various_spaces}), the Sobolev embedding $H^\ell(\cX) \hookrightarrow C_b(\cX)$ yields boundedness of $T_\eta \colon C_b(\cY) \to C_b(\cX)$.
    More generally, if differentiation under the integral sign is justified, the same Hölder argument applies pointwise to spatial derivatives of the kernel: for every multi-index $\alpha$ and $x \in \cX$,
    \begin{align*}
        |\partial_x^\alpha (Kg)(x)| \le \|g\|_{L^p(\nu_\cY)} \, \| (\partial_x^\alpha \eta)(x,\quark)\|_{L^q(\nu_{\cY})}.
    \end{align*}
    In particular, if $\sup_{x \in \cX} \|(\partial_x^\alpha \eta)(x,\quark)\|_{L^q(\nu_{\cY})} < \infty$ for every $|\alpha| \le \ell$, then the CEO maps $L^p$-observables into functions with uniformly bounded derivatives up to order $\ell$, i.e., $Kg \in W^{\ell,\infty}(\cX)$ for every $g \in L^p(\nu_\cY)$.
\end{remark}

\begin{remark}[Spectral structure and Koopman approximation]
\label{rem:spectral_koopman_approximation}
The Hilbert--Schmidt conclusions above imply compactness and hence provide a basic spectral approximation principle. In particular, if $K \in \HS( L^2(\nu_{\cY}) , L^2(\nu_\cX) )$, then it admits a singular value expansion
\begin{align*}
K g
&=
\sum_{n\geq 1}\sigma_n
\langle g,u_n\rangle_{L^2(\nu_{\cY})}\,v_n,
\qquad
\sum_{n\geq 1}\sigma_n^2<\infty.
\end{align*}
with orthonormal systems $(u_n)_{n \in \bN}\subset L^2(\nu_{\cY})$ and $(v_n)_{n \in \bN}\subset L^2(\nu_{\cX})$.
Truncating this expansion yields finite-rank approximations of $K$, which is the functional-analytic mechanism behind many spectral and Galerkin-type Koopman approximation methods.
If, in addition, $\cX=\cY$ and the Markov process is reversible with respect to the invariant measure $\nu = \nu_{\cX} = \nu_{\cY}$, then the corresponding Koopman operator is self-adjoint on $L^2(\nu)$, and the singular value expansion reduces to an eigenfunction expansion \citep{KlusNuskKolt2018:data_driven_model_red}, which provides a classical spectral setting for transfer and Koopman operators.
\end{remark}


\subsection{Verifiable Regularity Criterion for Conditional Mean Embeddings (CMEs)}
\label{sec:CME_and_abstract_results}

In this section, we specialize the preceding regularity criterion to the
operator-theoretic representation of conditional mean embeddings (CMEs), a
central tool in the kernel methods literature for representing conditional
distributions and a versatile tool across machine learning, including nonparametric
approaches to regression, probabilistic inference, and reinforcement learning
\citep{song2009hilbert,Grunewalder2012CMEregressors,Muandet2017KernelMeanEmbedding}.

For kernels $k_{\cX},k_{\cY}$ on $\cX,\cY$ with RKHSs $\cH_{k_{\cX}},\cH_{k_{\cY}}$
and feature maps $\phi_{\cX},\phi_{\cY}$ (cf.\ \Cref{sec:setup}), the object to be represented is the \emph{conditional kernel mean embedding} (CME)\footnote{Strictly speaking, we define these objects on $\cX_{Y} \coloneqq \{ x\in \cX \mid \bE[\norm{\phi_{\cY}(Y)}_{\cH_{k_{\cY}}}^{2}  \, | \, X=x] < \infty \}$, which has full $P_X$ measure by \Cref{ass:kernel_rv_pair}\ref{item:basic_assumptions_feature_map_L2}, and set them to zero on $\cX \setminus \cX_{Y}$.}
\begin{equation}
    \label{equ:MeanAndCovarianceConditional}
    \mu_{Y\mid X=x}
    =
    \bE[\phi_{\cY}(Y) \, | \, X=x],
    \qquad
    x\in \cX .
\end{equation}
The operator-theoretic identity underlying this representation---expressing $\mu_{Y\mid X=x}$, and more generally the CEO $K$, in terms of covariance and cross-covariance operators of the joint law of $(X,Y)$---can be traced back to \citet{fukumizu2004dimensionality}, \citet{song2009hilbert}, \citet{Klebanov2020CME}, \citet{Park2020KCME} and has since been studied extensively, also for other CEOs such as stochastic Koopman operators \citep{Klus2020Eigendecompositions}.
These works typically take the representability condition, that the CEO maps
between the prescribed RKHSs, i.e.\ \eqref{eq:cme_assumption} below, as a standing
\emph{assumption}, either imposed directly or encoded through source and embedding conditions \citep{Grunewalder2012CMEregressors,Klebanov2021LCM,Fukumizu2013KernelBayesRule,Li2024Optimal}.
This assumption is well known to be restrictive and hard to verify directly.
The main point of this section is that the abstract regularity criterion from
\Cref{thm:CEO_bounds_various_spaces} instead \emph{derives} it from regularity of the conditional density. In particular, when $\cH_{k_{\cX}}$ is norm-equivalent to a Sobolev space, representability reduces to Sobolev regularity of the transition density in the conditioning variable.

Throughout this subsection, we use the following assumption and notation:

\begin{assumption}[RKHS pair setting]
\label{assump:two_RKHS_setting}
$(\cX,\Sigma_{\cX})$ and $(\cY,\Sigma_{\cY})$ are standard Borel spaces,
$X\colon\Omega\to\cX$ and $Y\colon\Omega\to\cY$ are random variables.
We fix a regular conditional distribution $P_{Y\mid X= \quark}$ of $Y$ given $X$.
Further, $k_{\cX} \colon \cX \times \cX \to \bR$ and
$k_{\cY} \colon \cY \times \cY \to \bR$
are kernels such that both pairs $(X,k_{\cX})$ and $(Y,k_{\cY})$ satisfy \Cref{ass:kernel_rv_pair}.
\end{assumption}

For Hilbert-space valued random variables
$U\in \cL^2(\Omega,\Sigma,\bP;\cG)$ and
$V\in \cL^2(\Omega,\Sigma,\bP;\cH)$,
we denote the uncentred covariance operator by
\[
\Cov[U,V]
\coloneqq
\bE[U\otimes V]
\in \cL(\cH,\cG),
\]
where the expectation is understood as a Bochner integral \citep[Section~II.2]{diestel1977}.
Here, for $g\in\cG$ and $h\in\cH$, the rank-one operator
$g\otimes h\colon \cH\to\cG$ is defined by
$
(g\otimes h)h'
\coloneqq
\innerprod{h}{h'}_{\cH}\,g
$ for $h'\in\cH$.
Under \Cref{assump:two_RKHS_setting}, we denote the \emph{kernel mean embeddings} and the \emph{uncentred kernel covariance and cross-covariance operators} by
\begin{align*}
    \mu_{Y}
    & \coloneqq
    \bE[\phi_{\cY}(Y)] \in \cH_{k_{\cY}},&
    C_{YY}
    & \coloneqq
    \Cov[\phi_{\cY}(Y),\phi_{\cY}(Y)],&
    C_{YX}
    & \coloneqq
    \Cov[\phi_{\cY}(Y),\phi_{\cX}(X)],
    \\
    \mu_{X}
    & \coloneqq
    \bE[\phi_{\cX}(X)] \in \cH_{k_{\cX}},&
    C_{XY}
    & \coloneqq
    \Cov[\phi_{\cX}(X),\phi_{\cY}(Y)],&
    C_{XX}
    & \coloneqq
    \Cov[\phi_{\cX}(X), \phi_{\cX}(X)],
\end{align*}
as well as the conditional mean embedding (CME) $\mu_{Y\mid X=x}$ by~\eqref{equ:MeanAndCovarianceConditional}.

\medskip

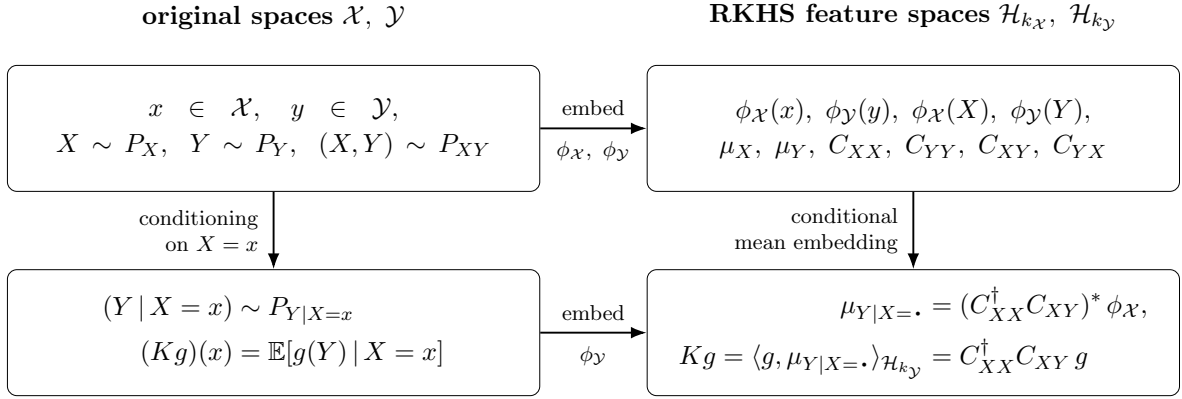
\begin{figure}[!t]
\centering
\adjustbox{scale=0.9}{
\begin{tikzpicture}[
    >=Latex,
    colbox/.style={
        draw, rounded corners, align=center, inner sep=6pt,
        text width=7.4cm, minimum height=1.85cm
    },
    arr/.style={->, thick},
    lab/.style={font=\footnotesize, inner sep=2pt},
    lableft/.style={lab, anchor=east, align=right, xshift=-4pt},
    labright/.style={lab, anchor=west, align=left, xshift=4pt}
]

\node[align=center] at (-4.7,1.1)
    {\textbf{original spaces} $\cX,\ \cY$};
\node[align=center] at (4.7,1.1)
    {\textbf{RKHS feature spaces} $\cH_{k_\cX},\ \cH_{k_\cY}$};

\node[colbox, anchor=north] (orig_top) at (-4.7,0.4)
{
$x\in\cX,\quad y\in\cY$,
\\[0.4ex]
$X\sim P_X,\ \ Y\sim P_Y,\ \ (X,Y)\sim P_{XY}$
};

\node[colbox, anchor=north] (rkhs_top) at (4.7,0.4)
{
$\phi_{\cX}(x),\ \phi_{\cY}(y),\ \phi_{\cX}(X),\ \phi_{\cY}(Y)$,
\\[0.4ex]
$\mu_X,\ \mu_Y,\ C_{XX},\ C_{YY},\ C_{XY},\ C_{YX}$
};

\node[colbox, anchor=north] (orig_cond) at (-4.7,-2.6)
{
$\begin{aligned}
(Y \, |\, X=x) &\sim P_{Y\mid X=x}
\\[0.4ex]
(Kg)(x) &= \bE[g(Y)\, |\, X=x]
\end{aligned}$
};

\node[colbox, anchor=north] (rkhs_cond) at (4.7,-2.6)
{
$\begin{aligned}
\mu_{Y\mid X = \quark} &= (C_{XX}^{\dagger} C_{XY})^{*}\,\phi_{\cX},
\\[0.4ex]
Kg = \innerprod{g}{\mu_{Y\mid X=\quark}}_{\cH_{k_\cY}}
&= C_{XX}^{\dagger} C_{XY}\, g
\end{aligned}$
};


\draw[arr] (orig_top.east) --
    node[lab, anchor=south, yshift=3pt]{embed}
    node[lab, anchor=north, yshift=-3pt]{$\phi_{\cX},\ \phi_{\cY}$}
    (rkhs_top.west);

\draw[arr] (orig_cond.east) --
    node[lab, anchor=south, yshift=3pt]{embed}
    node[lab, anchor=north, yshift=-3pt]{$\phi_{\cY}$}
    (rkhs_cond.west);

\draw[arr] (orig_top.south) --
    node[lableft]{conditioning\\ on $X=x$}
    (orig_top.south |- orig_cond.north);

\draw[arr] (rkhs_top.south) --
    node[lableft]{conditional\\ mean embedding}
    (rkhs_top.south |- rkhs_cond.north);

\end{tikzpicture}
}
\caption{
Conditioning in the original spaces (left)
and in the RKHS feature spaces (right), linked by the canonical feature maps
$\phi_{\cX},\phi_{\cY}$. After embedding, conditioning becomes a
\emph{linear} transformation of the feature vector $\phi_{\cX}(x)$.
}
\label{fig:combined_cme_applications}
\end{figure}

Note that $\mu_{Y\mid X=x}$ is well-defined for $P_{X}$-almost every $x \in \cX$.
The reproducing property \eqref{equ:basic_reproducing_property_RKHS} translates to kernel mean embeddings:
\begin{equation}
\label{equ:KME_reproducing_property}
\innerprod{g}{\mu_{Y}}_{\cH_{k_{\cY}}}
=
\innerprod{g}{\bE[\phi_{\cY}(Y)]}_{\cH_{k_{\cY}}}
=
\bE \big[\innerprod{g}{\phi_{\cY}(Y)}_{\cH_{k_{\cY}}} \big]
=
\bE[g(Y)],
\qquad
g \in \cH_{k_{\cY}},
\end{equation}
and similarly
$\innerprod{g}{\mu_{Y|X=x}}_{\cH_{k_{\cY}}}
=
\bE[g(Y)\, | \, X=x]$ $P_X$-almost everywhere.

\medskip

We briefly recall the basic CME representation result following \citet{Klebanov2020CME}.
\Cref{fig:combined_cme_applications} illustrates the two viewpoints on conditioning connected by this result: conditioning of $X,Y$ in the original spaces $\cX,\cY$ (left column), and its counterpart after embedding into the RKHS feature spaces $\cH_{k_\cX},\cH_{k_\cY}$, where it is captured by the covariance-operator representation (right column). The appeal of the latter description is that conditioning becomes a \emph{linear} transformation of the embedded objects, which can be estimated from samples of $P_{XY}$ by linear-algebraic operations alone.
The following theorem is a simplified version of \citet[Theorem~5.3]{Klebanov2020CME}.

\begin{theorem}[{CME representation; \citealt[Theorem~5.3]{Klebanov2020CME}}]\label{thm:cme_representation}
Under \Cref{assump:two_RKHS_setting}, if
\begin{equation}\label{eq:cme_assumption}
Kg
\coloneqq
\bE[g(Y) \, | \, X = \quark] \in \cH_{k_{\cX}}
\qquad
\text{for every } g \in \cH_{k_{\cY}},
\end{equation}
then the linear operator
$C_{XX}^{\dagger} C_{XY} \colon \cH_{k_{\cY}} \to \cH_{k_{\cX}}$
is bounded and, $P_X$-almost everywhere,
\begin{align}
\label{eq:uncentred_cme_formula}
\mu_{Y\mid X= \quark}
&=
(C_{XX}^{\dagger} C_{XY})^{*}\, \phi_{\cX},
\\
\label{eq:uncentred_cme_formula_g}
Kg
&=
C_{XX}^{\dagger} C_{XY}\, g,
&&
\text{for every } g \in \cH_{k_{\cY}}.
\end{align}
\end{theorem}

Note that \eqref{eq:uncentred_cme_formula_g} follows from \eqref{eq:uncentred_cme_formula} by the reproducing properties \eqref{equ:basic_reproducing_property_RKHS} and \eqref{equ:KME_reproducing_property}.
Conversely, \eqref{eq:uncentred_cme_formula_g} determines the conditional mean embedding through its action on all $g\in\cH_{k_{\cY}}$.
The preceding theorem shows that the CME representation is available once the representability condition \eqref{eq:cme_assumption} has been verified.
The following corollary gives precisely such a verification criterion in terms of the Radon--Nikodym derivative of the conditional law.

\begin{corollary}[CME representation under the density regularity condition]
\label{cor:regularity_conditional_expectation_operator_CME}
Under \Cref{assump:general_Radon_Nikodym_setting,assump:two_RKHS_setting} with $\cH = \cH_{k_{\cX}}$, if $\cH_{k_{\cY}} \hookrightarrow L^2(\nu_{\cY})$ continuously, then $K \in \HS(\cH_{k_\cY},\cH_{k_\cX})$, the CME assumption \eqref{eq:cme_assumption} is satisfied for every $g \in \cH_{k_{\cY}}$, and the CME representations \eqref{eq:uncentred_cme_formula}--\eqref{eq:uncentred_cme_formula_g} hold.
\end{corollary}
\begin{proof}
Since \Cref{assump:general_Radon_Nikodym_setting} holds with $\cH=\cH_{k_{\cX}}$, \Cref{thm:CEO_bounds_various_spaces}\ref{item:CEO_H_HS} yields $K \in \HS(L^2(\nu_{\cY}) , \cH_{k_{\cX}})$.
Hence, $K \in \HS(\cH_{k_\cY},\cH_{k_\cX})$ as the composition of the continuous embedding $\cH_{k_\cY} \hookrightarrow L^2(\nu_\cY)$ and a Hilbert--Schmidt operator; in particular, $Kg \in \cH_{k_{\cX}}$ for every $g\in\cH_{k_\cY}$.
Thus \eqref{eq:cme_assumption} holds, and the assertion follows from \Cref{thm:cme_representation}.
\end{proof}

\begin{remark}[Special cases of \Cref{cor:regularity_conditional_expectation_operator_CME}]
\label{rem:special_cases_koopman_mu}
The above corollary includes the following cases:
\begin{enumerate}[label=(\roman*)]

\item
The function
\[
\eta(x,y)
=
\frac{\rd P_{Y\mid X=x}}{\rd \nu_{\cY}}(y)
\]
is the transition density of the Markov kernel $P_{Y\mid X=x}$ with respect to $\nu_{\cY}$.
A particularly important case is when $\nu_{\cY}$ is the Lebesgue measure on
$\cY\subseteq\mathbb R^d$.
In this setting, the embedding
$\cH_{k_{\cY}}\hookrightarrow L^2(\nu_{\cY})$
can be ensured by the spectral density condition in \Cref{prop:rkhs_in_L2_via_spectral_density}, which is satisfied for many common kernels.

\item
If $\cX=\bR^d$ or $\cX\subseteq\bR^d$ is a bounded Lipschitz domain and
$\eta \in L^2(\nu_{\cY};\cH)$ for an RKHS
$\cH \simeq \cH_{k_{\cX}}$ with $\ell\ge 1$, then the CME representations hold.
This applies in particular to the following kernels, cf.\ \Cref{prop:matern_wendland_sobolev}:
\begin{enumerate}[label=(\alph*)]
\item
Mat\'ern kernels $k_{\cX}=k_\beta^{\mathrm{Mat}}$ with $2\beta+d\le2\ell$, cf.~\eqref{equ:def_Matern_kernel};
\item
Wendland kernels $k_{\cX}=k_{d,j}^{\mathrm{Wen}}$ with $2j+d+1\le2\ell$, cf.~\eqref{equ:def_Wendland_kernel}.
\end{enumerate}

\item
If $\nu_{\cY} = P_Y$, then
$\cH_{k_{\cY}} \hookrightarrow L^2(\nu_{\cY})$
is automatically satisfied under \Cref{ass:kernel_rv_pair}.
From a Bayesian viewpoint,
\[
\eta(x,y)
\coloneqq
\frac{\rd P_{Y\mid X=x}}{\rd P_Y}(y)
\]
is the likelihood function $L_x(y)$ up to an $x$-dependent normalizing constant
$Z(x)=\bE[L_x(Y)]$; cf.~\Cref{sec:Bayesian_IP}.
\end{enumerate}
\end{remark}

In Sections~\ref{sec:regression}--\ref{sec:koopman_sde_class}, we illustrate how to apply and leverage \Cref{cor:regularity_conditional_expectation_operator_CME} in various applications, including Bayesian inverse problems and conditional mean embeddings for Koopman operator approximation.

\subsection{Error Bounds for Conditional Expectation Operator Approximation} \label{sec:error_bound_implications}

The abstract regularity results of \Cref{sec:cond_exp_koopman_operator} and the covariance-operator representation derived in \Cref{sec:CME_and_abstract_results} naturally give rise to finite-dimensional approximations of the CEO $K$ defined in~\eqref{equ:generic_CEO}. On the one hand, once $K$ is known to be a Hilbert--Schmidt operator, it can be approximated by Galerkin projections onto finite-dimensional subspaces of the underlying RKHSs. On the other hand, the covariance-operator representation \eqref{eq:uncentred_cme_formula_g} yields the empirical conditional mean embedding estimator that is widely used in literature on kernel methods \citep{Grunewalder2012CMEregressors,Klus2020Eigendecompositions}.

In this section, we derive approximation error bounds for such estimators and relate the resulting convergence rates to the regularity of the conditional density established in the previous subsections. In particular, we consider numerical approximation via three approaches:
\begin{enumerate}[label = (\roman*)]
\item
the one-sided Galerkin projection $\widehat{K} = \cP_{\bsX} K$,
\item 
the two-sided Galerkin projection $\widetilde{K} \coloneqq \cP_{\bsX} K \cP_{\bsY}$, also known as kernel EDMD \citep{WillRowl15},
\item 
an approximation $\hat{K}_{\varepsilon}$ to the CME representation $K = C_{XX}^{\dagger} C_{XY}$, cf.\ \eqref{eq:uncentred_cme_formula_g},
\end{enumerate}
where $\cP_{\bsX}$ denotes the orthogonal projection onto $\operatorname{span} \{ \phi_\cX(x_1),\dots,\phi_\cX(x_{n_{\cX}})\} \subset \cH_{k_\cX}$ for data points $\bsX = \{x_1,\dots,x_{n_{\cX}}\} \subset \cX \subseteq \bR^{d}$ (and $\cP_{\bsY}$ accordingly).
For this purpose, we introduce the following notation:

\begin{notation}\label{notation:CEO_approximations}
Let $\bsX = (x_i)_{i=1}^{n_{\cX}}$ be a collection of data points in $\cX$ and 
let $\phi_{\cX} : \cX \to \cH_{k_{\cX}}$ 
denote the feature map associated with a kernel $k_{\cX}$ on $\cX$.
We define
\[
\Phi_{\bsX} = (\phi_{\cX}(x_1), \dots, \phi_{\cX}(x_{n_{\cX}})),
\quad
\text{and}
\quad 
f_{\bsX}
\coloneqq
\Phi_{\bsX}^{\top} f
\coloneqq
(\innerprod{\phi_{\cX}(x_{i})}{f}_{\cH_{k_\cX}})_{i=1}^{n_{\cX}}
=
(f(x_{i}))_{i=1}^{n_{\cX}}
\in\bR^{n_{\cX}},
\]
for $f \in \cH_{k_{\cX}}$, and the corresponding unregularized and regularized Gram matrices
\[
G_{\bsX \bsX} = \Phi_{\bsX}^\top \Phi_{\bsX} = (k_{\cX}(x_i, x_j))_{i,j=1}^{n_{\cX}},
\qquad
G_{\bsX \bsX}^{\varepsilon} = G_{\bsX \bsX} + n_{\cX} \, \varepsilon \, \Id_{n_{\cX}},
\]
where $\varepsilon \geq 0$.
Additionally, let $\bsY = (y_i)_{i=1}^{n_{\cY}}$ be a collection of data points in $\cY$ and 
let $\phi_{\cY} \colon \cY \to \cH_{k_{\cY}}$ 
denote the feature map associated with a kernel $k_{\cY}$ on $\cY$, and define $\Phi_{\bsY}$, $g_{\bsY}$ and $G_{\bsY \bsY}$ similarly.
We denote
\[
\bsK_{\bsX \bsY}
\coloneqq
\big( \bE[ k_\cY(y_j,Y)  \, | \, X=x_i] \big)_{i,j}
\approx
\Big( M^{-1}\sum_{m=1}^{M} k_\cY(y_j,z_{m}^{(i)}) \Big)_{i,j}
\eqqcolon
\widehat{\bsK}_{\bsX \bsY} \, ,
\]
where we used a Monte--Carlo approximation with independent samples $z_{m}^{(i)} \sim P_{Y|X=x_{i}}$.
In the case of independent pairs $(x_{i},y_{i}) \sim P_{XY}$, with $n_{\cX} = n_{\cY} \eqqcolon n$,
the empirical (cross-)covariance operators are defined by
\[
\widehat{C}_{\bsX \bsX}
= \frac{1}{n} \sum_{i=1}^n \phi_{\cX}(x_i) \otimes \phi_{\cX}(x_i)
= \frac{1}{n} \, \Phi_{\bsX} \Phi_{\bsX}^\top,
\quad
\widehat{C}_{\bsY \bsX}
= \frac{1}{n} \sum_{i=1}^n \phi_{\cY}(y_i) \otimes \phi_{\cX}(x_i)
= \frac{1}{n} \, \Phi_{\bsY} \Phi_{\bsX}^\top.
\]
\end{notation}

\begin{table}[!t]
	\centering
	\small
	\setlength{\tabcolsep}{5pt}
	%
	{\renewcommand{\arraystretch}{1.5}%
		\begin{tabular}{@{}l@{\hspace{3.6em}}l@{}}
			\toprule
			\multicolumn{2}{@{}l}{\textbf{Notation \& symbols}\quad (on $\cX$; the corresponding objects on $\cY$ are defined analogously)}\\
			\midrule
			domain $\cX$ &
			feature matrix $\Phi_{\bsX} = (\phi_{\cX}(x_i))_{i=1}^{n_\cX}$\\
			random variable $X \colon \Omega \to \cX$ &
			evaluation vector $f_{\bsX} = \Phi_{\bsX}^{\top} f = (f(x_i))_{i=1}^{n_\cX}$\\
			data $\bsX = (x_i)_{i=1}^{n_\cX}$ &
			Gram matrix $G_{\bsX\bsX} = \Phi_{\bsX}^{\top}\Phi_{\bsX}$\\
			kernel $k_{\cX}$ &
			conditional-evaluation matrix $\bsK_{\bsX\bsY} = \bigl(\bE[k_{\cY}(y_j,Y)  \, | \, X = x_i]\bigr)_{i,j}$\\
			RKHS $\cH_{k_\cX}$ &
			empirical covariance operators $\widehat{C}_{XX} = \tfrac{1}{n}\Phi_{\bsX}\Phi_{\bsX}^{\top}$,\;\; $\widehat{C}_{YX} = \tfrac{1}{n}\Phi_{\bsY}\Phi_{\bsX}^{\top}$\\
			feature map $\phi_{\cX}$ &
			fill distance $\fd_{\bsX} = \sup_{x \in \cX} \min_{x_i \in \bsX} \|x - x_i\|_2$\\
			\bottomrule
	\end{tabular}}
	
	\vspace{1.2ex}
	
	{\renewcommand{\arraystretch}{1.6}%
		\begin{tabular}{@{}>{\raggedright\arraybackslash}p{2.38cm}
				>{\raggedright\arraybackslash}p{4.15cm}
				>{\raggedright\arraybackslash}p{4.15cm}
				>{\raggedright\arraybackslash}p{4.15cm}@{}}
			\toprule
			&
			\textbf{One-sided projection}
			&
			\textbf{Two-sided projection}
			&
			\textbf{CME regression}
			\\
			\midrule
			\textbf{operator}
			&
			$K \colon L^2(\nu_\cY) \to \cH_{k_\cX}$
			&
			$K \colon \cH_{k_\cY} \to \cH_{k_\cX}$
			&
			$K \colon \cH_{k_\cY} \to \cH_{k_\cX}$
			\\
			\textbf{approxima-tion ansatz}
			&
			$\widehat{K} = \cP_{\bsX} K$ (projecting images onto trial space)
			&
			$\widetilde{K} = \cP_{\bsX} K \cP_{\bsY}$ (projection in domain and range)
			&
			$K = C_{XX}^{\dagger} C_{XY}$
            \newline
            (covariance representation)
			\\
			\textbf{required data}
			&
			$\bsX$ and $z_m^{(i)} \overset{\text{iid}}{\sim} P_{Y \mid X = x_i}$
			&
			$\bsX$, $\bsY$ and $z_m^{(i)} \overset{\text{iid}}{\sim} P_{Y \mid X = x_i}$
			&
			pairs $(x_i, y_i) \overset{\text{i.i.d.}}{\sim} P_{XY}$
			\\
			\textbf{$K$-estimator}
			&
			$\Phi_{\bsX} G_{\bsX\bsX}^{-1}\, \widehat{\Phi_{\bsX}^{\top} K}$
			&
			$\Phi_{\bsX} G_{\bsX\bsX}^{-1} \widehat{\bsK}_{\bsX\bsY} G_{\bsY\bsY}^{-1} \Phi_{\bsY}^{\top}$
			&
			$\Phi_{\bsX} \bigl(G_{\bsX\bsX}^{\varepsilon}\bigr)^{-1} \Phi_{\bsY}^{\top}$
			\\
			\textbf{deterministic error}
			&
			range projection
			&
			domain and range projection
			&
			regularization
			\\
			\textbf{statistical error}
			&
			Monte Carlo estimation of $(Kg)_{\bsX}$
			&
			Monte Carlo estimation of $\bsK_{\bsX\bsY}$
			&
			empirical estimation of $C_{XX}$ and $C_{XY}$
			\\			
            \textbf{density regularity}
			&
			enters range projection error bounds
			&
			enters range projection error bounds
			&
			implies RKHS inclusion, eigenvalue decay of $C_{XX}$
			\\
			\textbf{main result}
			&
			\Cref{thm:err_bounds_proj_unified}: fill distance bound (range)
			&
			\Cref{thm:err_bounds_proj_unified}: fill distance bound (range and domain)
			&
			\Cref{prop:cme_learning_rate}: learning rate via eigenvalue decay
			\\
			\bottomrule
	\end{tabular}}
	\caption{%
		Notation overview (top) and comparison of the three approximation strategies for CEO (bottom).
Projection-based methods approximate the \emph{action} of the operator on trial
spaces built from the nodes $\bsX$, $\bsY$; CME methods approximate the operator
itself via empirical covariance operators.
The hat in $\widehat{\Phi_{\bsX}^{\top} K}$ denotes the Monte Carlo estimate of
$g \mapsto (Kg)_{\bsX}$ based on the samples $z_m^{(i)}$.}
	\label{tab:notation_and_comparison}
\end{table}

Then we obtain the following formulas for the one-sided and two-sided projections of the CEO,

\begin{empheq}[box=\fbox]{align*}
\cP_{\bsX}
&=
\Phi_{\bsX} G_{\bsX \bsX}^{-1} \Phi_{\bsX}^{\top},
&
\cP_{\bsX} K \cP_{\bsY}
&=
\Phi_{\bsX} G_{\bsX \bsX}^{-1} \bsK_{\bsX \bsY} G_{\bsY \bsY}^{-1} \Phi_{\bsY}^{\top},
\\
\cP_{\bsX} K g
&=
\Phi_{\bsX} G_{\bsX \bsX}^{-1} (Kg)_{\bsX},
&
\cP_{\bsX} K \cP_{\bsY}g
&=
\Phi_{\bsX} G_{\bsX \bsX}^{-1} \bsK_{\bsX \bsY} G_{\bsY \bsY}^{-1} g_{\bsY},
\end{empheq}
where in practice the matrix $\bsK_{\bsX \bsY}$ is approximated by $\widehat{\bsK}_{\bsX \bsY}$ and, similarly, $(Kg)_{\bsX}$ is approximated by $\big( M^{-1}\sum_{m=1}^{M} g(z_{m}^{(i)}) \big)_{i}$.
Further, the empirical CEO and (regularized) CME read (for the final identities see, e.g., \citealt[Remark~2.9]{Klus2020Eigendecompositions})
\begin{empheq}[box=\fbox]{alignat=2}
\label{equ:K_CME_approximation}
K
&\ \approx\ 
\hat{K}_{\varepsilon}
\ \coloneqq\ 
\bigl(\widehat{C}_{\bsX \bsX} + \varepsilon \Id \bigr)^{-1} \widehat{C}_{\bsY \bsX}^{\top}
&\ &=\ 
\Phi_{\bsX} \bigl(G_{\bsX \bsX}^{\varepsilon}\bigr)^{-1} \Phi_{\bsY}^\top \, ,
\\[0.2ex]
\mu_{Y|X=\quark}
&\ \approx\ 
\hat{K}_{\varepsilon}^{\ast} \circ \phi_{\cX}
&\ &=\ 
\Phi_{\bsY} \bigl(G_{\bsX \bsX}^{\varepsilon}\bigr)^{-1} \Phi_{\bsX}^\top \, \circ \phi_{\cX} \, .
\end{empheq}

\Cref{tab:notation_and_comparison} summarizes the notation and different approximation philosophies.

\begin{remark}[Which approximation for which application?]
\label{rem:which_approximation}
The three schemes differ in the data they require (cf.\ \Cref{tab:notation_and_comparison}): the projection-based approximants $\widehat{K}$ and $\widetilde{K}$ need conditional samples $z_m^{(i)} \sim P_{Y \mid X = x_i}$, whereas the CME estimator $\hat{K}_\varepsilon$ only needs joint samples $(x_i, y_i) \sim P_{XY}$.
The former are thus natural for \emph{stochastic dynamical systems} (\Cref{sec:koopman_sde_class}), where conditional sampling amounts to simulating the process forward from $x_i$, while the CME estimator is preferable in \emph{Bayesian inverse problems} (\Cref{sec:Bayesian_IP}), where conditional sampling means costly posterior sampling, and in \emph{nonparametric regression} (\Cref{sec:regression}), where data naturally arrive as i.i.d.\ pairs $(x_i,y_i) \sim P_{XY}$.
\end{remark}

\subsubsection{Error Bounds for Galerkin-type Approximations} \label{subsubsec:galerkin_proj_approx}

The regularization results of \Cref{sec:cond_exp_koopman_operator} imply that CEOs often map into smooth function spaces, such as Sobolev spaces and their RKHS realizations. In the following result, we exploit this regularity gain to derive approximation error bounds for the Galerkin-type approximations $\widehat{K}$ and $\widetilde{K}$, the latter of which is also known as (kernel) extended dynamic mode decomposition (EDMD) \citep{WillRowl15} and has since become a standard tool for data-driven Koopman operator approximation, with quantitative error bounds established by, e.g., \citet{KohnPhil24}, \citet{hertel2025koopmanstochasticdynamicserror}, \citet{PhilScha25:kernel:JNLS}.
For this, we briefly recall the class of kernels induced by \emph{radial basis functions} (RBF), which are a special case of translation-invariant kernels characterized by a function $\psi: \bR_{\ge 0} \to \bR$ such that $k(x, x') = \psi(\|x - x'\|_2)$ for all $x, x'$. 

\begin{theorem}[Convergence rates for projection-based CEO approximations]\label{thm:err_bounds_proj_unified}
    Let $\cH_{k_\cX}$ and $\cH_{k_\cY}$ be RKHSs over $\cX$ and $\cY$, respectively. Assume that \Cref{assump:general_Radon_Nikodym_setting} holds with $\cH = \cH_{k_\cX}$. Let $\cF$ be a normed linear space of functions over $\cX$ such that $\cH \hookrightarrow \cF$ continuously, and assume that $\cH_{k_\cY} \hookrightarrow L^2(\nu_\cY)$ continuously. Let data sets $\bsX = (x_i)_{i=1}^{n_\cX} \subset \cX$ and $\bsY = (y_i)_{i=1}^{n_\cY} \subset \cY$ be given.
    Then, using \Cref{notation:CEO_approximations}, the following hold:

    \begin{enumerate} [label = (\alph*)]
        \item \label{item:thm_err_bounds_proj_unified_general_one_sided}
        $\displaystyle \norm{K - \widehat{K}}_{\cL(L^2(\nu_\cY) , \cF)} \leq \norm{\eta}_{L^2(\nu_{\cY};\cH)} \, \norm{\Id - \cP_{\bsX}}_{\cL(\cH,\cF)}$, where $\displaystyle \norm{\eta}_{L^2(\nu_{\cY};\cH)} < \infty$. 
        \item \label{item:thm_err_bounds_proj_unified_general_two_sided}
        There exist $C_\cX, C_\cY > 0$ such that
        \begin{equation}        \begin{aligned}\label{eq:general_bound_two_sided_proj}
            \displaystyle \norm{K - \widetilde{K}}_{\cL(\cH_{k_\cY},\cF)} &\leq C_{\cY} \norm{\eta}_{L^2(\nu_\cY; \cH_{k_\cX})} \, \norm{\Id -\cP_{\bsX}}_{\cL(\cH_{k_\cX}, \cF)} \nonumber \\ 
            &+ C_{\cX} \, \norm{\eta}_{L^2(\nu_\cY; \cH_{k_\cX})} \, \norm{\Id-\cP_{\bsY}}_{\cL(\cH_{k_\cY}, L^2(\nu_\cY))}.
        \end{aligned}
        \end{equation}
        \item
        \textbf{Generic RBF error rates:}\label{item:thm_err_bounds_proj_unified_RBF}
        Let $\cX \subset \bR^d$ be a Lipschitz bounded domain and $k_{\cX}$ be a radial basis function kernel satisfying \Cref{ass:kernel_rv_pair}\ref{item:basic_assumptions_kernel} such that $k_{\cX} \in C^{2m_\cX}(\cX \times \cX)$, $m_\cX \in \bN_{0}$, with continuous derivatives on $\overline{\cX} \times \overline{\cX}$ of order $2m_\cX$.
        Then there exists $C > 0$ such that
        \begin{equation}
        \label{equ:one_sided_rbf_bound_unified}
            \norm{K - \widehat{K}}_{\cL(L^2(\nu_\cY),C_{b}(\cX))} \leq
            C \, \norm{\eta}_{L^2(\nu_{\cY};\cH)} \, \fd_{\bsX}^{m_\cX} \, .
        \end{equation}
        If, additionally, $\cY \subset \bR^d$ is Lipschitz bounded, $\nu_\cY \ll \bslambda$, and $k_\cY \in C^{2m_{\cY}}(\cY \times \cY)$ is a radial basis function kernel satisfying \Cref{ass:kernel_rv_pair}\ref{item:basic_assumptions_kernel} with continuous derivatives on $\overline{\cY} \times \overline{\cY}$ of order $2m_\cY$, $m_{\cY} \in \bN_{0}$, then there exist $C_\cX, C_\cY > 0$ such that 
        \begin{equation}\label{equ:two_sided_rbf_bound_unified}
        \begin{aligned}
            \norm{K - \widetilde{K}}_{\cL(\cH_{k_\cY},C_{b}(\cX))} &\leq C_{\cY} 
            \norm{\eta}_{L^2(\nu_\cY; \cH_{k_\cX})} \, \fd_{\bsX}^{m_\cX}
            + C_{\cX} \, \norm{\eta}_{L^2(\nu_{\cY} ; \cH_{k_\cX})} \, \fd_{\bsY}^{m_\cY}.
        \end{aligned}            
        \end{equation}
        In particular, this applies to Mat\'ern kernels $k_\star = k_\beta^\mathrm{Mat}$, $\beta > 0$, for which $m_\star = \lceil \beta \rceil - 1$, where $\star \in \{\cX, \cY\}$.
        \item \label{item:thm_err_bounds_proj_unified_wendland}
        \textbf{Improved Wendland error rates:}
        If $k_\star= k_{d,j}^\mathrm{Wen}$ is a Wendland kernel for some $\star \in \{\cX, \cY\}$, then, in the corresponding fill-distance terms in \eqref{equ:one_sided_rbf_bound_unified} or \eqref{equ:two_sided_rbf_bound_unified}, the exponent may be replaced by $m_\star = j + 1/2$.
    \end{enumerate}
\end{theorem}
\begin{proof}
    Under \Cref{assump:general_Radon_Nikodym_setting}, the operator $K$ is well-defined and bounded by \Cref{thm:CEO_bounds_various_spaces}\ref{item:CEO_H_HS}, and satisfies $\|K \|_{\cL(L^2(\nu_\cY),\cH)} \leq \norm{\eta}_{L^2(\nu_{\cY};\cH)} < \infty$. Factoring $K - \cP_{\bsX} K = (I - \cP_{\bsX}) K$ together with submultiplicativity of operator norms then immediately implies \ref{item:thm_err_bounds_proj_unified_general_one_sided}.

    If additionally the embedding $\cE_\cY \colon \cH_{k_\cY} \hookrightarrow L^2(\nu_\cY)$ is continuous, the restricted operator $K \in \cL (\cH_{k_\cY} , \cH_{k_\cX})$.
    We now decompose:
    \begin{align*}
    \|K - \widetilde K\|_{\cL(\cH_{k_\cY}, \cF)} &\leq
    \|(\Id-\cP_{\bsX}) K \|_{\cL(\cH_{k_\cY}, \cF)} +
    \|\cP_{\bsX} K (\Id-\cP_{\bsY})\|_{\cL(\cH_{k_\cY}, \cF)} \\
    &\leq \|\Id-\cP_{\bsX}\|_{\cL(\cH_{k_\cX}, \cF)} \, \|K\|_{\cL(L^2(\nu_\cY) , \cH_{k_\cX})} \, \| \cE_{\cY} \|_{\cL(\cH_{k_\cY}, L^2(\nu_\cY))} \\
    &\quad + \|\cP_{\bsX}\|_{ \cL ( \cH_{k_\cX} , \cF )}
    \,
    \|K \|_{\cL(L^2(\nu_\cY) , \cH_{k_\cX})} \, \|\Id-\cP_{\bsY} \|_{\cL(\cH_{k_\cY}, L^2(\nu_\cY))}.
    \end{align*}
    By contractivity of the orthogonal projection, $\norm{\cP_{\bsX}}_{\cL(\cH_{k_\cX}, \cH_{k_\cX})} \leq 1$ and continuity (hence, boundedness) of $\cH_{k_\cX} \hookrightarrow \cF$, there exists $C_\cX > 0$ such that $\|\cP_{\bsX}\|_{\cL(\cH_{k_\cX}, \cF)} \leq C_{\cX}$. Finally, continuity of the embedding $\cE_\cY$ yields $\norm{\cE_{\cY}}_{\cL(\cH_{k_\cY}, L^2(\nu_\cY))} \leq C_{\cY}$, proving \ref{item:thm_err_bounds_proj_unified_general_two_sided}.

    If $\cX \subset \bR^d$ is a Lipschitz bounded domain (and therefore satisfies the interior cone condition
    ; see \citet[][Paragraph 4.11]{Adams2003} and \citet[Appendix A]{KohnPhil24}), \citet[Theorem 11.13]{Wendland2005Scattered}
    and the subsequent discussion then imply the existence of $C > 0$ such that
    \begin{align*}
        \norm{f - \cP_{\bsX} f}_{\infty} \leq C \, \fd_{\bsX}^{m_\cX} \, \norm{f}_{\cH} \qquad \forall f \in \cH.
    \end{align*}
   Then 
   $\norm{\Id - \cP_{\bsX}}_{\cL(\cH,C_{b}(\cX)} = \sup_{\| f \|_\cH = 1 } \|f - P_{\bsX} f \|_\infty \leq C \, \fd_{\bsX}^{m_\cX}$ by definition, so \eqref{equ:one_sided_rbf_bound_unified} follows from \ref{item:thm_err_bounds_proj_unified_general_one_sided} with $\cF = C_b(\cX)$.
    
    If additionally $\cY$ is Lipschitz bounded and $\nu_\cY \ll \bslambda$,  we have $\|f\|_{L^2(\nu_\cY)} \leq \nu_\cY(\cY)^{1/2}\|f\|_\infty \leq C \bslambda(\cY)^{1/2} \norm{f}_\infty$, and therefore $\norm{\Id-\cP_{\bsY}}_{\cL(\cH_{k_\cY}, L^2(\nu_\cY))} \leq \widetilde{C}  \norm{\Id-\cP_{\bsY}}_{\cL(\cH_{k_\cY}, C_b(\cY))}$ with $\widetilde{C} > 0$. If $\cH_{k_\cY}$ is induced by a radial basis function kernel $k_\cY \in C^{2 m_\cY}(\cY \times \cY)$, then \eqref{equ:two_sided_rbf_bound_unified} follows analogously as \eqref{equ:one_sided_rbf_bound_unified}. 

    In particular, by \Cref{prop:matern_wendland_sobolev}, the Mat\'ern kernel $k_\beta^\mathrm{Mat}$ is $\lceil 2 \beta - 1 \rceil$-times continuously differentiable. Furthermore, we have $\lceil 2 \beta - 1 \rceil \geq 2 \lceil \beta - 1 \rceil = 2 (\lceil \beta \rceil - 1)$, as can be seen the following way. For $\beta \in \bN$, it holds that $2 \lceil \beta - 1 \rceil = 2(\beta - 1) < 2 \beta - 1 = \lceil 2 \beta - 1 \rceil$. For $\beta = l + \delta$ with $l \in \bN$, $\delta \in (0,1)$, we have $2 \lceil \beta - 1 \rceil = 2l$ and $\ 2 \beta - 1 \ = \ 2l - 2 \delta + 1 \in (2l - 1, 2l + 1)$, so $\lceil 2\beta - 1\rceil \geq 2l = 2 \lceil \beta - 1 \rceil$. Hence, we have $k_\beta^\mathrm{Mat} \in C^{2m}(\bR^d \times \bR^d)$ for $m = \lceil \beta - 1 \rceil = \lceil \beta \rceil - 1$.     
    Consequently, its derivatives of order $2 m$ are also continuous on $\overline{\cX} \times \overline{\cX}$ and, in the case of Lipschitz bounded $\cY \subset \bR^d$, on $\overline{\cY} \times \overline{\cY}$. Thus, \eqref{equ:one_sided_rbf_bound_unified} and \eqref{equ:two_sided_rbf_bound_unified} hold with $m_\star = \lceil \beta \rceil - 1$, concluding the proof of \ref{item:thm_err_bounds_proj_unified_RBF}. 

    Finally, for Wendland RKHS $k_\star = k_{d,j}^\mathrm{Wen}$, $j \in \bN$, $\star \in \{\cX, \cY\}$, \citet[Theorem 11.17]{Wendland2005Scattered} yields the sharper projection error bound $\norm{\Id - \cP_{\bsX}}_{\cL(\cH_{k_{\cX}},C_{b}(\cX))}\leq C \, \fd_{\bsX}^{j + 1/2}$, which together with \ref{item:thm_err_bounds_proj_unified_general_one_sided} implies \eqref{equ:one_sided_rbf_bound_unified} with convergence rate $m_\star = j + 1/2$. By the same argument as above, under the according assumptions 
    this also implies $\norm{\Id - \cP_{\bsY}}_{\cL(\cH_{k_{\cY}},L^2(\nu_\cY))}\leq C \, \fd_{\bsY}^{j + 1/2}$, such that \eqref{equ:two_sided_rbf_bound_unified} follows analogously as \ref{item:thm_err_bounds_proj_unified_general_two_sided} with convergence rate $m_\star = j + 1/2$, proving \ref{item:thm_err_bounds_proj_unified_wendland}.
\end{proof}

\begin{remark}[Projection as interpolation]\label{rem:rkhs_proj_interpol}
    The projection error bounds in \Cref{thm:err_bounds_proj_unified}\ref{item:thm_err_bounds_proj_unified_RBF} rely on the reproducing kernel Hilbert space structure. Indeed, the orthogonal projection $\cP_{\bsX}$ onto the span of kernel sections coincides with the kernel interpolant at the sampling points $\bsX$ \citep[see, e.g.,][]{KohnPhil24}. Consequently, the estimates reduce to classical interpolation error bounds from scattered data approximation \citep{Wendland2005Scattered}. The use of kernel feature dictionaries in kernel EDMD approximations \citep{WillRowl15, BoldPhilSchaWort2025:kernel_koopman_flexible_sampling} thereby circumvents the ``finite-data error'' of EDMD in arbitrary Hilbert spaces, where the inner products in the solution of the regression problem have to be approximated empirically.
\end{remark}

\begin{remark}[Optimal choice of regularization order]\label{rem:optimal_choice_regularity}
    Note that the ``range projection'' error rate in \eqref{equ:one_sided_rbf_bound_unified} and \eqref{equ:two_sided_rbf_bound_unified} improves with increasing kernel regularity, for example in the case of Sobolev RKHS in \Cref{thm:err_bounds_proj_unified}\ref{item:thm_err_bounds_proj_unified_wendland}.
    However, since in this case the upper bound $\norm{\eta}_{L^2(\nu_{\cY};\cH)}$ typically grows\footnote{Generally, the growth of the operator norm with $\ell$ is expected to be faster than the fill distance rate can compensate for.
    Not only does the number of derivative terms entering the Sobolev norm equal $\binom{\ell + d}{d}$, but higher-order derivatives with respect to $x$ also might produce larger integrals over the bounded domain $\cX$.}
    with regularization order $m_\cX$, the `optimal choice of smoothness' to achieve the tightest error bound depends on the specific transition density $\eta(x,y)$ and its partial derivatives $\partial_x^\alpha \eta(x,y)$ with respect to the conditioning variable $x$. 
\end{remark}

\begin{remark}[One-sided vs.\ two-sided projection schemes]
\label{rem:discussion_one_two_sided_proj}

Classical kernel EDMD
\citep{ZhanZuaz23,BoldPhilSchaWort2025:kernel_koopman_flexible_sampling,
PhilScha24:kernel:ACHA,PhilScha25:kernel:JNLS}
typically employs the two-sided approximation $\widetilde K=\cP_{\bsX}K\cP_{\bsY}$,
usually under the mapping assumption $K\cH_{k_\cY}\subset\cH_{k_\cX}$.\footnote{Often, the case $\cX = \cY$ with identical kernels $k_\cX = k_\cY$ is considered, such that the RKHS $\cH_{k_\cX}$ is assumed \citep{Klus2020Eigendecompositions, Hou2026:SparseOnlineKoopmanLearning} or proven \citep{KohnPhil24, hertel2025koopmanstochasticdynamicserror} to be \emph{invariant} under the Koopman operator $K$.} 
In contrast, the one-sided Galerkin approximation $\widehat K=\cP_{\bsX}K$ acts on the full input space $L^2(\nu_\cY)$ and projects only the image of the operator.

The bounds \eqref{equ:one_sided_rbf_bound_unified} and
\eqref{equ:two_sided_rbf_bound_unified} reveal a fundamental trade-off.
Since the one-sided approximation does not restrict the domain of $K$, it fully exploits the regularizing effect of conditional expectation and incurs only the image-space projection error. In the RBF setting, this error is governed by the fill distance in the (typically much smoother) RKHS $\cH_{k_\cX}$.

The two-sided approximation, on the other hand, introduces an additional domain projection error through the embedding $\cH_{k_\cY}\hookrightarrow L^2(\nu_\cY)$, so that the overall convergence is limited by the slower of the domain and image approximation rates. This additional approximation error is compensated by a substantial computational advantage: the action of $K$ is represented by a single finite-dimensional propagation matrix $\bsK_{\bsX \bsY}$, which can be assembled once and subsequently applied to arbitrary observables. In contrast, the one-sided approximation requires the conditional expectation $(Kg)(x_i)=\bE[g(Y) \, | \, X=x_i]$ to be estimated separately for each observable $g$.
\end{remark}

\subsubsection{Error Bounds for CME-based Approximation}
\label{subsubsec:cme_approx}

By the choice of data points $\bsX$ and $\bsY$, the projection-based approximants $\widehat{K}$ and $\widetilde{K}$ impose deterministic, fill-distance-type error rates. In turn, approximations based on the CME representation $Kg = \bE[g(Y)  \, | \, X = \quark] = C_{XX}^\dagger C_{XY} g$, cf.\ \eqref{eq:uncentred_cme_formula_g}, impose the RKHS setting $K: \cH_{k_\cY} \to \cH_{k_\cX}$, but approximate $K$ directly from sampled data pairs $(x_i,y_i)$ rather than through explicit conditional sampling. This regression-type interpretation, and the resulting empirical estimator $\hat{K}_\varepsilon$, were first developed by \citet{Grunewalder2012CMEregressors} and have since been studied extensively as a nonparametric approximation problem for CEOs \citep{mollenhauer2023nonparametric}.

By \Cref{cor:regularity_conditional_expectation_operator_CME}, the CME representation need not simply be assumed but follows from verifiable conditions, so our abstract results also yield error bounds for CME estimators. Convergence rates for such CME-type regularized regression range from the foundational vector-valued rates of \citet{Caponnetto2007},
through the operator-free consistency rate of \citet{Park2020KCME} and the general-regularization-scheme rates of \citet{MollenhauerMueckeSullivan2022}, to the sharpest, source-condition-adaptive rates of \citet{Li2024Optimal}, which also cover the misspecified case.
The following result is a direct consequence of \citet[Theorem~3]{Li2024Optimal}
with $\alpha = \beta = 1$ and $\gamma = 0$, combined with
\Cref{cor:regularity_conditional_expectation_operator_CME}, which provides the well-specifiedness of the CME in our setting.

\begin{theorem}[Convergence rate for CME approximation]
\label{prop:cme_learning_rate}
Let \Cref{assump:general_Radon_Nikodym_setting,assump:two_RKHS_setting}
hold with $\cH = \cH_{k_\cX}$ and $\cH_{k_\cY} \hookrightarrow L^2(\nu_\cY)$,
and assume that $k_\cX$ is bounded, i.e., $\sup_{x \in \cX} k_\cX(x,x) < \infty$.
Let $(\lambda_j)_{j \in \bN}$ denote the non-increasing eigenvalues of $C_{XX}$,
and suppose there exist $c > 0$ and $p \in (0,1]$ such that
$\lambda_j \leq c\, j^{-1/p}$ for all $j \in \bN$.
Then the regularized CME estimator
$\hat{K}_{\varepsilon_{n}} \coloneqq (\widehat{C}_{XX} + \varepsilon_{n}\,\Id)^{-1}\widehat{C}_{XY}$
based on $n$ i.i.d.\ pairs $(x_i,y_i) \sim P_{XY}$
with $\varepsilon_n \asymp n^{-1/(1+p)}$ satisfies,
for all $\tau \geq \log(5)$,
\begin{equation}
\label{equ:CME_approximation_rate}
\norm{\hat{K}_{\varepsilon_n} - K}_{\mathsf{HS}(\cH_{k_\cY},\,L^2(P_X))}^2
\leq
\tau^2 A\,
n^{-\frac{1}{1+p}}
\end{equation}
with $P_{XY}^n$-probability at least $1 - 5 e^{-\tau}$,
for some constant $A > 0$ independent of $\tau$
and all sufficiently large $n$.
In particular, the assumption $\lambda_j \leq c\, j^{-1/p}$ always holds with
$p = 1$ and rate $n^{-1/2}$.
If moreover $\cH_{k_\cX} \simeq H^\ell(\cX)$
on a bounded Lipschitz domain $\cX \subset \bR^d$ with $\ell > d/2$
(e.g.\ for Mat\'ern or Wendland kernels, cf.\ \Cref{prop:matern_wendland_sobolev}),
then $p = d/(2\ell) < 1$ and the rate improves to $n^{-2\ell/(2\ell+d)}$.
\end{theorem}

\begin{proof}
In the notation of \citet{Li2024Optimal},
we verify the assumptions of their Theorem~2 as follows.
\Cref{cor:regularity_conditional_expectation_operator_CME} yields
$K \in \HS(\cH_{k_\cY},\cH_{k_\cX})$,
so $K^* \in \HS(\cH_{k_\cX},\cH_{k_\cY})$,
which is the well-specified case of \citet{Li2024Optimal}
with output space $Y = \cH_{k_\cY}$, covariate RKHS $H_X = \cH_{k_\cX}$,
and regression function $F^* = K^* \circ \, \phi_{\cX} \colon \cX \to \cH_{k_\cY}$,
giving (SRC) with $\beta = 1$.
Boundedness of $k_\cX$ gives (EMB) at level $\alpha = 1$
since $[\cH_{k_\cX}]^1 \simeq \cH_{k_\cX} \hookrightarrow L^\infty(P_X)$
with constant $A = \sup_{x \in \cX} k_\cX(x,x)^{1/2}$.
The eigenvalue decay assumption is (EVD) with the given $p$.
Since $\beta + p = 1 + p > 1 = \alpha$ for all $p \in (0,1]$,
we are always in Case~2 of \citet[Theorem~3]{Li2024Optimal} with $\gamma = 0$,
yielding
\[
\|\hat{K}_{\varepsilon_n}^* - K^*\|_{\mathsf{HS}(L^2(P_X),\,\cH_{k_\cY})}^2
\leq \tau^2 A\, n^{-1/(1+p)}.
\]
Since the Hilbert--Schmidt norm is invariant under taking adjoints, we obtain the stated bound.
The regularized CME estimator $\hat{K}_{\varepsilon}^*$ coincides with
the estimator of \citet{Li2024Optimal}
as shown in \citet[Section~3]{mollenhauer2023nonparametric};
see also the proof of \citet[Theorem~3]{Li2024Optimal}.
Since $\norm{C_{XX}}_{\mathsf{Tr}} = \bE[k_\cX(X,X)] < \infty$
by \Cref{ass:kernel_rv_pair}, $C_{XX}$ is trace class on $\cH_{k_\cX}$,
implying $\lambda_j = O(j^{-1})$, i.e., (EVD) with $p = 1$
and rate $n^{-1/2}$.
For $\cH_{k_\cX} \simeq H^\ell(\cX)$
on a bounded Lipschitz domain $\cX \subset \bR^d$ with $\ell > d/2$,
note that $C_{XX} = (\cE^{L^2(P_X)})^* \, \cE^{L^2(P_X)} \colon \cH_{k_\cX} \to \cH_{k_\cX}$
\citep[Section~2]{Li2024Optimal},
so the eigenvalues of $C_{XX}$ are the squared singular values of the embedding
$\cH_{k_\cX} \hookrightarrow L^2(P_X)$.
By norm-equivalence $\cH_{k_\cX} \simeq H^\ell(\cX)$ and the classical decay
$a_j(\cE^{L^2(\cX)}) \asymp j^{-\ell/d}$ of the approximation numbers
of the Sobolev embedding $\cE^{L^2(\cX)} \colon H^\ell(\cX) \hookrightarrow L^2(\cX)$
\citep{EdmundsTriebel1996:function_spaces_entropy},
we obtain $\lambda_j(C_{XX}) \asymp j^{-2\ell/d}$,
giving $p = d/(2\ell)$ and rate $n^{-2\ell/(2\ell+d)}$.
\end{proof}

\begin{remark}[Eigenvalue decay on unbounded domains]
\label{rem:evd_unbounded}
On unbounded domains such as $\cX = \bR^d$, the Sobolev embedding
$H^\ell(\bR^d) \hookrightarrow L^2(\bR^d)$ is not compact, so the
sharp rate $p = d/(2\ell)$ of \Cref{prop:cme_learning_rate} is not
available in general, and $p = 1$ is the best guaranteed by trace class alone.
For $\cH_{k_\cX} \simeq H^\ell(\bR^d)$ and $P_X$ with
sufficiently fast-decaying tails, a sharper rate $p < 1$ may be recoverable
via weighted Sobolev embedding theory
\citep{EdmundsTriebel1996:function_spaces_entropy},
but this falls outside the scope of the present work.
\end{remark}

\begin{remark}[Strong norm bounds from extra smoothness]
If the transition density satisfies $\eta \in L^2(\nu_\cY; H^s(\cX))$
for some $s > \ell$,
then applying \citet[Theorem~3]{Li2024Optimal} with $\gamma = 1$
yields a bound in the stronger norm $\mathsf{HS}(\cH_{k_\cY},\,\cH_{k_\cX})$ of the form
\[
\|\hat{K}_{\varepsilon_n} - K\|_{\mathsf{HS}(\cH_{k_\cY},\,\cH_{k_\cX})}^2
\leq
\tau^2 A\, n^{-\frac{s-\ell}{s+d/2}}
\]
with high probability.
In particular, stronger smoothness of the conditional law ($s$ larger)
gives a faster rate, reflecting the regularizing effect of the transition kernel.
\end{remark}

\begin{remark}[CME convergence for trajectory data]
\label{rem:cme_trajectory}
\Cref{prop:cme_learning_rate} assumes i.i.d.\ data pairs
$(x_i, y_i) \sim P_{XY}$.
In applications where data arise as a single trajectory of a stochastic
process, the same rate is expected to hold up to logarithmic corrections,
provided the process is $\alpha$-mixing with summable coefficients
$\sum_{t=1}^{\infty} \alpha(t) < \infty$ and the kernel $k_X$ is bounded.
This condition is satisfied, for instance, for geometrically ergodic
Markov processes (with $\alpha(t) = O(e^{-ct})$), including solutions
to SDEs satisfying common assumptions (cf.\ \Cref{ass:sde_unif_elliptic_regular} in \Cref{sec:koopman_sde_class}).
A rigorous proof would replace the Bernstein inequality for i.i.d.\
variables used in the variance bound of \cite{Li2024Optimal} with the
analogous concentration inequality for $\alpha$-mixing processes from
\cite[Theorem~21]{Mollenhauer2022kernel_autocovariance}, and is left for future work.
\end{remark}


\section{\texorpdfstring{Verification of the Density Regularity Condition $\eta \in L^2(\nu_{\cY};\cH)$}{Verification of the Density Regularity Condition}}

\label{sec:verification_density_regularity}

In the previous section, we established that the condition $\eta \in L^2(\nu_{\cY};\cH)$ implies a range of regularity, representation, and approximation results for conditional expectation operators. The purpose of the upcoming section is to show that \Cref{assump:general_Radon_Nikodym_setting} can be verified in several important problem classes. While the underlying models differ substantially, we exploit problem-specific regularity of the conditional density in every outlined application to establish the required Hilbert space regularity.
The abstract results of \Cref{sec:consequences_density_regularity} apply in each of these settings, once the standing assumptions have been verified. As the applications below address different communities with different interests, we restate only a selection of these results in each case.

\begin{remark}[CME rates for bounded vs.\ unbounded conditioning domains]
\label{rem:rate_bounded_vs_unbounded}
Whenever $k_\cX$ is bounded, i.e., $\sup_{x\in\cX} k_\cX(x,x)<\infty$ (in particular, whenever $\cX$ is bounded and $k_\cX$ continuous), the CME convergence rate of \Cref{prop:cme_learning_rate} applies directly, since it only requires this boundedness together with the eigenvalue decay of $C_{XX}$. For unbounded $\cX = \bR^d$, we instead work with the exponentially reweighted kernel $k_\tau$ from \Cref{lemma:reweighted_kernel}, whose reweighting compensates for the lack of spatial integrability but whose diagonal is necessarily unbounded.
To the best of our knowledge, none of the existing convergence-rate results for CME-type or vector-valued ridge regression estimators \citep{Caponnetto2007,Park2020KCME,MollenhauerMueckeSullivan2022,Li2024Optimal} cover kernels with unbounded diagonal; we also state only qualitative Hilbert--Schmidt and CME-representability results below whenever $\cX$ is unbounded, and reserve quantitative rates for the case of bounded $\cX$.
Establishing such a rate for the reweighted setting appears genuinely challenging and we leave this as an interesting direction for future work.
\end{remark}

\subsection{Nonparametric Regression (Forward Uncertainty)}
\label{sec:regression}

In statistical learning theory and nonparametric statistics,
a canonical object of study is the regression model
\begin{equation}
	\label{equ:regression_setting}
	Y = f(X) + \varepsilon,
\end{equation}
where $X \colon \Omega \to \cX$ is the covariate, $Y \colon \Omega \to \bR$
the response, $f \colon \cX \to \bR$ the unknown signal,
and $\varepsilon$ a centred noise variable independent of $X$.
The classical goal is to estimate the regression function
$f^*(x) = \bE[Y  \, | \, X=x] = f(x)$ from i.i.d.\ observations
$(x_i, y_i) \sim P_{XY}$.
Beyond point estimation of $f^*$, the CME $\mu_{Y \mid X=x} \in \cH_{k_\cY}$
provides a unified representation of $P_{Y \mid X=x}$
from which $\bE[g(Y) \,|\, X=x]$ can be recovered for any $g \in \cH_{k_\cY}$ via the reproducing property \eqref{equ:KME_reproducing_property},
for instance for (forward) uncertainty quantification
\citep{Gneiting2007Probabilistic}.

The following theorem verifies the density regularity condition
$\eta \in L^2(\nu_{\cY}; \cH_{k_\cX})$ (\Cref{assump:general_Radon_Nikodym_setting}) in the regression setting \eqref{equ:regression_setting} under the \emph{source condition} $f \in H^\ell(\cX)$ \citep{Caponnetto2007, Fischer2020sobolev},
and the subsequent corollary establishes the validity of the CME representation $K = C_{XX}^\dagger C_{XY}$
and the optimal estimation rate of \Cref{prop:cme_learning_rate}.

\begin{theorem}[Sobolev regularity criterion for nonparametric regression]
	\label{thm:kRR_well_specified_homoskedastic}
	Let $(\cY,\nu_\cY) = (\bR,\bslambda)$, $\cX \subseteq \bR^{d}$ and consider the regression model \eqref{equ:regression_setting},
	where $f \in H^\ell(\cX)$ for some $\ell \in \bN$ with $\ell > d/2$ and $\varepsilon$ is independent of $X$ and admits a zero-mean Lebesgue density $p_\varepsilon \in W^{\ell,2}(\bR)$.\footnote{This covers Gaussian noise
	$\varepsilon \sim \cN(0,\sigma^2)$, $\sigma > 0$, for every $\ell \in \bN$.}
    Then the Radon--Nikodym derivative
    $\eta(x,y) \coloneqq \frac{\rd P_{Y\mid X=x}}{\rd \nu_{\cY}}(y)$
    exists and satisfies
    $\eta(x,y) = p_\varepsilon\big(y - f(x)\big)$.
    Further,
    \begin{equation}
    \label{equ:conditional_density_regularity_regression}
    \eta \in L^2(\nu_\cY;\cH),
    \qquad
	K \in \HS(L^2(\nu_\cY),\cH)
    \quad
    \text{with}
    \quad    
	\norm{K}_{\HS(L^2(\nu_\cY),\cH)} = \norm{\eta}_{L^2(\nu_\cY;\cH)},
    \end{equation}    
    if either of the following holds:
	\begin{enumerate}[label=(\Alph*)]
	\item
	\label{item:kRR_homosked_bounded}
	\textbf{bounded domain, Sobolev RKHS:}
	$\cX \subset \bR^d$ is a bounded Lipschitz domain and
	$\cH \simeq H^\ell(\cX)$.
    In this case we additionally have $K \in \Tr( L^2(\nu_{\cY}) , L^2(\nu_{\cX}) )$ for $\nu_{\cX} = \bslambda$;
	\item
	\label{item:kRR_homosked_Rd}
	\textbf{unbounded domain, reweighted Sobolev RKHS:}
	$\cX = \bR^d$, $k_0$ is a kernel on
	$\bR^d$ with $\cH_{k_0} \simeq H^\ell(\bR^d)$, and $\cH = \cH_{k_{\cX}}$ where $k_{\cX} = k_\tau$ is the Gaussian-reweighted kernel \eqref{equation:Gaussian_reweighted_kernel} for $\tau > 0$ and symmetric, strictly positive definite $\Gamma \in \bR^{d\times d}$.\footnote{The
	reweighting cannot be omitted: since
	$\int_\bR \absval{\eta(x,y)}^2\,\rd y
	= \norm{p_\varepsilon}_{L^2(\bR)}^2 > 0$ for every $x \in \bR^d$ by
	translation invariance, one has
	$\eta \notin L^2(\nu_\cY;H^\ell(\bR^d))$ already for $\ell = 0$,
	regardless of $f$.\label{footnote:reweighting_necessary}}
	\end{enumerate}
\end{theorem}

\begin{proof}
The proof is given in \Cref{sec:proofs_regression}.
\end{proof}

\begin{corollary}[CME representation for nonparametric regression]
	\label{cor:kRR_CME_representation}
	Let the assumptions of \Cref{thm:kRR_well_specified_homoskedastic} hold
	and assume that the canonical embedding
	$\cH_{k_\cX} \hookrightarrow L^2(P_X)$ is injective. In case
	\ref{item:kRR_homosked_bounded}, assume additionally that $k_\cX$ is
	measurable; in case \ref{item:kRR_homosked_Rd}, assume that $k_0$ is
	continuous and that
	\begin{equation}
	\label{equ:regression_exp_moment_X}
	\bE\big[\exp\big(\tau \norm{X}_{\Gamma}^{2}\big)\big] < \infty.
	\end{equation}
	Then $(X,k_\cX)$ satisfies \Cref{ass:kernel_rv_pair}.
	Further, if $k_\cY$ is a kernel on $\bR$ such that $(Y,k_\cY)$ also
	satisfies \Cref{ass:kernel_rv_pair} and
	$\cH_{k_\cY} \hookrightarrow L^2(\nu_\cY)$ continuously, then
	\begin{enumerate}[label=(\alph*)]
	\item
	\label{item:kRR_CME_representation_HS}
	$K \in \HS(\cH_{k_\cY},\cH_{k_\cX})$, the CME assumption
	\eqref{eq:cme_assumption} is satisfied for every $g \in \cH_{k_\cY}$,
	and the CME representations
	\eqref{eq:uncentred_cme_formula}--\eqref{eq:uncentred_cme_formula_g}
	hold;
	\item
	\label{item:kRR_CME_representation_rate}
	in case \ref{item:kRR_homosked_bounded}, the CME estimator
	$\hat{K}_{\varepsilon}$ in \eqref{equ:K_CME_approximation} satisfies the
	estimation rate \eqref{equ:CME_approximation_rate}.\footnote{In case
	\ref{item:kRR_homosked_Rd}, \Cref{prop:cme_learning_rate} does not
	apply, since $\sup_{x \in \bR^d} k_\tau(x,x) = \infty$, cf.\ \Cref{rem:rate_bounded_vs_unbounded}.}
	\end{enumerate}
\end{corollary}

\begin{proof}
The proof is given in \Cref{sec:proofs_regression}.
\end{proof}

\subsection{Bayesian Inverse Problems (Backward Uncertainty)}
\label{sec:Bayesian_IP}

The results of \Cref{cor:regularity_conditional_expectation_operator_CME}---and hence the CME assumption
\eqref{eq:cme_assumption} and representations
\eqref{eq:uncentred_cme_formula}--\eqref{eq:uncentred_cme_formula_g}---apply in
particular to the standard setting of Bayesian inverse problems
\citep{Stuart2010IP}, where $X$ is a noisy measurement of $F(Y)$ corrupted by
Gaussian noise, under rather mild tail conditions on $F(Y)$.
Throughout this section, we use the following notation for a reweighted inner product on $\bR^{d}$ and the corresponding norm, where $\Gamma \in \bR^{d \times d}$ denotes symmetric and positive definite matrix and $x,y \in \bR^{d}$:
\[
\innerprod{x}{y}_{\Gamma} \coloneqq x^\top \Gamma^{-1} y,
\qquad
\norm{x}_{\Gamma}^{2} \coloneqq \innerprod{x}{x}_{\Gamma}
= x^\top \Gamma^{-1} x .
\]

\begin{assumption}[Bayesian setting]
	\label{assump:Bayesian_setting}
    $\cX = \bR^d$, $\cY$ is a standard Borel space, $Y \colon \Omega \to \cY$ a random variable, and
	\[
	X = F(Y) + \varepsilon,
	\qquad
	\varepsilon \sim \cN(0,\Gamma),
	\]
	where $F \colon \cY \to \bR^d$ is measurable,
	$\Gamma \in \bR^{d\times d}$ is symmetric and strictly positive definite,
	and $\varepsilon$ is independent of $Y$.
    We denote by $\eta(x,y) \coloneqq \frac{\rd P_{Y\mid X=x}}{\rd P_{Y}}(y)$ the Radon--Nikodym derivative of the posterior with respect to the prior, that is,
    \begin{equation}
    \label{equ:Radon_Nikodym_BIP_exponential_tilting}
    \eta(x,y)	
	=
	\frac{\exp\big( - \tfrac{1}{2} \norm{ x - F(y) }_{\Gamma}^{2} \big)}{\int \exp\big( - \tfrac{1}{2} \norm{ x - F(\tilde y) }_{\Gamma}^{2} \big) P_{Y}(\rd \tilde y)}
	=
	\frac{\exp\!\Big(
	\innerprod{x}{F(y)}_{\Gamma}
	-
	\frac12 \norm{F(y)}_\Gamma^2
	\Big)}{\Xi(x)},
    \end{equation}	
    where
    \[
	\Xi(x)
	\coloneqq
	\int_{\cY}
	\exp\big(
	\innerprod{x}{F(y)}_{\Gamma}
	-\tfrac12 \norm{F(y)}_\Gamma^2
	\big)\,P_Y(\rd y)
    > 0.
	\]
\end{assumption}

Recall from \Cref{lemma:reweighted_kernel} the Gaussian-reweighted kernel
$k_{\tau}(x,x') = m_\tau(x)^{-1}\, k_0(x,x')\, m_\tau(x')^{-1}$ with weight
$m_\tau(x) = \exp\big(-\tfrac{\tau}{2}\norm{x}_{\Gamma}^{2}\big)$,
cf.~\eqref{equation:Gaussian_reweighted_kernel}, whose RKHS consists of all
functions $f$ with $m_\tau f \in \cH_{k_0}$.

\begin{theorem}[Weighted Sobolev regularity criterion for Bayesian inverse problems]
	\label{thm:unbounded_F_weighted_sobolev}
	Let \Cref{assump:Bayesian_setting} hold.
	Let $\ell \in \bN_0$ and let
	$k_0$ be a continuous, translation-invariant, symmetric and positive definite kernel on $\cX = \bR^d$ such that
	$\cH_{k_0} \simeq H^{\ell}(\cX)$.\footnote{By \Cref{prop:matern_wendland_sobolev}, Mat\'ern and Wendland kernels satisfy this condition for certain choices of parameters.}
	Let $\tau > 0$ and $k_{\cX}(x,x') \coloneqq k_{\tau}(x,x')$ with the reweighted kernel $k_{\tau}$ given by \eqref{equation:Gaussian_reweighted_kernel}.
    Assume that $\Lambda \coloneqq \log \Xi \in C^{\ell}(\cX)$ and that its derivatives grow at most polynomially, that is, there exist $C_{\Lambda}>0$ and $r \geq 1$ such that, for every multi-index
	$\alpha\in \bN_0^d$ with $1\le |\alpha|\le \ell$ and all $x \in \cX$,
	\begin{equation}
		\label{eq:log_partition_growth_assumption}
		|D^\alpha \Lambda(x)|
		\le
		C_{\Lambda} (1+\|x\|_2)^r.
	\end{equation}
    Then $\eta(y)\in \cH_{k_{\cX}}$ for $P_Y$-almost every $y\in \cY$.
	Moreover, $\eta \in L^2(P_Y ; \cH_{k_{\cX}})$, and $K \in \HS(L^2(P_{Y}),\cH_{k_{\cX}})$ 
	with 
	$\norm{K}_{\HS(L^2(P_{Y}),\cH_{k_{\cX}})} = \norm{\eta}_{L^2(P_Y ; \cH_{k_{\cX}})}$ if
    \begin{equation}
    \label{equ:conditions_on_tau}
        \tau > 1
        \qquad\quad
        \text{or}
        \qquad\quad
        \exists \lambda > 0\colon
        \quad
        \tau > \frac{1}{1+\lambda}
        \ \text{ and }\ 
        \bE\big[\exp \big( \lambda \norm{F(Y)}_\Gamma^2 \big) \big] < \infty.
    \end{equation}
\end{theorem}

\begin{proof}
The proof is given in \Cref{sec:proofs_bip}.
\end{proof}

Theorem~\ref{thm:unbounded_F_weighted_sobolev} is formulated in terms of the
growth condition \eqref{eq:log_partition_growth_assumption} on the log-partition
function
\[
\Lambda(x)=\log \Xi(x),\qquad
\Xi(x)=\int_{\cY}
\exp\!\Big(
\innerprod{x}{F(y)}_{\Gamma}
-\tfrac12\|F(y)\|_\Gamma^2
\Big)\,P_Y(\rd y).
\]
In the Bayesian setting of \Cref{assump:Bayesian_setting}, the Radon--Nikodym
derivative \eqref{equ:Radon_Nikodym_BIP_exponential_tilting}
shows that the posterior distribution \(P_{Y \, | \,  X=x}\) is obtained from the
prior \(P_Y\) by exponential tilting with potential
\(\innerprod{x}{a(y)}-\tfrac12\norm{F(y)}_\Gamma^2\), where $a(y)\coloneqq\Gamma^{-1}F(y)$.
The derivatives of \(\Lambda\) admit the probabilistic representations
\[
\nabla\Lambda(x)=\bE[a(Y) \, | \,  X=x], \qquad
D^2\Lambda(x)=\Cov(a(Y) \, | \,  X=x).
\]
More generally, \(D^\alpha\Lambda(x)\) equals the cumulant of order \(|\alpha|\)
of \(a(Y)\) under the conditional law \(P_{Y \, | \,  X=x}\).
Hence the growth condition \eqref{eq:log_partition_growth_assumption}
amounts to a uniform control of the cumulants of \(a(Y)\)
under the posterior laws \(P_{Y \, | \,  X=x}\).
The following proposition provides several concrete sufficient conditions
ensuring that this assumption holds.
Their logical relations are illustrated in
\Cref{fig:log_partition_implications}.

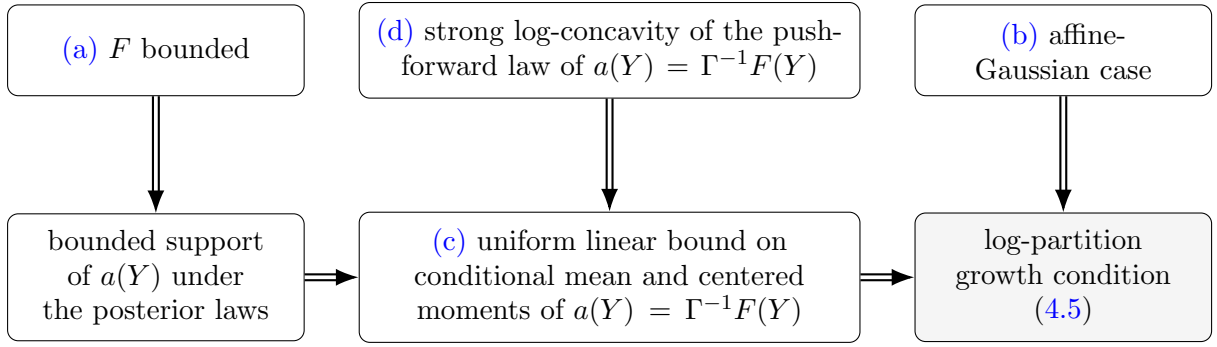
\begin{figure}[!t]
\centering
\begin{tikzpicture}[
    >=Latex,
    node distance=11mm and 14mm,
    box/.style={
        draw, rounded corners, align=center, inner sep=6pt,
        text width=6.2cm, minimum height=1.2cm
    },
    smallbox/.style={
        draw, rounded corners, align=center, inner sep=6pt,
        text width=3.5cm, minimum height=1.2cm
    },
    arr/.style={->, thick, double, double distance=1.4pt},
    lab/.style={font=\footnotesize, inner sep=2pt},
    ]
    
    \node[smallbox] (bounded) at (-6,3)
    {%
        \ref{item:logpart_bounded} $F$ bounded%
    };
    
    \node[smallbox] (affine) at (6,3)
    {%
        \ref{item:logpart_affine_gaussian} affine-Gaussian case%
    };
    
    \node[smallbox] (rawmom) at (-6,0)
    {%
        bounded support of
        \(a(Y)\) under the posterior laws%
    };

    \node[box] (tilted) at (0,0)
    {%
        \ref{item:logpart_tilted_moments}
        uniform linear bound on
        conditional mean and centered moments
        of \(a(Y)=\Gamma^{-1}F(Y)\)%
    };

    \node[box] (structural) at (0,3)
    {%
        \ref{item:logpart_logconcave}
        strong log-concavity of the
        pushforward law of \(a(Y)=\Gamma^{-1}F(Y)\)%
    };
    
    \node[smallbox, fill=gray!8] (lambda) at (6,0)
    {%
        log-partition growth condition\\
        \eqref{eq:log_partition_growth_assumption}%
    };
    
    \draw[arr] (bounded) -- (rawmom);
    \draw[arr] (rawmom) -- (tilted);
    \draw[arr] (affine) -- (lambda);
    \draw[arr] (structural) -- (tilted);
    \draw[arr] (tilted) -- (lambda);
    
\end{tikzpicture}
\caption{Logical relations between the sufficient conditions in
    \Cref{cor:log_partition_sufficient_conditions} ensuring the log-partition growth condition \eqref{eq:log_partition_growth_assumption}.}
\label{fig:log_partition_implications}
\end{figure}

\begin{proposition}[Concrete sufficient conditions for the log-partition growth condition \eqref{eq:log_partition_growth_assumption}]
	\label{cor:log_partition_sufficient_conditions}
    The log-partition growth condition \eqref{eq:log_partition_growth_assumption} in \Cref{thm:unbounded_F_weighted_sobolev} is satisfied if any of the following hold:
	\begin{enumerate}[label=(\alph*)]
		
		\item
		\label{item:logpart_bounded}
		$F$ is bounded (or \(F\) is \(P_Y\)-essentially bounded).
		
		\item
		\label{item:logpart_affine_gaussian}				
		\(\cY\) is a Banach space, \(Y\) is Gaussian, and \(F(y)=Ay+b\)
		with $b \in \bR^{d}$ and a bounded linear operator \(A:\cY\to\mathbb R^d\).
		
		\item
		\label{item:logpart_tilted_moments}
		Let \(a(y)\coloneqq \Gamma^{-1}F(y)\) and assume \(\Xi\in C^\ell(\bR^d)\).
		For \(x\in\bR^d\), define the conditional mean and conditional centered moments by
		\[
		m_1^a(x)\coloneqq \bE[a(Y) \, | \,  X=x],
		\qquad
		m_j^a(x)\coloneqq \bE \big[\|a(Y)-m_1^a(x)\|_2^j  \, | \,  X=x\big],
		\quad j=2,\dots,\ell.
		\]
		Assume that there exists a constant \(c>0\), such that, for every \(x\in\bR^d\),
		\[
		\max \big( \|m_1^a(x)\|_2, m_2^a(x),\dots,m_\ell^a(x) \big) \le c \, (1+\|x\|_2).
		\]
	
		\item
		\label{item:logpart_logconcave}
		The pushforward law of \(a(Y)=\Gamma^{-1}F(Y)\) is strongly log-concave.
	\end{enumerate}
\end{proposition}

\begin{proof}
The proof is given in \Cref{sec:proofs_bip}.
\end{proof}

\begin{corollary}[CME representation for Bayesian inverse problems]
\label{cor:unbounded_F_weighted_sobolev_CME}
Let the assumptions of \Cref{thm:unbounded_F_weighted_sobolev} hold and let $\tau > 0$ and $\lambda > 0$ be such that\footnote{Note that the two conditions on $\tau$ in \eqref{equ:balancing_tau_and_lambda} are jointly satisfiable only if $\lambda > \frac{1+\sqrt5}{2}$, the golden ratio.}
\begin{equation}
\label{equ:balancing_tau_and_lambda}
    \frac{1}{1+\lambda} < \tau < \frac{\lambda}{1+2\lambda}
    \qquad
    \text{and}
    \qquad
    \bE\big[\exp \big( \lambda \norm{F(Y)}_\Gamma^2 \big) \big] < \infty.
\end{equation}
Then the pair $(X,k_{\cX})$ satisfies \Cref{ass:kernel_rv_pair}.
Further, if $k_{\cY}$ is a kernel on $\cY$ such that $(Y,k_{\cY})$ also satisfies \Cref{ass:kernel_rv_pair}, then
$K \in \HS(\cH_{k_\cY} , \cH_{k_\cX})$,
the CME assumption \eqref{eq:cme_assumption} is satisfied for every $g \in \cH_{k_{\cY}}$,
and the CME representations \eqref{eq:uncentred_cme_formula}--\eqref{eq:uncentred_cme_formula_g} hold.
\end{corollary}

\begin{proof}
The proof is given in \Cref{sec:proofs_bip}.
\end{proof}

\begin{remark}[Role of the $\tau$-reweighting and the exponential-moment assumption in \eqref{equ:conditions_on_tau} and \eqref{equ:balancing_tau_and_lambda}]
	\label{rem:why_reweighting}
	The kernel reweighting in the definition of $k_{\tau}$ (cf.\ \Cref{lemma:reweighted_kernel})
	can, in general, not be omitted.
	Indeed, even in the degenerate case where $F\equiv c$ is constant, one has
	$\eta(\quark,y)\equiv 1$ for every $y$, so $\eta(\quark,y)\notin L^2(\bR^d)$ and hence
	$\eta(\quark,y)$ cannot belong to an RKHS of Sobolev type on $\bR^d$.	
	The reweighting by $m_\tau$ replaces the requirement
	$\eta(\quark,y)\in H^\ell(\bR^d)$
	by the weaker condition
	$m_\tau(\quark)\eta(\quark,y)\in H^\ell(\bR^d)$, ensuring
	$\eta(\quark,y)\in \cH_{k_{\cX}}$.	
	By \Cref{thm:unbounded_F_weighted_sobolev}, this pointwise RKHS-membership is guaranteed under the stated assumptions for every $\tau>0$.
	If $\tau$ is chosen sufficiently large, for example $\tau>1$ as in \eqref{equ:conditions_on_tau}, this even yields
	$\eta\in L^2(P_Y;\cH_{k_{\cX}})$.
    Note that for $\tau > 0$ the kernel $k_{\tau}$ is unbounded, hence \Cref{prop:cme_learning_rate} is not applicable, cf.\ \Cref{rem:rate_bounded_vs_unbounded}.
	
	On the other hand, increasing $\tau$ makes the diagonal of the reweighted kernel grow faster, and therefore makes it harder to satisfy the integrability condition
	$\bE[k_{\cX}(X,X)]<\infty$
	from \Cref{ass:kernel_rv_pair}.
	Thus, condition \eqref{equ:balancing_tau_and_lambda} balances the choice of $\tau$
	so that both
	\[
	\eta\in L^2(P_Y;\cH_{k_{\cX}})
	\qquad\text{and}\qquad
	(X,k_{\cX}) \text{ satisfies \Cref{ass:kernel_rv_pair}}
	\]
	hold simultaneously under the additional assumption
	$\bE\big[\exp\big(\lambda \norm{F(Y)}_\Gamma^2\big)\big]<\infty$ for some $\lambda>\frac{1+\sqrt5}{2}$,
	which is a natural exponential-tail condition on the pushforward random variable $F(Y)$
	and is satisfied, for example, in each of the following situations:
	\begin{enumerate}[label=(\roman*)]
		
		\item
		$F$ is bounded (or \(F\) is \(P_Y\)-essentially bounded).
		
		\item
		$Y$ is Gaussian in a Banach space $\cY$, 
		$F(y)=Ay+b$ with a bounded linear operator $A:\cY\to\mathbb R^d$,
		and the covariance matrix $\Sigma_F$ of the (Gaussian) random variable $F(Y)$ satisfies
		\begin{equation}
		\label{equ:F_Y_covariance_small_relative_to_Gamma}
		\lambda_{\max}\big(\Gamma^{-1/2}\Sigma_F\Gamma^{-1/2}\big)<\frac{1}{1+\sqrt5},
		\end{equation}
		where $\lambda_{\max}$ denotes the largest eigenvalue (that is, the covariance of $F(Y)$ is sufficiently small relative to $\Gamma$).
			
		\item
		more generally, $F(Y)$ is sub-Gaussian with parameter matrix $\Sigma_F$
		satisfying
		\eqref{equ:F_Y_covariance_small_relative_to_Gamma}.
		
		\item
		more generally still, $F(Y)$ has sufficiently light tails so that
		$\norm{F(Y)}_\Gamma^2$ admits an exponential moment of order strictly larger than $\frac{1+\sqrt5}{2}$.
		
	\end{enumerate}
\end{remark}

\subsection{Koopman Operator of Stochastic Differential Equations} \label{sec:koopman_sde_class}

In this subsection, we verify the density regularity assumption $\eta_t \in L^2(\nu_{\cY};\cH)$ (\Cref{assump:general_Radon_Nikodym_setting}) for uniformly elliptic stochastic differential equations. Throughout, we take $\cY = \bR^d$ and the reference measure $\nu_{\cY} = \bslambda$ to be the $d$-dimensional Lebesgue measure. 

From the perspective of conditional expectation operators, stochastic differential equations provide a dynamical counterpart to the regression setting of \Cref{sec:regression}. While regression studies a single conditional expectation $x \mapsto \bE[Y \, | \, X = x]$, stochastic dynamics give rise to an entire family of conditional expectation operators describing the evolution of observables over time, namely the Koopman semigroup $(K_t)_{t \geq 0}$. Stochastic differential equations therefore constitute one of the principal applications for Koopman analysis of stochastic dynamical systems \citep{CrnjaricMacesivMesic:koopman_spectrum_random_ds} with broad application areas such as molecular dynamics \citep{NateghiNuske2025:coarse_graining_kedmd}, climate science \citep{SantosLucarini2020:Koopman_climate}, and quantum systems \citep{KlusNuskePeitz2022:Koopman_quantum}, among many others. Kernel-based numerical methods such as kernel EDMD and related covariance operator representations have been widely studied for the approximation of Koopman operators \citep{WillRowl15, KlusKoltaiSchutte2016:num_approx_koopman, KlusNuskeHamzi2020:KoopmanGeneratorEDMD, Klus2020Eigendecompositions}, while more recent works develop statistical learning approaches in RKHSs \citep{Kostic2022:KooopmanRegressionRKHS, Hou2026:SparseOnlineKoopmanLearning} and  approximation error bounds \citep{PhilScha25:kernel:JNLS, hertel2025koopmanstochasticdynamicserror}. The regularity theory developed in this work provides a unified analytical foundation for these results by relating the smoothing properties of stochastic Koopman operators directly to regularity of the underlying transition densities.

Consequently, our key analytical ingredient is the existence of such smooth transition densities satisfying quantitative derivative estimates with respect to the initial (or \emph{backward}) variable $x$. These bounds allow us to establish the density regularity condition
for suitable Sobolev-type Hilbert spaces $\cH$. For this, we consider the time-homogeneous Itô diffusion 
\begin{equation} \label{eq:basic_sde_time_homogeneous}
    \rd X_t = b(X_t)\, \rd t + \sigma(X_t)\, \rd W_t, \quad t > 0,
    \qquad
    X_0 = x \in \cX \subseteq \bR^{d},
\end{equation}
where $(W_t)_{t \geq 0}$ is a $d$-dimensional Brownian motion, $b\colon \bR^{d} \to \bR^d$ denotes the drift, and $\sigma \colon \bR^{d} \to \bR^{d \times d}$ the diffusion coefficient.
Under standard Lipschitz and growth assumptions, \eqref{eq:basic_sde_time_homogeneous} admits a unique strong solution $(X_{t})_{t \geq 0}$ \citep[Theorem 5.2.1]{Oksendal2003SDE}.
Its transition densities $(\eta_t)_{t > 0}$ induce the family of
\emph{stochastic Koopman operators}
\begin{align}
\label{equ:Koopman_definition_and_integral_form}
    K_tg(x) \coloneqq \bE[g(X_t) \, | \, X_0 = x] = \int_\cY g(y) \, \eta_t(x,y) \, \nu_\cY(\rd y)
\end{align}
for bounded, measurable $g: \bR^d \to \bR$, which coincide precisely with the CEO \eqref{equ:generic_CEO} for $X = X_{0}$ and $Y = X_{t}$.  

The family $(K_t)_{t \geq 0}$ forms a Markov semigroup associated with the diffusion \eqref{eq:basic_sde_time_homogeneous}. Moreover, the transition densities $(\eta_t)_{t > 0}$ correspond to the fundamental solution of the parabolic operator associated with the SDE. Under the ellipticity and smoothness assumptions below, classical parabolic regularity theory yields Gaussian-type bounds on the derivatives of $\eta_t$ w.r.t.\ $x$. Integrating these bounds in the forward variable $y$ implies the required Sobolev regularity of $\eta_t$ and thereby verifies \Cref{assump:general_Radon_Nikodym_setting}. 
For the remainder of this section, we fix the following assumptions on the coefficients of the SDE \eqref{eq:basic_sde_time_homogeneous}. 

\begin{assumption}[Regular Itô diffusion with uniformly elliptic noise]
\label{ass:sde_unif_elliptic_regular}
We consider the SDE \eqref{eq:basic_sde_time_homogeneous} and assume:
\begin{enumerate}[label=(\roman*)]
    \item
    \label{item:sde_unif_ellipt_final}
    \textbf{Uniform ellipticity:} There exists $\kappa > 0$ such that $\kappa \, \|\xi\|_2^2 \le \norm{\sigma(x)^\top \xi}_{2}^{2}$ for all $ x, \xi \in \bR^d$, or equivalently, $a(x) = \sigma(x) \sigma(x)^\top$ is uniformly positive definite. 

    \item
    \label{item:sde_coeff_regularity}
    \textbf{Coefficient regularity:} For some $\ell \in \bN_{\geq 1}$, the SDE coefficients satisfy $b \in C_b^{\ell}(\bR^d, \bR^d)$ and $\sigma \in C_b^{\ell}(\bR^d, \bR^{d\times d})$,
    where $C_b^\ell(\cS, \cT)$ denotes the space of $\ell$-times continuously differentiable functions $f: \cS \to \cT$ whose derivatives up to order $\ell$ are bounded. 
\end{enumerate}
\end{assumption}

While not necessarily minimal (see \Cref{rem:extensions_unified}), \Cref{ass:sde_unif_elliptic_regular} provides a convenient regime in which classical parabolic regularity theory yields the required density regularity condition $\eta_t \in L^2(\nu_{\cY}, H^\ell(\cX) )$.
For this, we first gather the following properties, which are standard results in both the partial differential equations (PDE) and stochastic analysis literature \citep{Frie1964, StroVara1997}.

\begin{proposition}[SDE density regularity and derivative bounds]\label{prop:sde_density_regularity_bounds}
Let \Cref{ass:sde_unif_elliptic_regular} hold and fix $T > 0$. Then the SDE \eqref{eq:basic_sde_time_homogeneous} admits a unique strong solution $(X_t)_{t \in [0,T]}$. Moreover, for every $t \in (0,T]$, the conditional law $P_{X_t|X_0=x}$ is absolutely continuous w.r.t.\ $\nu_\cY = \bslambda$ with transition density $\eta_t$ satisfying \eqref{equ:Koopman_definition_and_integral_form} for all bounded measurable $g: \bR^d \to \bR$. 

Further, for every multi-index $|\alpha| \le \ell$, the map $x \mapsto \eta_t(x,y)$ admits weak partial derivatives $\partial_x^\alpha \eta_t$, and there exist constants $C_\alpha, \kappa > 0$ independent of $t, x, y$ such that
    \[
    |\partial_x^\alpha \eta_t(x,y)|
    \leq
    C_\alpha t^{-(|\alpha|+d)/2} \exp\left( -\kappa \, \frac{\|x-y\|_2^2}{t} \right).
    \]
\end{proposition}

\begin{proof}
    The proof is given in \Cref{sec:proofs_sde}.
\end{proof}

We note that the classical results from \citet{Frie1964} were originally formulated for SDEs with time-dependent coefficients $b=b(t,x)$ and $\sigma=\sigma(t,x)$. Hence, the arguments in the proof of \Cref{prop:sde_density_regularity_bounds} extend to this case, see~\Cref{sec:technical_details_proofs}. We restrict our presentation to time-homogeneous diffusions because the associated CEOs then form a Koopman semigroup $(K_t)_{t \geq 0}$ satisfying $K_{t+s} = K_t K_s$, which is fundamental for prediction by iterative application of the learned operator approximation over one time interval. 

The Gaussian derivative bounds of \Cref{prop:sde_density_regularity_bounds} are sufficient to verify the density regularity condition
$\eta_t \in L^2(\nu_{\cY}; \cH)$
from \Cref{assump:general_Radon_Nikodym_setting}. The remaining issue is the choice of the Hilbert space $\cH$ on the conditioning domain $\cX$. As the following theorem shows, this naturally splits into two cases, cf.\ \Cref{rem:rate_bounded_vs_unbounded}:
\begin{enumerate}[label=(\Alph*)]
    \item If $\cX \subset \bR^d$ is a bounded Lipschitz domain, one may take $\cH \simeq H^\ell(\cX)$.
    \item If $\cX = \bR^d$, one instead works with a reweighted Sobolev-type RKHS $\cH = \cH_{k_\tau}$, where $k_{\tau}(x,x') = m_\tau(x)^{-1}\, k_0(x,x')\, m_\tau(x')^{-1}$ with $\cH_{k_{0}} \simeq H^\ell(\cX)$ and weight $m_\tau(x) = \exp\big(-\tfrac{\tau}{2}\norm{x}_{\Gamma}^{2}\big)$, cf.~\eqref{equation:Gaussian_reweighted_kernel}.
\end{enumerate}

\begin{theorem}[Regularity of Koopman operators for uniformly elliptic SDEs]
\label{thm:sde_density_koopman_unified}
    Let $\cX \subseteq \bR^{d}$ with $\nu_{\cX} \coloneqq \bslambda$,
    let $(\cY,\nu_{\cY}) = (\bR^{d},\bslambda)$,
    and let \Cref{ass:sde_unif_elliptic_regular} hold with $\ell \in \bN_{\geq 1}$.
    Let $\eta_t$ denote the transition density and $K_{t}$ the Koopman operator \eqref{equ:Koopman_definition_and_integral_form} associated with the SDE \eqref{eq:basic_sde_time_homogeneous} at time $t > 0$.    
    Then
    \[
    \eta_{t} \in L^2(\nu_\cY;\cH),
    \qquad
	K_{t} \in \HS(L^2(\nu_\cY),\cH)
    \quad \text{with} \quad
    \norm{K_{t}}_{\HS(L^2(\nu_\cY),\cH)} = \norm{\eta_{t}}_{L^2(\nu_\cY;\cH)},
    \]
    if either of the following holds:
    \begin{enumerate}[label=(\Alph*)]
        \item
        \label{item:thm_sde_koopman_bounded}
        \textbf{bounded domain, Sobolev RKHS:}
        $\cX \subset \bR^d$ is a bounded Lipschitz domain and $\cH \simeq H^\ell(\cX)$. In this case we additionally have
        \[
        K_{t} \in \HS( L^2(\nu_{\cY}) , L^2(\nu_{\cX}) )
        \qquad \text{ and }\qquad        
        K_{t} \in \Tr( L^2(\nu_{\cY}) , L^2(\nu_{\cX}) )
        \ \text{ if }\ \ell > d/2;
        \]
        
        \item
        \label{item:thm_sde_koopman_unbounded}
        \textbf{unbounded domain, reweighted Sobolev RKHS:}
        $\cX = \bR^d$, $k_0$ is a kernel on
    	$\bR^d$ with $\cH_{k_0} \simeq H^\ell(\bR^d)$, and $\cH = \cH_{k_{\cX}}$ where $k_{\cX} = k_\tau$ is the Gaussian-reweighted kernel
    	\eqref{equation:Gaussian_reweighted_kernel} for $\tau > 0$ and symmetric, strictly positive definite $\Gamma \in \bR^{d\times d}$.
    \end{enumerate}
\end{theorem}

\begin{proof}
    The proof is given in \Cref{sec:proofs_sde}. 
\end{proof}

By \Cref{thm:sde_density_koopman_unified}, under sufficiently strong regularity assumptions on the coefficients of the SDE, the abstract results of \Cref{sec:consequences_density_regularity} immediately yield Hilbert--Schmidt properties of the associated stochastic Koopman operators. The only distinction between the bounded and unbounded settings lies in the choice of the target Hilbert space: bounded domains naturally lead to Sobolev spaces, whereas on $\bR^d$ exponential reweighting compensates for the lack of spatial integrability. In the former case, the results of \Cref{sec:error_bound_implications} immediately become available.

\begin{corollary}[Error rates for Galerkin-type approximations of SDE Koopman operators]
\label{cor:sde_koopman_one_sided_error_rate}
Let $\cX \subset \bR^d$ be Lipschitz bounded, let \Cref{ass:sde_unif_elliptic_regular} hold with $\bN \ni \ell > d/2$ and let $\cH \simeq H^{\ell}(\cX)$ be an RKHS induced by (cf.\ \Cref{prop:matern_wendland_sobolev})
\begin{enumerate}[label = (\roman*)]
\item 
the Mat\'ern kernel with parameter $\beta = \ell - d/2$, \textcolor{darkgreen}{where we assume $\ell > d/2 + 1$} and set $m = \lceil \beta \rceil - 1$,
or
\item 
the Wendland kernel with parameter $j = \ell - (d+1)/2$, where we assume $d$ to be odd and $d \geq 3$ if $j = 0$, and set $m = j + 1/2$.
\end{enumerate}
Using \Cref{notation:CEO_approximations}, we obtain for $X = X_0$, $Y = X_t$ and a data set $\bsX = \{x_i\}_{i=1}^n$,
\begin{align}
\label{eq:sde_one_sided_proj_fd_bound}
    \norm{K_t - \cP_\bsX K_t}_{\cL(L^2(\nu_\cY), C_b(\cX))} \leq C \, \norm{\eta_{t}}_{L^2(\nu_{\cY};\cH)} \, \fd_{\bsX}^{m}.
\end{align}  
\end{corollary}

\begin{proof}
    The proof is given in \Cref{sec:proofs_sde}.
\end{proof}

\Cref{cor:sde_koopman_one_sided_error_rate} distinguishes between Mat\'ern kernels, which induce $\cH \simeq H^\ell(\cX)$ for arbitrary smoothness levels $\ell > d/2$, and Wendland kernels, which can only realize Sobolev spaces of integer and half-integer smoothness. In turn, whenever applicable, Wendland kernels yield a slightly sharper fill distance rate. 

The analysis in this subsection focuses on uniformly elliptic SDEs (cf.\ \Cref{ass:sde_unif_elliptic_regular}) with Lipschitz bounded $\cX$ and $\cY = \bR^d$, leading naturally to the one-sided Galerkin approximation considered in \Cref{cor:sde_koopman_one_sided_error_rate}.
This raises the question of how the CEO regularity theory developed in this work extends to broader classes of stochastic dynamics and to additional approximation schemes.

\begin{remark}[Extensions of CEO regularity and approximation results for SDEs] \label{rem:extensions_unified}
    Several additional results from \Cref{sec:CME_and_abstract_results,sec:error_bound_implications} may admit analogous applications in the SDE setting.
    \begin{itemize}
        \item \textbf{Error rates for two-sided Galerkin-type approximations}. The two-sided projection estimates from \Cref{thm:err_bounds_proj_unified} are not directly applicable in our present SDE setting, since they require the state space $\cY$ to be a bounded Lipschitz domain to obtain a fill distance bound on the domain of the Koopman operator. For uniformly elliptic SDEs on $\bR^d$, however, the transition law typically has unbounded support. A possible remedy could be to extend the present analysis, in particular the transition density regularity analysis, to diffusions confined to bounded domains, e.g., by considering reflected diffusions or suitably modified coefficients. 
        \item
        \textbf{CME representation and approximation bounds.} For any RKHS $\cH_{k_\cY} \hookrightarrow L^2(\bslambda)$, in both cases of \Cref{thm:sde_density_koopman_unified} the assumptions required for the covariance-operator representation of conditional mean embeddings (\Cref{cor:regularity_conditional_expectation_operator_CME}) are satisfied. 
        Consequently, the corresponding CME-based approximation and convergence results from \Cref{subsubsec:cme_approx} apply equally to stochastic Koopman operators associated with uniformly elliptic SDEs, provided the assumptions for the CME convergence rate (\Cref{prop:cme_learning_rate}) are met. However, the latter requires a bounded kernel, which only holds in case \ref{item:thm_sde_koopman_bounded} but fails in case \ref{item:thm_sde_koopman_unbounded} of \Cref{thm:sde_density_koopman_unified}, cf.\ \Cref{rem:rate_bounded_vs_unbounded}.
        A detailed investigation of the CME-based approximation results lies beyond the scope of the present work.
    \end{itemize}
    Beyond the presented uniformly elliptic setting, relaxed or modified assumptions on the SDE coefficients might provide notable extensions to the CEO regularity and approximation results of this subsection.
    \begin{itemize}
        \item \textbf{Relaxing uniform ellipticity.} Smooth transition densities can also be obtained under substantially weaker non-degeneracy conditions, most notably Hörmander's bracket condition via Malliavin calculus \citep{Mall1978:stoch_calc_hypoelliptic, KusuokaStroock1985:malliavin_part_ii,Hair2011}, or through differentiability properties of the associated stochastic flow \citep{Stro1981:malliavin_part_1,Kuni1990}. In the case of state space $\cY = \bR^d$, the key question then becomes the availability of quantitative (e.g., Gaussian-type) bounds for spatial derivatives as in \Cref{prop:sde_density_regularity_bounds}.  
        \item \textbf{Relaxing regularity and boundedness of coefficients.} Gaussian bounds for transition densities of uniform elliptic diffusions were originally proven for bounded measurable coefficients \citep{Aron1967} and have recently been extended to settings with linearly growing, unbounded drift by \citet{MenoPescZhan2021:density_sde_estimates}, while still preserving on uniform ellipticity. Moreover, under suitable Hölder regularity assumptions on the coefficients, the latter work also derives Gaussian estimates for arbitrary spatial derivatives of the transition density. 
        Extending the Sobolev regularity analysis developed here to this broader setting would substantially enlarge the class of stochastic Koopman operators covered by our framework. In particular, it would cover benchmark models such as the Ornstein--Uhlenbeck process, whose drift is linear and therefore falls outside \Cref{ass:sde_unif_elliptic_regular}, despite the process admitting an explicit Gaussian density satisfying the required derivative estimates, see~\Cref{ex:ou_sde_density_hl}.
    \end{itemize}
\end{remark}

\begin{example}[Ornstein--Uhlenbeck process: Explicit density under unbounded drift]
\label{ex:ou_sde_density_hl}

On a Lipschitz bounded initial domain $\cX \subset \bR^d$ and $\cY = \bR^d$, we consider an Ornstein--Uhlenbeck process, i.e., the solution to the linear SDE
\begin{align}\label{eq:ornstein_uhlenbeck_sde}
   \rd X_t = A X_t \,\rd t + B \, \rd W_t,
\end{align}
with constant matrices $A, B \in \bR^{d \times d}$. Since the drift function $b: \bR^d \to \bR^d, b(x) = Ax$ is clearly unbounded, we cannot directly apply \Cref{thm:sde_density_koopman_unified}\ref{item:thm_sde_koopman_bounded}. However, the solution process is Gaussian and satisfies $(X_t  \, | \, X_0 = x) \sim \cN(e^{tA}x, Q_t)$, where $Q_t = \int_0^t e^{sA} B B^\top e^{s A^\top} \rd s$, see, e.g., \citep[Corollary 1]{Vati2019}. Assume that $B B^\top > 0$, such that $Q_t$ is positive definite (and the uniform ellipticity condition \Cref{ass:sde_unif_elliptic_regular}\ref{item:sde_unif_ellipt_final} is satisfied) and we obtain the explicit transition density $\eta_t(x,y) = (2\pi)^{-d/2}\det(Q_t)^{-1/2} \exp\left( -\frac{1}{2}(y - e^{tA}x)^\top Q_t^{-1}(y - e^{tA}x) \right)$. 
Hence, the Koopman operator $K_t g(x) = \bE[g(X_t) \, | \, X_0 = x] = \int_{\bR^d} g(y) \, \eta_t(x,y) \, \rd y$ is well-defined and bounded.

We now verify the Sobolev-type integrability directly as in the proof of \Cref{thm:sde_density_koopman_unified}. Since the exponent is quadratic in $(x,y)$, every $x$-derivative yields a polynomial prefactor, i.e., $\partial_x^\alpha \eta_t(x,y) = f_\alpha(x,y)\,\eta_t(x,y)$, where $f_\alpha$ is a polynomial of degree at most $|\alpha|$. Consequently,
\begin{align*}
    |\partial_x^\alpha \eta_t(x,y)|^2
    \le C_\alpha \big(1 + \|x\|^{2|\alpha|} + \|y\|^{2|\alpha|}\big)\,\eta_t(x,y)^2 \leq C \big(1 + \|x\|^{2|\alpha|} + \|y\|^{2|\alpha|}\big) e^{-c\|y - e^{tA}x\|^2}
\end{align*}
with suitable constants $C_\alpha, C, c > 0$, where we used the explicit Gaussian form of $\eta_t$. Hence, for fixed $x$ and with the change of variables $z=y-e^{tA}x$,
\begin{align*}
    \int_{\bR^d} |\partial_x^\alpha \eta_t(x,y)|^2\, \rd y
    &\le C \int_{\bR^d}
    \big(1 + \|x\|^{2|\alpha|} + \|y\|^{2|\alpha|}\big)
    e^{-c\|y - e^{tA}x\|^2} \rd y \\
    &=
    C \int_{\bR^d}
    \big(1 + \|x\|^{2|\alpha|} + \|z + e^{tA}x\|^{2|\alpha|}\big)
    e^{-c\|z\|^2} \rd z \le \widetilde{C} (1 + \|x\|^{2|\alpha|})
\end{align*}
with $\widetilde{C} >0$, where we used finiteness of Gaussian moments. Finally, integrating over the bounded domain $\cX$ yields $\int_{\cX} \int_{\bR^d} |\partial_x^\alpha \eta_T(x,y)|^2\,dy\,dx < \infty$ for all $|\alpha|\le \ell$. This proves $\eta_t \in L^2(\nu_{\cY}; H^\ell(\cX))$ for $\nu_\cY = \bslambda$ and arbitrary $\ell \in \bN$. In particular, the $C^\infty$-density means that the Koopman operator $K_t$ associated with the SDE is `arbitrarily smoothing'.
\end{example}

\section{Conclusions}\label{sec:conclusions}

In this work, we traced the regularizing effect of conditional expectation operators (CEOs) back to the regularity of the associated conditional densities. Our main result establishes a simple, verifiable regularity criterion: if the conditional density belongs to a suitable Bochner space $L^2(\nu_{\cY};\cH)$, then the corresponding CEO defines a bounded, and in fact Hilbert--Schmidt, operator from $L^2(\nu_{\cY})$ into $\cH$. This provides a unified perspective on regularity, compactness, and smoothing properties of CEOs.

Beyond its operator-theoretic interest, this criterion has two immediate consequences. On the one hand, it provides a verifiable sufficient condition for the representability assumptions underlying conditional mean embeddings. Rather than assuming a priori that conditional expectations take values in a prescribed RKHS, we show that this property follows from regularity of the conditional density, justifying operator-based CME representations through analytical properties of the underlying probabilistic model.
On the other hand, the regularity criterion yields quantitative approximation results for CEOs. In particular, it provides error bounds for Galerkin-type and CME-based approximations whose convergence rates are determined by the smoothing properties of the underlying conditional density. Since stochastic Koopman operators are themselves CEOs, these results apply directly in the setting of stochastic dynamical systems.

To demonstrate the practical verifiability of the proposed abstract regularity criterion, we established it in three representative applications. In nonparametric regression, smoothness assumptions on the regression function and noise model imply the required density regularity and thereby justify the operator-based CME representation. For Bayesian inverse problems, we derived sufficient conditions ensuring that posterior densities induce regularization into weighted Sobolev-type RKHSs. For uniformly elliptic stochastic differential equations, classical transition density regularity results and derivative bounds yield smoothing properties of the associated stochastic Koopman operators.

Overall, the results show that operator-theoretic properties of CEOs can be reduced to verifiable regularity properties of the associated conditional densities. By replacing abstract RKHS representability assumptions with analytical conditions that can be checked in concrete probabilistic models, the proposed framework unifies conditional mean embeddings and stochastic Koopman theory across many application domains and provides a natural bridge between probability theory, operator theory, and numerical approximation.

\bibliographystyle{abbrvnat}
\bibliography{bibliography_arxiv}
\addcontentsline{toc}{section}{References}

\appendix

\section{Proofs and Auxiliary Results}
\label{sec:technical_details_proofs}

In this section, we state the detailed proofs of our main results. Therein, we will make use of the following basic inequalities for arbitrary $a,b,\gamma > 0$:
\begin{align}
\label{equ:technical_inequality_splitting_1}
(a+b)^{2} = a^{2} + 2 \gamma^{1/2} a \gamma^{-1/2} b + b^{2}
&\leq
a^{2} + \gamma a^{2} + \gamma^{-1} b^{2} + b^{2}
=
(1+\gamma) a^{2} + (1 + \gamma^{-1}) b^{2},
\\
\label{equ:technical_inequality_splitting_2}
2 a b + b^2
&\leq
\gamma a^{2} + (1 + \gamma^{-1}) b^{2},
\end{align}
where the second one follows directly from the first.

\begin{proof}[Proof of \Cref{prop:rkhs_in_L2_via_spectral_density}]
    By Bochner's theorem \citep[Theorem~6.6]{Wendland2005Scattered}, $\widehat\psi \geq 0$ is nonnegative.
By \citep[Theorem~10.12]{Wendland2005Scattered},
\[
\cH_k
=
\left\{
f\in C(\bR^d)\cap L^2(\bR^d)
\;\middle|\;
\int_{\bR^d}\frac{|\widehat f(\omega)|^2}{\widehat\psi(\omega)}\, \rd \omega<\infty
\right\},
\qquad
\|f\|_{\cH_k}^2
=
(2\pi)^{-d/2}
\int_{\bR^d}\frac{|\widehat f(\omega)|^2}{\widehat\psi(\omega)}\, \rd \omega .
\]
Hence, using Plancherel's theorem and $\widehat\psi\in L^\infty(\bR^d)$,
\[
\|f\|_{L^2(\bR^d)}^2
=
\|\widehat f\|_{L^2(\bR^d)}^2
=
\int_{\bR^d}
\widehat\psi(\omega)\,
\frac{|\widehat f(\omega)|^2}{\widehat\psi(\omega)}\, \rd \omega
\le
(2\pi)^{d/2}\|\widehat\psi\|_{L^\infty(\bR^d)}\,\|f\|_{\cH_k}^2 .
\]
This proves the continuous embedding $\cH_k \hookrightarrow L^2(\bR^d)$.
For the domain version, let $g\in\cH_k(\cX)$.
By formula \eqref{equ:Wendland_restriction_RKHS} for the restriction-space norm, for every $\varepsilon>0$
there exists $f\in\cH_k(\bR^d)$ such that $f|_{\cX}=g$ and
\[
\|f\|_{\cH_k(\bR^d)} \le \|g\|_{\cH_k(\cX)}+\varepsilon .
\]
Since $\|g\|_{L^2(\cX)}\le \|f\|_{L^2(\bR^d)}$, the embedding on $\bR^d$
yields
\[
\|g\|_{L^2(\cX)}
\le
\|f\|_{L^2(\bR^d)}
\le
(2\pi)^{d/4}\|\widehat\psi\|_{L^\infty(\bR^d)}^{1/2}
\,\|f\|_{\cH_k(\bR^d)}
\leq
(2\pi)^{d/4}\|\widehat\psi\|_{L^\infty(\bR^d)}^{1/2}
\, \big( \|g\|_{\cH_k(\cX)}+\varepsilon \big).
\]
Letting $\varepsilon\to0$ proves the claim.
\end{proof}

\subsection{\texorpdfstring{Proofs of \Cref{sec:regression}}{Proofs of Section 4.1}}\label{sec:proofs_regression}

\begin{proof}[Proof of \Cref{thm:kRR_well_specified_homoskedastic}]
	Throughout, $C > 0$ denotes a constant depending only on $d$, $\ell$,
	$\tau$, $\Gamma$, and (in case \ref{item:kRR_homosked_bounded}) on $\cX$,
	and may change from line to line.
	Since $\varepsilon$ is independent of $X$,
    the conditional density is given by
	$\eta(x,y) = p_\varepsilon\big(y - f(x)\big)$.

    \medskip

    \noindent    
	\textbf{Step 1: Derivative structure via Fa\`a di Bruno.}	
	By the multivariate Fa\`a di Bruno formula
	\citep[Theorem~2.1]{ConstantineSavits1996}, for $1 \le |\alpha| \le \ell$,
    \begin{equation}
	\label{equ:FdB_homoskedastic}
	\partial_x^{\alpha} \eta(x,y)
	=
	\sum_{s=1}^{|\alpha|} (-1)^{s}\,
	p_\varepsilon^{(s)}\big(y - f(x)\big)
    G_{\alpha,s}(x),
    \qquad
    G_{\alpha,s}(x)
    \coloneqq
	\sum_{\substack{\beta_1 + \dots + \beta_s = \alpha \\ |\beta_j| \ge 1}}
	c_{\bsbeta}
	\prod_{j=1}^{s} \partial^{\beta_j} f(x),
    \end{equation}
	with combinatorial constants $c_{\bsbeta} \ge 0$, $\bsbeta = (\beta_1,\dots,\beta_s)$.
	(Here \eqref{equ:FdB_homoskedastic} is understood as an identity of weak derivatives on $\cX$, valid for $f \in H^\ell(\cX)$ by approximation with smooth functions; since $\ell > d/2$, the products $\prod_j \partial^{\beta_j} f$ lie in $L^2(\cX)$, see \eqref{equ:Moser_homoskedastic} below.)

	\medskip

    \noindent
	\textbf{Step 2: Integration in $y$.}
	Fix $x \in \cX$. By translation invariance of the Lebesgue measure,
	\[
	\int_\bR
	\absval{p_\varepsilon^{(s)}\big(y - f(x)\big)}^2\,\rd y
	=
	\norm{p_\varepsilon^{(s)}}_{L^2(\bR)}^2
	\le
	\norm{p_\varepsilon}_{W^{\ell,2}(\bR)}^2,
	\qquad
	0 \le s \le \ell.
	\]
	Hence, applying the Cauchy--Schwarz inequality to the finite sum
	\eqref{equ:FdB_homoskedastic},
	\begin{equation}
	\label{equ:y_int_homoskedastic}
	\int_\bR \absval{\partial_x^{\alpha}\eta(x,y)}^2\,\rd y
	\le
	C \, \norm{p_\varepsilon}_{W^{\ell,2}(\bR)}^2\,
	G(x)^2,
	\qquad
	G(x)
	\coloneqq
    \max_{1 \le |\alpha| \le \ell}
	\sum_{s=1}^{|\alpha|}
    \absval{G_{\alpha,s}(x)}
	\end{equation}
	for $1 \le |\alpha| \le \ell$, while for $\alpha = 0$ the left-hand side
	equals $\norm{p_\varepsilon}_{L^2(\bR)}^2$.

	\medskip

    \noindent
	\textbf{Step 3: Integration in $x$.}
	For each tuple $(\beta_1,\dots,\beta_s)$ in
	\eqref{equ:y_int_homoskedastic} we have
	$\sum_j |\beta_j| = |\alpha| \le \ell$,
    so that the Gagliardo--Nirenberg--Moser product estimate
	\citep[see][Ch.~13, \S3]{Taylor2011PDE3}
	yields, using $f \in H^\ell(\cX) \hookrightarrow L^\infty(\cX)$
	for $\ell > d/2$,
	\begin{equation}
	\label{equ:Moser_homoskedastic}
	\bigg\| \prod_{j=1}^{s} \partial^{\beta_j} f \bigg\|_{L^2(\bR^d)}
	\le
	C \norm{f}_{L^\infty(\bR^d)}^{s-1}\norm{f}_{H^{\ell}(\bR^d)}
	\le
	C \norm{f}_{H^{\ell}(\bR^d)}^{s},
	\end{equation}
	and therefore $\norm{G}_{L^2(\cX)} \le
	C\big(1 + \norm{f}_{H^\ell(\cX)}\big)^{\ell}$ is finite, where for bounded Lipschitz domains $\cX$ we applied
	\eqref{equ:Moser_homoskedastic} to a Stein extension of $f$
	\citep[Theorem~5.24]{Adams2003}.

	\emph{Case \ref{item:kRR_homosked_bounded}.}	
	By Tonelli's theorem and \eqref{equ:y_int_homoskedastic},
	\[
	\norm{\eta}_{L^2(\nu_\cY;H^\ell(\cX))}^2
	=
	\sum_{|\alpha| \le \ell}
	\int_{\cX}\int_\bR
	\absval{\partial_x^{\alpha}\eta(x,y)}^2\,\rd y\,\rd x
	\le
	\bslambda(\cX)\norm{p_\varepsilon}_{L^2(\bR)}^2
	+
	C \, \norm{p_\varepsilon}_{W^{\ell,2}(\bR)}^2
	\norm{G}_{L^2(\cX)}^{2}
    <
    \infty,
	\]    
	where the first summand is the $\alpha = 0$ term.
	By norm-equivalence $\cH \simeq H^\ell(\cX)$,
	\Cref{assump:general_Radon_Nikodym_setting} holds, and the claim follows from
	\Cref{thm:CEO_bounds_various_spaces}\ref{item:CEO_H_HS} and~\ref{item:CEO_L2_Tr}.

	\emph{Case \ref{item:kRR_homosked_Rd}.}
	By \Cref{lemma:reweighted_kernel} and $\cH_{k_0} \simeq H^\ell(\bR^d)$,
	\[
	\norm{\eta(\quark,y)}_{\cH_{k_\tau}}
	=
	\norm{m_\tau\,\eta(\quark,y)}_{\cH_{k_0}}
	\asymp
	\norm{m_\tau\,\eta(\quark,y)}_{H^\ell(\bR^d)},
	\qquad
	m_\tau(x) = \exp\big(-\tfrac{\tau}{2}\norm{x}_\Gamma^2\big).
	\]
	Since every derivative of $m_\tau$ is a polynomial times the same
	Gaussian weight, $\absval{\partial^{\gamma} m_\tau(x)}
	\le C (1+\norm{x}_2)^{|\gamma|}\, m_\tau(x)$ for $|\gamma| \le \ell$.
	Hence, by the Leibniz rule and \eqref{equ:y_int_homoskedastic},
	for every $|\alpha| \le \ell$,
	\[
	\int_\bR
	\absval{\partial_x^{\alpha}\big(m_\tau\,\eta\big)(x,y)}^2\,\rd y
	\le
	C \norm{p_\varepsilon}_{W^{\ell,2}(\bR)}^2\,
	(1+\norm{x}_2)^{2\ell}\, m_\tau(x)^2\,
	\big(1 + G(x)\big)^2.
	\]
	Since $G \in L^2(\bR^d)$ by \eqref{equ:Moser_homoskedastic} and $w_\tau(x) \coloneqq (1+\norm{x}_2)^{2\ell} \, m_\tau(x)^2$ satisfies
	$w_\tau \in L^1(\bR^d) \cap L^\infty(\bR^d)$, we conclude with Tonelli's
	theorem that
	\[
	\norm{\eta}_{L^2(\nu_\cY;\cH_{k_\tau})}^2
	\le
	C \, \norm{p_\varepsilon}_{W^{\ell,2}(\bR)}^2
	\int_{\bR^d} w_\tau(x)\big(1+G(x)\big)^2\,\rd x
	\le
	C \, \norm{p_\varepsilon}_{W^{\ell,2}(\bR)}^2
	\big( \norm{w_\tau}_{L^1} + \norm{w_\tau}_{L^\infty}\norm{G}_{L^2}^2 \big)
	\]
    is finite.
	\Cref{assump:general_Radon_Nikodym_setting} thus holds with
	$\cH = \cH_{k_\tau}$, and the claim again follows from \Cref{thm:CEO_bounds_various_spaces}\ref{item:CEO_H_HS}.
\end{proof}

\begin{proof}[Proof of \Cref{cor:kRR_CME_representation}]
	We first verify that $(X,k_\cX)$ satisfies \Cref{ass:kernel_rv_pair};
	condition \ref{item:basic_assumptions_ae_separation} holds by assumption
	in both cases, so it remains to check
	\ref{item:basic_assumptions_kernel} and
	\ref{item:basic_assumptions_feature_map_L2}.

	\emph{Case \ref{item:kRR_homosked_bounded}.}
	Since $k_\cX$ is measurable, symmetric, and positive definite as a
	kernel, and $\cH_{k_\cX} \simeq H^\ell(\cX)$ is separable,
	\ref{item:basic_assumptions_kernel} holds.
	For \ref{item:basic_assumptions_feature_map_L2}, note that since
	$\ell > d/2$, the Sobolev embedding theorem gives
	$H^\ell(\cX) \hookrightarrow C_b(\cX)$ continuously
	\citep[Theorem~4.12]{Adams2003}; composed with the norm-equivalence
	$\cH_{k_\cX} \simeq H^\ell(\cX)$, this yields a bounded embedding
	$\cH_{k_\cX} \hookrightarrow C_b(\cX)$ with some operator norm $C > 0$.
	Evaluating at the feature vector $\phi_\cX(x) = k_\cX(x,\quark)$ and
	using the reproducing property gives
	\[
	k_\cX(x,x)
	\leq
	\norm{\phi_\cX(x)}_{\infty}
	\leq
	C \norm{\phi_\cX(x)}_{\cH_{k_\cX}}
	=
	C \sqrt{k_\cX(x,x)},
	\]
	hence $\bE[k_\cX(X,X)] \leq \sup_{x \in \cX} k_\cX(x,x) \leq C^2 < \infty$ regardless of the law of $X$,
	proving \ref{item:basic_assumptions_feature_map_L2}.

	\emph{Case \ref{item:kRR_homosked_Rd}.}
	Since $k_0$ is continuous and
	$m_\tau(x) = \exp(-\tfrac{\tau}{2}\norm{x}_\Gamma^2) > 0$ is continuous,
	$k_\tau$ is continuous, hence measurable. Symmetry and positive
	definiteness of $k_\tau$ follow from \Cref{lemma:reweighted_kernel},
	which also shows that the map $f \mapsto m_\tau f$ is an isometry from
	$\cH_{k_\tau}$ onto $\cH_{k_0} \simeq H^\ell(\bR^d)$; since
	$H^\ell(\bR^d)$ is separable, so is $\cH_{k_\tau}$, proving
	\ref{item:basic_assumptions_kernel}.
	For \ref{item:basic_assumptions_feature_map_L2}, translation invariance
	of $k_0$ implies that $c_0 \coloneqq k_0(x,x)$ is constant, so that, by \eqref{equ:regression_exp_moment_X},
	\[
	\bE\big[k_\cX(X,X)\big]
	=
	\bE\big[m_\tau(X)^{-2}\, k_0(X,X)\big]
	=
	c_0\, \bE\big[\exp\big(\tau\norm{X}_\Gamma^2\big)\big]
	< \infty.
	\]
	In both cases, \Cref{assump:general_Radon_Nikodym_setting} holds with
	$\cH = \cH_{k_\cX}$ by \Cref{thm:kRR_well_specified_homoskedastic}, and
	$\cH_{k_\cY} \hookrightarrow L^2(\nu_\cY)$ continuously, so
	\Cref{cor:regularity_conditional_expectation_operator_CME} proves
	\ref{item:kRR_CME_representation_HS}.
	In case \ref{item:kRR_homosked_bounded}, the boundedness
	$\sup_{x \in \cX} k_\cX(x,x) < \infty$ established above allows the
	application of \Cref{prop:cme_learning_rate}, which yields
	\ref{item:kRR_CME_representation_rate}.
\end{proof}

\subsection{\texorpdfstring{Proofs of \Cref{sec:Bayesian_IP}}{Proofs of Section 4.2}} \label{sec:proofs_bip}

\begin{proof}[Proof of \Cref{thm:unbounded_F_weighted_sobolev}]
	Throughout the proof, positive constants of the form
	$C_{\mathrm{index}}$, $C_{\mathrm{index}}'$, $C_{\mathrm{index}}''$
	are chosen as needed from the context without further comment and may depend on the parameters indicated in the subscript.	
	
	\medskip

    \noindent
    \textbf{Step 1: Gaussian envelope for $u_{y}\coloneqq m_\tau(\quark)\eta(\quark,y)$.}
	Since $F(Y)$ is finite $P$-almost surely, there exists $R>0$ such that
	\[
	\pi_R \coloneqq P\big(\norm{F(Y)}_\Gamma\le R\big)>0.
	\]
	Fix $0<\delta_{1}<\tau$. Restricting the integral defining $\Xi(x)$ to the event
	$\{\norm{F(Y)}_\Gamma\le R\}$ yields
	\[
	\Xi(x)
	\ge
	\pi_R
	\exp\big(
	-R\norm{x}_\Gamma-\tfrac12 R^2
	\big)
	\geq
	\pi_R
	\exp \big( -\tfrac{1}{2} \big( \delta_{1}\norm{x}_\Gamma^2 + (1+\delta_{1}^{-1}) R^2\big)\big),
	\]
	where we used \eqref{equ:technical_inequality_splitting_2} with $a=\norm{x}_\Gamma$,
	$b=R$ and $\gamma=\delta_{1}$.
	Consequently,
	\begin{equation}
	\label{eq:uy_pointwise_bound_general}
	\begin{split}
	u_y(x)
	\coloneqq
	m_\tau(x)\eta(x,y)
	&\le
	C_{\delta_{1}}\,
	\exp\!\Big(
	-\frac{\tau-\delta_{1}}{2}\norm{x}_\Gamma^2
	+
	\innerprod{x}{F(y)}_{\Gamma}
	-
	\frac12 \norm{F(y)}_\Gamma^2
	\Big)
	\\
	&=
	C_{\delta_{1}}\,
	\exp\big(
	a_{\tau,\delta_{1}}
	\norm{F(y)}_\Gamma^2
	\big)
	\exp \big(
	-\tfrac{\tau-\delta_{1}}{2}
	\norm{
		x-b_{\tau,\delta_{1}}(y)
	}_\Gamma^2
	\big),
	\end{split}
	\end{equation}
	with $a_{\tau,\delta_{1}}
	\coloneqq
	\frac{1+\delta_{1}-\tau}{2(\tau-\delta_{1})}$ and $b_{\tau,\delta_{1}}(y)\coloneqq \frac{F(y)}{\tau-\delta_{1}}$,
	where in the last step we completed the square.

	\medskip

    \noindent
    \textbf{Step 2: Derivative bounds for $u_y$.}	
	Write
	\[
	u_y(x)=\exp(h_y(x)),\qquad
	h_y(x)=
	-\tfrac{\tau}{2}\|x\|_\Gamma^2
	+
	\innerprod{x}{F(y)}_\Gamma
	-
	\tfrac12\|F(y)\|_\Gamma^2
	-
	\Lambda(x).
	\]
    Using \eqref{eq:log_partition_growth_assumption} we obtain, for every
	multi-index $\alpha$ with $1 \leq |\alpha|\le \ell$,
	\begin{equation}
		\label{eq:hy_derivative_bound}
		|D^\alpha h_y(x)|
		\le
		C_\Lambda'
		\big(1+\|x\|_2+\|F(y)\|_2\big)^r.
	\end{equation}
    By the multivariate Fa\`a di Bruno formula, for every multi-index
	$\alpha$ with $|\alpha|\le\ell$,
	\[
	D^\alpha u_y(x)
	=
	u_y(x)\,
	B_\alpha\Big((D^\beta h_y(x))_{1\le|\beta|\le|\alpha|}\Big),
	\]
	where $B_\alpha$ is a (multivariate) polynomial whose monomials have the form
	\[
	\prod_{j=1}^s D^{\beta_j}h_y(x),
	\qquad
	s \leq |\alpha|,
	\quad
	|\beta_1|+\cdots+|\beta_s|=|\alpha|.
	\]
	Since each factor $D^{\beta_j}h_y(x)$ is bounded by 
	$C_\Lambda'(1+\|x\|_2+\|F(y)\|_2)^r$ by \eqref{eq:hy_derivative_bound}, 
	and each monomial contains at most $|\alpha|$ such factors, we obtain
	\begin{equation}
		\label{eq:uy_derivative_bound_general}
		\begin{split}
			|D^\alpha u_y(x)|
			&\leq
			C_\Lambda''
			\big(1+\|x\|_2+\|F(y)\|_2\big)^{r|\alpha|}
			\,u_y(x)
			\\
			&\leq
			C_{\Lambda,\delta_{1}}\,
			\big(1+\|x\|_2+\|F(y)\|_2\big)^{r|\alpha|}\,
			\exp\big(
			a_{\tau,\delta_{1}}
			\norm{F(y)}_\Gamma^2
			\big)
			\exp \big(
			-\tfrac{\tau-\delta_{1}}{2}
			\norm{
				x-b_{\tau,\delta_{1}}(y)
			}_\Gamma^2
			\big).
		\end{split}
	\end{equation}

	\medskip

    \noindent
    \textbf{Step 3: $H^{\ell}$-bound.}
	Using $\|z+b_{\tau,\delta_{1}}(y)\|_2\le \|z\|_2+\|b_{\tau,\delta_{1}}(y)\|_2$ and
	$\|b_{\tau,\delta_{1}}(y)\|_2
	\lesssim
	\|F(y)\|_2$,
	we obtain, for every $z \in \bR^{d}$ and $\delta_{2} > 0$,
	\begin{align*}
	\big(1+\|z+b_{\tau,\delta_{1}}(y)\|_2+\|F(y)\|_2\big)^{2 r |\alpha|}
	&\le
	C_{\Lambda,\tau,\delta_{1}}
	(1+\|z\|_2)^{2 r |\alpha|}
	(1+\|F(y)\|_2)^{2 r |\alpha|}
	\\
	&\leq
	C_{\Lambda,\tau,\delta_{1},\delta_{2}}
	(1+\|z\|_2)^{2 r |\alpha|}
	\exp \big(2\delta_{2} \norm{F(y)}_\Gamma^2\big),
	\end{align*}
	since polynomial factors in $\norm{F(y)}_\Gamma$ can be absorbed into
	$\exp(\delta_{2}\norm{F(y)}_\Gamma^2)$ up to a multiplicative constant (using equivalence of norms on $\bR^d$).
	Squaring and integrating \eqref{eq:uy_derivative_bound_general}, and changing variables to $z=x-b_{\tau,\delta_{1}}(y)$, yields
	\begin{align*}
		\|D^\alpha u_y\|_{L^2(\bR^d)}^2
		&\le
		C_{\Lambda,\tau,\delta_{1}}'
		\exp\!\big(2 a_{\tau,\delta_{1}}\norm{F(y)}_\Gamma^2\big)\, \times
		\\
		&\qquad\qquad\qquad\times\, 
		\int_{\bR^d}
		\big(1+\|z+b_{\tau,\delta_{1}}(y)\|_2+\|F(y)\|_2\big)^{2 r |\alpha|}
		\exp\!\big(
		-(\tau-\delta_{1})\norm{z}_\Gamma^2
		\big)\,\rd z
		\\
		&\leq
		C_{\Lambda,\tau,\delta_{1},\delta_{2}}'
		\exp\big(2 (a_{\tau,\delta_{1}} + \delta_{2})\norm{F(y)}_\Gamma^2\big)
		\int_{\bR^d}
		(1+\|z\|_2)^{2 r |\alpha|}
		\exp\big(
		-(\tau-\delta_{1})\norm{z}_\Gamma^2
		\big)\,\rd z
		\\
		&\leq
		C_{\Lambda,\tau,\delta_{1},\delta_{2}}''
		\exp\big(2 (a_{\tau,\delta_{1}} + \delta_{2})\norm{F(y)}_\Gamma^2\big).
	\end{align*}
	Taking the square root and summing over all $|\alpha|\le \ell$ gives
	\[
	\|u_y\|_{H^{\ell}(\bR^d)}
	\le
	C_{\ell,\tau,\delta_{1},\delta_{2}}'
	\exp\big((a_{\tau,\delta_{1}} + \delta_{2})\norm{F(y)}_\Gamma^2\big).	
    \]
    Since $u_y=m_\tau(\quark)\eta(\quark,y)$ and
	$\cH_{k_0} \simeq H^{\ell}(\bR^d)$, \Cref{lemma:reweighted_kernel}
	implies
	\[
	\eta(\quark,y)\in \cH_{k_{\cX}}
	\qquad\text{and}\qquad
	\norm{\eta(\quark,y)}_{\cH_{k_{\cX}}}
	=
	\norm{u_y}_{\cH_{k_0}}
	\leq
	C_{\ell,\tau,\delta_{1},\delta_{2}}
	\exp\big((a_{\tau,\delta_{1}} + \delta_{2})\norm{F(y)}_\Gamma^2\big).
	\]

    \medskip

    \noindent
    \textbf{Step 4: $\eta \in L^2(P_Y ; \cH_{k_{\cX}})$ under 
	\eqref{equ:conditions_on_tau}.}
	To prove $\eta \in L^2(P_Y ; \cH_{k_{\cX}})$, we verify the finiteness of 
	$\bE \big[ \norm{\eta(\quark ,Y)}_{\cH_{k_{\cX}}}^{2} \big]$
	using the pointwise bound established in Step~3.
	
    If $\tau > 1$, then we can choose 
	$\delta_{1},\delta_{2} > 0$ sufficiently small such that 
	$a_{\tau,\delta_{1}} + \delta_{2} < 0$, and thereby 
	$\bE \big[ \norm{\eta(\quark ,Y)}_{\cH_{k_{\cX}}}^{2} \big] 
	\leq C_{\ell,\tau,\delta_{1},\delta_{2}}^{2} < \infty$.

	If $\tau > \frac{1}{1+\lambda}$ for some $\lambda > 0$ with 
	$\bE\big[\exp \big( \lambda \norm{F(Y)}_\Gamma^2 \big) \big] < \infty$,
    then $\lambda > \frac{1-\tau}{\tau}$ and we can choose 
	$\delta_{1},\delta_{2} > 0$ sufficiently small such that 
	$2(a_{\tau,\delta_{1}} + \delta_{2}) \leq \lambda$, yielding
	\[
	\bE \big[ \norm{\eta(\quark ,Y)}_{\cH_{k_{\cX}}}^{2} \big]
	\leq
	C_{\ell,\tau,\delta_{1},\delta_{2}}^{2} \,
	\bE \big[ \exp\big(
	\lambda \norm{F(Y)}_{\Gamma}^{2}
	\big) \big]
	< \infty.
	\]
	In both cases, \Cref{thm:CEO_bounds_various_spaces} with $\nu_{\cY}=P_Y$ implies 
	that
    $K \in \HS(L^2(P_{Y}),\cH_{k_{\cX}})$ 
	with Hilbert--Schmidt norm
	$\norm{K}_{\HS(L^2(P_{Y}),\cH_{k_{\cX}})} = \norm{\eta}_{L^2(P_Y ; \cH_{k_{\cX}})}$.    
\end{proof}

\begin{lemma}[Cumulants bounded by centered moments]
	\label{lem:cumulant_moment_bound}
	Let \(Z\colon \Omega \to \mathbb R^d\) be a random vector such that
	\(\bE[\|Z\|_2^j]<\infty\) for some integer \(j\ge2\).
	For every multi-index \(\alpha\in \bN_0^d\) with \(|\alpha|=j\), let
	\[
	\kappa_\alpha(Z) \coloneqq D^\alpha K_Z(0),
	\qquad
	K_Z(t)\coloneqq \log \bE \big[e^{\innerprod{t}{Z}}\big],
	\]
	denote the corresponding mixed cumulant of order \(j\).
	Then there exists a constant \(c_\alpha>0\), depending only on \(\alpha\), such that
	$|\kappa_\alpha(Z)|
	\le
	c_\alpha\,
	\bE \big[\|Z-\bE Z\|_2^j\big]$.
\end{lemma}

\begin{proof}
	Cumulants can be written as polynomials in the moments of \(Z\);
	see \cite[Section~2.3.4]{McCullagh1987}.  
	Applying this identity to the centered vector
	\(\widetilde Z = Z-\bE Z\), and using that cumulants of order \(j\ge2\)
	are invariant under centering (see the affine transformation formula
	in \cite[Section~2.4]{McCullagh1987}), we obtain
	\[
	\kappa_\alpha(Z)
	=
	\sum_{\pi}
	c_{\alpha,\pi}
	\prod_{B\in\pi}
	\bE[\widetilde Z^{\beta_B}],
	\]
	where the sum contains finitely many terms and
	\(\sum_{B\in\pi}|\beta_B|=j\).
	By Jensen's inequality,
	\[
	\left|
	\prod_{B\in\pi}\bE \big[\widetilde Z^{\beta_B} \big]
	\right|	
	\le
	\prod_{B\in\pi}\bE \Big[\|\widetilde Z\|_2^{|\beta_B|}\Big]
	\le
	\prod_{B\in\pi}\bE \Big[\|\widetilde Z\|_2^{j}\Big]^{|\beta_B|/j}
	=
	\bE\Big[\|\widetilde Z\|_2^{j}\Big]^{\big(j^{-1} \sum_{B\in\pi}|\beta_B|\big)}
	=
	\bE\Big[\|\widetilde Z\|_2^{j}\Big].
	\]
	Summing over the finitely many terms yields
	\[
	|\kappa_\alpha(Z)|
	\le
	c_\alpha
	\bE[\|Z-\bE Z\|_2^j]
	\]
	for some constant \(c_\alpha>0\) depending only on \(\alpha\).
\end{proof}

\begin{proof}[Proof of \Cref{cor:log_partition_sufficient_conditions}]
	We verify that each of the conditions \ref{item:logpart_bounded},\ref{item:logpart_affine_gaussian},\ref{item:logpart_tilted_moments}, \ref{item:logpart_logconcave}
	implies \eqref{eq:log_partition_growth_assumption}.

    \medskip

    \noindent
    \textbf{\ref{item:logpart_affine_gaussian} $\Rightarrow$ \eqref{eq:log_partition_growth_assumption}.}
	If \(Y\) is Gaussian and \(F(y)=Ay+b\), then
	$X=F(Y)+\varepsilon$
	is Gaussian on \(\bR^d\) with some mean $m_{X}$ and covariance matrix $\Sigma_{X}$.
	Hence its density satisfies
	\[
	p_X(x)
	=
	c_1\, \exp\big(-\tfrac12\|x\|_\Gamma^2\big) \, \Xi(x)
	=
	c_2\, \exp\big(-\tfrac12\norm{x-m_X}_{\Sigma_X}^{2}\big),
	\]
	for certain constants $c_{1},c_{2} > 0$, where we used the definition of $\Xi$. Taking logarithms gives
	\[
	\Lambda(x)=\tfrac12\|x\|_\Gamma^2 - \tfrac12\norm{x-m_X}_{\Sigma_X}^{2} + \log \tfrac{c_2}{c_{1}}.
	\]
	Thus \(\Lambda\) is a quadratic polynomial, and therefore
	satisfies \eqref{eq:log_partition_growth_assumption}.
    
	\medskip

    \noindent
    \textbf{\ref{item:logpart_tilted_moments} $\Rightarrow$ \eqref{eq:log_partition_growth_assumption}.}
	Let \(\Lambda(x)=\log\Xi(x)\). Since \(\Xi\in C^\ell(\bR^d)\), the derivatives
	of \(\Lambda\) up to order \(\ell\) are well-defined.
	By differentiating the log-partition function,
	\[
	\norm{\nabla \Lambda(x)}_{2}
	=
	\norm{\bE[a(Y) \, | \,  X=x]}_{2}
	=
	\norm{m_1^a(x)}_{2}
	\le
	c\, (1+\|x\|_2).
	\]	
	Now let \(\alpha\in\bN_0^d\) with \( 2 \leq |\alpha|=j\leq \ell\). By definition of cumulants,
	\(D^\alpha\Lambda(x)\) is the mixed cumulant of \(a(Y)\) under
	the conditional law \(P_{Y \, | \,  X=x}\). Hence, by
	\Cref{lem:cumulant_moment_bound} applied conditionally,
	\[
	|D^\alpha\Lambda(x)|
	\le
	c_\alpha\,
	\bE\left[
	\|a(Y)-m_1^a(x)\|_2^j
	\, | \, X=x
	\right]
	=
	c_\alpha\, m_j^a(x)
	\le
	c_\alpha \, c \, (1+\|x\|_2).
	\]
	This shows \eqref{eq:log_partition_growth_assumption}.

    \medskip

    \noindent
    \textbf{\ref{item:logpart_bounded} $\Rightarrow$ \ref{item:logpart_tilted_moments}.}
    If \(F\) is \(P_Y\)-essentially bounded, then \(a(Y)=\Gamma^{-1}F(Y)\) is also
    \(P_Y\)-essentially bounded. Hence there exists \(R>0\) such that
    $ \|a(Y)\|_2\le R$ $P_Y$-a.s.
    Since \(P_{Y \, | \,  X=x}\ll P_Y\) for every \(x\in\bR^d\), the same bound holds
    \(P_{Y \, | \,  X=x}\)-almost surely. Therefore, for every \(x\in\bR^d\),
    \[
    \|m_1^a(x)\|_2
    =
    \|\bE[a(Y) \, | \,  X=x]\|_2
    \le
    R,
    \]
    and, for \(j=2,\dots,\ell\),
    \[
    m_j^a(x)
    =
    \bE\big[\|a(Y)-m_1^a(x)\|_2^j  \, | \,  X=x\big]
    \le
    (2R)^j.
    \]
    Thus, the bounds in \ref{item:logpart_tilted_moments} hold.		
	
	\medskip

    \noindent
    \textbf{\ref{item:logpart_logconcave} $\Rightarrow$ \ref{item:logpart_tilted_moments}.}
	Let \(\rho\) denote the Lebesgue density of the pushforward law of \(a(Y)=\Gamma^{-1}F(Y)\).
	By assumption, \(\rho\) is strongly log-concave, so we may write
	\[
	\rho(z)\propto e^{-V(z)}
	\qquad
	\text{with}
    \qquad
	\nabla^2 V(z)\succeq cI		
	\]
	for some \(c>0\).	
	For each \(x\in\bR^d\), let \(P_x^a\) denote the law of \(a(Y)\) under the conditional
	distribution \(P_{Y \, | \,  X=x}\).
	By \eqref{equ:Radon_Nikodym_BIP_exponential_tilting},
	the law \(P_x^a\) has density
	\[
	\rho_x(z)
	\propto
	\exp\!\Big(\innerprod{x}{z}-\tfrac12 z^\top \Gamma z\Big)\rho(z)
	=
	\exp\!\Big(
	-\Big[V(z)-\innerprod{x}{z}+\tfrac12 z^\top\Gamma z\Big]
	\Big).
	\]
	Hence \(P_x^a\) is strongly log-concave with a lower Hessian bound that is
	uniform in \(x\):
	\[
	\nabla^2\!\Big(V(z)-\innerprod{x}{z}+\tfrac12 z^\top\Gamma z\Big)
	\succeq
	\lambda_\ast\, I,
	\qquad
	\lambda_\ast \coloneqq c+\lambda_{\min}(\Gamma) >0.
	\]
	Let \(G,G'\sim\mathcal N(0,\lambda_\ast^{-1}I)\) be independent.
	By Caffarelli's contraction theorem \citep[Theorem~9.7]{SaumardWellner2014},
	for every \(x\in\bR^d\), there exists a contraction
	\(T_x:\bR^d\to\bR^d\) such that \(T_x(G)\sim P_x^a\).
	Then, for \(j=2,\dots,\ell\),
	\[
	m_j^a(x)
	=
	\bE\big[\| T_x(G) - \bE[T_x(G)]\|_2^j\big]
	\le
	\bE\big[\|T_x(G) - T_x(G')\|_2^j\big]
	\le
	\bE\big[\|G-G'\|_2^j\big]
	\eqqcolon C_j ,
	\]
	Thus the conditional centered moments are bounded uniformly in \(x\).
	Further, since \(V\) is strongly convex (\(\nabla^2V\succeq cI\)), the density \( \rho(z)\propto e^{-V(z)}\) has Gaussian tails.
	Hence the integrand in the definition of \(\Xi\) is dominated by an
	integrable function uniformly on compact sets in \(x\).
	Differentiation under the integral sign is therefore justified to arbitrary
	order, yielding \(\Xi\in C^\infty(\bR^d)\) and thus \(\Lambda\in C^\infty(\bR^d)\).
	Since 
	\[
	m_1^a(x)
	=
	\bE[a(Y) \, | \, X=x]
	=
	\nabla\Lambda(x),
	\quad
	\text{and}
	\quad
	\norm{D^2\Lambda(x)}
	=	
	\norm{\Cov[a(Y) \, | \, X=x]}
	\leq
	m_2^a(x)
	\leq
	C_{2}
	\]	
	uniformly in \(x\),
	\[
	\|m_1^a(x)\|_2
	\le
	\Norm{\nabla\Lambda(0) + \int_0^1 D^2\Lambda(tx)\,x\, \rd t}_2
	\le
	\|\nabla\Lambda(0)\|_2 + C_{2}\|x\|_2
	\le
	C_1(1+\|x\|_2)
	\]
	for some $C_{1} > 0$ and all $x \in \bR^{d}$.
	Hence the bounds in \ref{item:logpart_tilted_moments} hold with $C \coloneqq \max_{j=1,\dots,\ell} C_{j}$.
\end{proof}

\begin{proof}[Proof of \Cref{cor:unbounded_F_weighted_sobolev_CME}]
We first verify that $(X,k_\cX)$ satisfies \Cref{ass:kernel_rv_pair}.

\ref{item:basic_assumptions_kernel}: Since $k_0$ is continuous (\Cref{thm:unbounded_F_weighted_sobolev}) and $m_\tau(x) = \exp(-\tfrac{\tau}{2}\norm{x}_\Gamma^2) > 0$ is continuous, $k_\tau$ is continuous, hence measurable. Symmetry and positive definiteness of $k_\tau$ follow from \Cref{lemma:reweighted_kernel}, which also shows that the map $f \mapsto m_\tau f$ is an isometry from $\cH_{k_\tau}$ onto $\cH_{k_0} \simeq H^\ell(\bR^d)$; since $H^\ell(\bR^d)$ is separable, so is $\cH_{k_\tau}$.

\ref{item:basic_assumptions_feature_map_L2}: Since $\tau < \frac{\lambda}{1+2\lambda}$, $\zeta \coloneqq 2\lambda$ satisfies
\[
\tau (1+\zeta) < \lambda
\quad
\text{and}
\quad
s
\coloneqq
\tau (1+\zeta^{-1})
<
\tfrac{1}{2}.
\]
It follows from \eqref{equ:technical_inequality_splitting_1} that
\[
\tau \norm{X}_{\Gamma}^{2}
\leq
\tau \big( \norm{F(Y)}_{\Gamma}+\norm{\varepsilon}_{\Gamma} \big)^{2}
\leq
\tau (1+\zeta)\norm{F(Y)}_{\Gamma}^{2}
+
\tau (1 + \zeta^{-1})\norm{\varepsilon}_{\Gamma}^{2}
\leq
\lambda\norm{F(Y)}_{\Gamma}^{2}
+
s\norm{\varepsilon}_{\Gamma}^{2}.
\]
Since $Y$ and $\varepsilon$ are independent and $c_{0} = k_{0}(x,x)$ is constant by translation-invariance of $k_{0}$,
\begin{align*}
    \bE \big[ k_{\cX}(X,X) \big]	
    &=
    \bE \big[ m_\tau(X)^{-2}\, k_{0}(X,X) \big]
    =
    c_0\, \bE \big[ \exp\big( \tau \norm{X}_{\Gamma}^{2} \big) \big]
    \\
    &\leq
    c_0\,
    \bE \big[ \exp\big( \lambda \norm{F(Y)}_{\Gamma}^{2} \big) \big]\,
    \bE \big[ \exp\big( s \norm{\varepsilon}_{\Gamma}^{2} \big) \big].
\end{align*}
The first expectation is finite by assumption, while the second one is finite since $s < 1/2$ and $\norm{\varepsilon}_{\Gamma}^{2} \sim \chi^{2}_{d}$ with moment generating function $\bE \big[ \exp(s \norm{\varepsilon}_{\Gamma}^{2}) \big] = (1-2s)^{-d/2}$. Hence $\bE[k_\cX(X,X)] < \infty$.

\ref{item:basic_assumptions_ae_separation}: Since $X = F(Y) + \varepsilon$ with $\varepsilon \sim \mathcal{N}(0,\Gamma)$ and $\Gamma$ strictly positive definite, the conditional distribution $P_{X\mid Y=y} = \mathcal{N}(F(y),\Gamma)$ has a strictly positive density on $\bR^d$ for every $y$. Consequently, for every non-empty open set $U \subset \bR^d$, $P_X(U) = \int P_{X\mid Y=y}(U)\,P_Y(\rd y) > 0$, so $\operatorname{supp}(P_X) = \bR^d = \cX$. Since $k_\cX = k_\tau$ is continuous (by \ref{item:basic_assumptions_kernel}), this proves injectivity of the canonical embedding $\cH_{k_\cX} \hookrightarrow L^2(P_X)$.

Hence $(X,k_\cX)$ satisfies \Cref{ass:kernel_rv_pair}. Consequently, if \eqref{equ:balancing_tau_and_lambda} holds and $(Y,k_{\cY})$ also satisfies \Cref{ass:kernel_rv_pair} (implying $\cH_{k_{\cY}} \hookrightarrow L^2(P_{Y})$), then \Cref{thm:unbounded_F_weighted_sobolev,cor:regularity_conditional_expectation_operator_CME} with $\nu=P_Y$ prove the remaining claims.
\end{proof}

\subsection{\texorpdfstring{Proofs of \Cref{sec:koopman_sde_class}}{Proofs of Section 4.3}}\label{sec:proofs_sde}

Throughout the proofs of this section, $C$ and $c$ denote positive constants that depend only on the quantities indicated in the subscript.

\begin{proof}[Proof of \Cref{prop:sde_density_regularity_bounds}]
Under \Cref{ass:sde_unif_elliptic_regular}, $\|\nabla b\|_\infty < \infty$ and $\|\nabla \sigma\|_\infty < \infty$, so by the mean value theorem the coefficients $b$ and $\sigma$ are globally Lipschitz continuous. By \citep[Theorem 5.2.1]{Oksendal2003SDE}, the SDE \eqref{eq:basic_sde_time_homogeneous} admits a unique strong solution.
Let $(K_t)_{t \ge 0}$ denote the associated Markov (Koopman) semigroup $K_t f(x) \coloneqq \mathbb{E}[f(X_t)\, | \, X_0=x]$. Its infinitesimal generator $A$, defined on $C_0^2(\mathbb{R}^d)$, is given by \citep[Theorem 7.3.3]{Oksendal2003SDE}
\begin{align*}
(Ag)(x) = \lim_{t \downarrow 0} \frac{K_t g(x) - g(x)}{t} = \sum_{i=1}^d b_i(x) \partial_{x_i} g(x) + \frac{1}{2} \sum_{i,j=1}^d (\sigma(x) \sigma(x)^\top)_{ij} \partial_{x_i} \partial_{x_j} g(x), \; x \in \bR^d.
\end{align*}
Under \Cref{ass:sde_unif_elliptic_regular}, $A$ is a uniformly elliptic second-order differential operator with bounded coefficients and bounded derivatives up to order $\ell$. 

Classical parabolic PDE theory \citep[e.g.][]{Frie1964} now implies that $A$ admits a fundamental solution $p(t,x,y)$ which satisfies $K_t f(x) = \int_{\bR^d} f(y)\, p(t,x,y)\, \rd y$ and coincides with the transition density of the diffusion process, i.e., $ \eta_t(x,y) = p(t,x,y)$. 
\citet[Chapter 9, Theorem 7]{Frie1964} then states that for every $t>0$, $p(t,\quark,y)$ admits weak derivatives $\partial_x^\alpha p(t,x,y)$ up to order $|\alpha| \le \ell$, satisfying the Gaussian bounds
\[
|\partial_x^\alpha p(t,x,y)|
\le C_\alpha\, t^{-(|\alpha|+d)/2}\,
\exp \left(-\kappa \, \frac{\|x-y\|_2^2}{t}\right),
\]
with constants $\kappa , C_\alpha > 0$. 
\end{proof}

\begin{proof}[Proof of \Cref{thm:sde_density_koopman_unified}]
    We first verify that the representation \eqref{equ:Koopman_definition_and_integral_form} holds with density $\eta_t$, before treating the two cases separately.

    \medskip

    \noindent
    \textbf{Step 1: Koopman representation and derivative estimates.}
    By \Cref{prop:sde_density_regularity_bounds}, the transition density $\eta_t(x,y)$ exists and \eqref{equ:Koopman_definition_and_integral_form} holds, i.e., $(K_tg)(x) = \bE[g(X_t) \, | \,  X_0=x] = \int_{\bR^d} g(y)\eta_t(x,y)\,\rd y$. Moreover, \Cref{prop:sde_density_regularity_bounds}
    shows that, for every multi-index $|\alpha|\le \ell$, there exists $\kappa>0$ such that $|\partial_x^\alpha \eta_t(x,y)| \le C_{\alpha,t} \exp\left( -\kappa(\|x-y\|_2^2)/t \right)$. Since these derivatives are integrable in $y$, differentiation under the integral sign is justified, yielding $\partial_x^\alpha K_tg(x) = \int_{\bR^d} g(y)\, \partial_x^\alpha\eta_t(x,y)\,\rd y$ for every $|\alpha|\le\ell$. Thus, it remains to verify the required regularity of
    $\eta_t(\quark,y)$.

    \medskip

    \noindent
    \textbf{Step 2: Case \ref{item:thm_sde_koopman_bounded}.}
    Assume that $\cX\subset\bR^d$ is Lipschitz bounded. Using the Sobolev norm representation, Fubini--Tonelli, and the derivative estimates above, we obtain
    \begin{align*}
        \|\eta_t\|_{L^2(\nu_\cY;H^\ell(\cX))}^2 
        = \sum_{|\alpha|\le\ell} \int_{\cX} \int_{\bR^d} |\partial_x^\alpha\eta_t(x,y)|^2 \,\rd y\,\rd x \leq \sum_{|\alpha|\le\ell} C_{\alpha,t} \int_{\cX} \int_{\bR^d} \exp\left( -\frac{2\kappa}{t}\|x-y\|_2^2 \right) \rd y\,\rd x.
    \end{align*}
    For the inner integral we easily see that $ \int_{\bR^d} \exp\left(-\frac{2\kappa}{t}\|x-y\|_2^2 \right) \rd y = C_{\kappa,t}$, independently of $x$. By boundedness of $\cX$, we obtain
    $\|\eta_t\|_{L^2(\nu_\cY;H^\ell(\cX))}^2 \leq C_{\ell,t} \bslambda(\cX) < \infty$, 
    which proves that $\eta_t \in L^2(\nu_\cY; H^\ell(\cX))$. By \Cref{thm:CEO_bounds_various_spaces}\ref{item:CEO_H_HS}, $K_t \in \HS(L^2(\nu_{\cY}), H^\ell(\cX))$ with $\|K_t\|_{\HS(L^2(\nu_{\cY}), H^\ell(\cX))} = \|\eta_t\|_{L^2(\nu_\cY;H^\ell(\cX))}$ follows.

Further, setting $\nu_\cX = \bslambda$, we immediately obtain $K_{t} \in \HS( L^2(\nu_{\cY}) , L^2(\nu_{\cX}))$ by \Cref{thm:CEO_bounds_various_spaces}\ref{item:CEO_L2_HS}, and $K_t \in \Tr( L^2(\nu_{\cY}) , L^2(\nu_\cX))$ for $\ell > d/2$ by \Cref{thm:CEO_bounds_various_spaces}\ref{item:CEO_L2_Tr}. 

\medskip

\noindent
\textbf{Step 3: Case \ref{item:thm_sde_koopman_unbounded}.}
Assume now that $\cX=\bR^d$. By
\Cref{lemma:reweighted_kernel} and the norm equivalence between $\cH_{k_0}$ and $H^\ell(\bR^d)$, we have $\|f\|_{\cH_{k_\tau}} \simeq \|m_\tau f\|_{H^\ell(\bR^d)}$, where $m_\tau(x) = \exp(-\frac{\tau}{2}\|x\|_\Gamma^2)$. It therefore suffices to estimate $\|m_\tau  \eta_t(\quark,y)\|_{H^\ell(\bR^d)}$.

By the product rule, every derivative $\partial_{x}^{\alpha} (m_\tau\eta_t(\quark,y))$ is a finite sum of terms $(\partial^\beta m_\tau) (\partial_x^{\alpha-\beta}\eta_t)$ with $|\beta| \leq |\alpha| \leq \ell$. 
Since every derivative of $m_\tau$ is a polynomial multiplied by the same Gaussian weight, there exist constants $C_{\alpha,\beta}$ such that $|\partial^\beta m_\tau(x)| \le C_{\alpha,\beta} (1+\|x\|_2)^{\ell}\exp\left(-\frac{\tau}{2}\|x\|_\Gamma^2\right)$.

Further, since $\Gamma$ is symmetric and positive definite, there exist constants $0 < \lambda_\mathrm{min} \leq \lambda_\mathrm{max}$ such that $\lambda_\mathrm{min} \|x\|_2^2 \leq \|x\|_\Gamma^2 \leq \lambda_\mathrm{max} \|x\|_2^2$, hence $e^{-\frac{\tau}{2} \|x\|_\Gamma^2} \leq e^{-\frac{\tau}{2} \lambda_\mathrm{min} \|x\|_2^2}$. Combining this with the Gaussian estimate for $\partial^{\alpha-\beta}\eta_t$ yields
\begin{align*}
    |\partial^\alpha(m_\tau\eta_t)(x,y)|^2 \le C_{\alpha,t,\tau}\, (1+\|x\|_2)^N \exp\left(- \tau \lambda_{\mathrm{min}} \|x\|_2^2\right) \, \exp\left(-\frac{2\kappa}{t}\|x-y\|_2^2 \right)
\end{align*}
for some integer $N$. Completing the square in the exponent gives
\begin{align*}
    -\tau\|x\|_\Gamma^2 - \frac{2\kappa}{t}\|x-y\|_2^2 = -(A+B) \left\| x-\frac{B}{A+B}y \right\|^2 - c_{\tau,\kappa,t} \|y\|_2^2,
\end{align*}
with positive constants $A=\tau \lambda_\mathrm{min}$, $B=2\kappa/t$, and $c_{\tau,\kappa,t}>0$.
Consequently, after integrating over $x$ for fixed $y \in \bR^d$, we obtain $\|m_\tau \eta_t(\quark,y) \|_{H^\ell(\bR^d)}^2 \le C_{\ell,t,\tau}\, \exp\left(-c_{\tau,\kappa,t}\|y\|_2^2 \right)$, where we used the finiteness of Gaussian moments for any $N \in \bN$. The expression on the right-hand side is clearly integrable over $\cY = \bR^d$, hence $\eta_t \in L^2(\nu_\cY;\cH_{k_\tau})$ follows. This means that \Cref{assump:general_Radon_Nikodym_setting} holds, such that \Cref{thm:CEO_bounds_various_spaces}\ref{item:CEO_H_HS} implies $K_t \in \HS(L^2(\nu_\cY), \cH_{k_\tau})$.
\end{proof}

\begin{proof}[Proof of \Cref{cor:sde_koopman_one_sided_error_rate}]
    Under \Cref{ass:sde_unif_elliptic_regular}, \Cref{thm:sde_density_koopman_unified}\ref{item:thm_sde_koopman_bounded} yields $\eta_t \in L^2(\nu_\cY; H^\ell(\cX))$ with $\nu_\cY = \bslambda$. Thus \Cref{assump:general_Radon_Nikodym_setting} holds with $\cH = H^\ell(\cX)$.     
    By \Cref{prop:matern_wendland_sobolev}, the Mat\'ern kernel with parameter $\beta = \ell - d/2 > 0$ induces the RKHS $\cH_{k_\beta^\mathrm{Mat}}(\cX) \simeq H^\ell(\cX)$. \Cref{thm:err_bounds_proj_unified}\ref{item:thm_err_bounds_proj_unified_RBF} then implies \eqref{eq:sde_one_sided_proj_fd_bound} with $m = \lceil \beta \rceil - 1 = \lceil \ell - d/2 \rceil - 1$. 
    
    Further, by \Cref{prop:matern_wendland_sobolev} the Wendland kernel $k_{d,j}^\mathrm{Wen}$ with odd $d$ and parameter $j = \ell - (d+1)/2 \in \bN_0$ induces the RKHS $\cH_{k_{d,j}^{\mathrm{Wen}}} \simeq H^\ell(\cX)$. Analogously, \eqref{eq:sde_one_sided_proj_fd_bound} then follows from \Cref{thm:err_bounds_proj_unified}\ref{item:thm_err_bounds_proj_unified_wendland} with $m = j + 1/2 = \ell - d/2$. 
\end{proof}

\end{document}

%% file: Packages_arxiv.tex
\usepackage{amsfonts,amssymb,amsmath,amsthm,amstext,amssymb,mathtools}
\usepackage{subcaption,float,color,graphicx,enumitem}
\usepackage{dsfont}  
\usepackage[normalem]{ulem} 
\usepackage[authoryear,sort,round]{natbib}  
\usepackage{fullpage} 
\usepackage{hyperref,nicefrac}
\usepackage{algorithm}
\usepackage{algpseudocode}
\usepackage{tabularx,booktabs,multirow, makecell, array}

\usepackage[noabbrev,capitalise,nosort,nameinlink]{cleveref}

\usepackage{aliascnt}

\usepackage[table]{xcolor}

\usepackage{empheq} 

\usepackage{adjustbox}
\usepackage{tcolorbox}
\usepackage{tikz}
\usetikzlibrary{bayesnet,er,trees,shapes.symbols,mindmap,arrows,arrows.meta,decorations,shapes.misc,shapes.arrows,chains,matrix,positioning,scopes,decorations.pathmorphing,patterns, calc, arrows.meta, decorations.pathreplacing}
\usepackage{tikz-cd} 


%% file: Macros_arxiv.tex
\newcommand{\HS}{\mathsf{HS}}
\newcommand{\Tr}{\mathsf{Tr}}

\newcommand{\fd}{\mathsf{h}}

\newcommand*{\bR}{\mathbb{R}}

\newcommand*{\bN}{\mathbb{N}}

\newcommand*{\bE}{\mathbb{E}}

\newcommand*{\bP}{\mathbb{P}}

\newcommand*{\cN}{\mathcal{N}}
\newcommand*{\cE}{\mathcal{E}}
\newcommand*{\cF}{\mathcal{F}}
\newcommand*{\cH}{\mathcal{H}}
\newcommand*{\cD}{\mathcal{D}}
\newcommand*{\cG}{\mathcal{G}}
\newcommand*{\cP}{\mathcal{P}}
\newcommand*{\cX}{\mathcal{X}}
\newcommand*{\cY}{\mathcal{Y}}

\newcommand*{\Cov}{\mathsf{C}}

\newcommand*{\cL}{\mathcal{L}}

\newcommand*{\cS}{\mathcal{S}}
\newcommand*{\cT}{\mathcal{T}}

\makeatletter
\newcommand*\quark{\mathpalette\quark@{.5}}
\newcommand*\quark@[2]{\mathbin{\vcenter{\hbox{\scalebox{#2}{$\; \m@th#1\bullet \;$}}}}}
\makeatother

\makeatletter
\newcommand{\mylabel}[2]{#2\def\@currentlabel{#2}\label{#1}}
\makeatother

\newcommand*{\rd}{\mathrm{d}}

\DeclarePairedDelimiterX{\infdivx}[2]{(}{)}{#1\;\delimsize\|\;#2}

\DeclareMathOperator{\Id}{Id}

\definecolor{darkgreen}{rgb}{0,0.4,0}

\newcommand*{\todored}[1]{\bgroup\color{red}TODO:~#1\egroup}
\newcommand*{\mh}[1]{\bgroup\color{orange}MH:~#1\egroup}
\newcommand*{\ik}[1]{\bgroup\color{cyan}IK:~#1\egroup}
\newcommand*{\ms}[1]{\bgroup\color{violet}MS:~#1\egroup}
\newcommand*{\kw}[1]{\bgroup\color{olive}KW:~#1\egroup}

\DeclarePairedDelimiter\abs{\lvert}{\rvert}%
\DeclarePairedDelimiter\mynorm{\lVert}{\rVert}%
\makeatletter
\let\oldabs\abs
\def\abs{\@ifstar{\oldabs}{\oldabs*}}
\let\oldnorm\mynorm
\def\mynorm{\@ifstar{\oldnorm}{\oldnorm*}}
\makeatother

\theoremstyle{plain}
\newtheorem{theorem}{\sffamily Theorem}[section]
\newaliascnt{proposition}{theorem}
\newtheorem{proposition}[proposition]{\sffamily Proposition}
\aliascntresetthe{proposition}
\newaliascnt{lemma}{theorem}
\newtheorem{lemma}[lemma]{\sffamily Lemma}
\aliascntresetthe{lemma}
\newaliascnt{corollary}{theorem}
\newtheorem{corollary}[corollary]{\sffamily Corollary}
\aliascntresetthe{corollary}
\theoremstyle{definition}

\newaliascnt{notation}{theorem}
\newtheorem{notation}[notation]{\sffamily Notation}
\aliascntresetthe{notation}
\newaliascnt{example}{theorem}
\newtheorem{example}[example]{\sffamily Example}
\aliascntresetthe{example}

\newaliascnt{remark}{theorem}
\newtheorem{remark}[remark]{\sffamily Remark}
\aliascntresetthe{remark}

\newaliascnt{assumption}{theorem}
\newtheorem{assumption}[assumption]{\sffamily Assumption}
\aliascntresetthe{assumption}

\newcommand{\absval}[1]{\lvert #1 \rvert}
\newcommand{\innerprod}[2]{\langle #1 , #2 \rangle}
\newcommand{\norm}[1]{\lVert #1 \rVert}

\newcommand{\Innerprod}[2]{\left\langle #1 , #2 \right\rangle}
\newcommand{\Norm}[1]{\left\Vert #1 \right\Vert}

\numberwithin{equation}{section}
\numberwithin{figure}{section}
\numberwithin{table}{section}

\usepackage{cleveref}
\Crefname{theorem}{Theorem}{Theorems}
\Crefname{proposition}{Proposition}{Propositions}
\crefname{corollary}{Corollary}{Corollaries}
\Crefname{definition}{Definition}{Definitions}
\crefname{example}{Example}{Examples}
\Crefname{condition}{Condition}{Conditions}
\Crefname{remark}{Remark}{Remarks}

\crefname{assumption}{Assumption}{Assumptions}
\Crefname{assumption}{Assumption}{Assumptions}
\crefalias{assumption}{assumption}

\usepackage{todonotes}

\newcommand{\bsK}{{\boldsymbol{K}}}

\newcommand{\bsX}{{\boldsymbol{X}}}
\newcommand{\bsY}{{\boldsymbol{Y}}}

\newcommand{\bsbeta}{{\boldsymbol{\beta}}}

\newcommand{\bslambda}{{\boldsymbol{\lambda}}}

